\documentclass[12pt]{article}

\usepackage[utf8]{inputenc}

\usepackage{graphicx}

\usepackage{appendix}

\newcommand*{\Cdot}{\raisebox{-0.25ex}{\scalebox{1.2}{$\cdot$}}}

\usepackage[margin=2cm]{geometry}
\usepackage{mathrsfs}

\usepackage{mathtools}
\usepackage{amsfonts}
\usepackage{amssymb}
\usepackage{bm}
\usepackage{amsthm}
\usepackage{bbm}
\usepackage{enumitem}
\usepackage{accents}
     
\usepackage[hyphens]{url}
\usepackage{authblk}

\newtheorem{theo}{Theorem}
\newtheorem{lemma}[theo]{Lemma}
\newtheorem{cor}[theo]{Corollary}
\newtheorem{prop}[theo]{Proposition}

\newtheorem{conjecture}[theo]{Conjecture}
\newtheorem*{claim*}{Claim}
\newtheorem{question}[theo]{Question}
\theoremstyle{definition}\newtheorem*{defn*}{Definition}
\theoremstyle{definition}
\theoremstyle{remark}\newtheorem{remark}[theo]{Remark}

\numberwithin{theo}{section}
\numberwithin{equation}{section}

\newcommand{\R}{\mathbb{R}}
\newcommand{\N}{\mathbb{N}}
\newcommand{\C}{\mathbb{C}}
\newcommand{\Z}{\mathbb{Z}}

\newcommand{\BP}{\mathbb{P}}

\newcommand{\MCC}{\mathcal{C}}

\newcommand{\MCI}{\mathcal{I}}

\newcommand{\MCM}{\mathcal{M}}
\newcommand{\MCP}{\mathcal{P}}
\newcommand{\MCS}{\mathcal{S}}

\newcommand{\diff}{\mathop{}\!\mathrm{d}}
\newcommand{\RP}{\R \BP}

\newcommand{\GL}{\mathrm{GL}}
\newcommand{\SL}{\mathrm{SL}}

\newcommand{\innerproduct}[2]{\langle #1, #2 \rangle}
\DeclareMathOperator{\interior}{int}
\DeclareMathOperator{\supp}{supp}

\title{Frostman dimension of Furstenberg measure for \(\SL(2,\R)\) random matrix products}
\author{Tom Rush}
\date{}

\begin{document}
\maketitle

\begin{abstract}
For compactly supported \(\mu \in \MCP(\SL(2,\R))\) satisfying strong irreducibility and proximality, we obtain a formula for the Frostman dimension of the associated Furstenberg measure. We also describe the left neighbourhood of 0 for which the classical transfer operators defined by Le Page have a spectral gap on H\"older spaces in this setting.
\end{abstract}

\section{Introduction}
Let \(\mu\) be a compactly supported Borel probability measure on \(\GL(d,\R)\) and let \(S_{\mu} \subseteq \GL(d,\R)\) be the closed semigroup generated by \(\supp \mu\). We say \(\mu\) is \textit{strongly irreducible} if there is no finite collection of proper subspaces of \(\R^d\) which are \(S_{\mu}\) invariant and \textit{proximal} if \(S_{\mu}\) contains an element whose leading eigenvalue is simple. We say \(\mu\) satisfies \textit{SIP} if it is strongly irreducible and proximal.

For \(x=\R u \in \R \mathbb{P}^{d-1}\) and \(g \in \GL(d,\R)\), we denote \(gx \in  \R \mathbb{P}^{d-1}\) to be the element \(\R g u\) and we will write \(|gx|:=\frac{|gu|}{|u|}\), where \(|\cdot|\) denotes the Euclidean norm. In the seminal paper \cite{Pag82}, Le Page defined the transfer operators
\[P_t \psi(x)=\int |gx|^t \psi(g x) \diff \mu(g),\]
and proved a spectral gap for \(t=0\) on the space of \(\zeta\)-H\"older continuous functions \(C^{\zeta}(\RP^{d-1})\) for all \(\zeta\) sufficiently small. By perturbation theory this implies \(P_t\) has a spectral gap for all \(t\) in a neighbourhood of 0. This was used to prove various limit theorems (see \cite{Pag82} or \cite[\S 1.5]{BQ16}) and has since found applications across probability theory, dynamical systems, and mathematical physics (see e.g. \cite{BL85,CKN86,CPV93,Via14,BQ16,AEV23}).

Under the above assumptions there is a unique measure \(\nu_F \in \MCP(\RP^{d-1})\), called the \textit{Furstenberg measure}, satisfying \(\mu * \nu_F =\nu_F\), where \(*\) denotes convolution. Equivalently, it is the leading eigenmeasure corresponding to \(P_0\). The following definition will be important in this paper: we define the \textit{Frostman dimension} of a probability measure \(\nu \in \MCP(\RP^{d-1})\) to be the quantity 
\[\dim_F \nu:= \sup \{s \geq 0: \exists C>0 \mathrm{\ s.t. \ }\nu(B(x,r)) \leq C r^s, \forall r>0, \forall x \in \RP^{d-1}\},\] 
where \(B(x,r)\) is the closed ball with respect to the natural metric \(d_{\RP^{d-1}}(\R u,\R v):= \frac{|u \wedge v|}{|u| |v|} \).\footnote{This has also been referred to as the \textit{H\"older regularity} of a measure. In this paper, the Frostman dimension of measures is intimately related to the H\"older exponent of eigenfunctions (of different exponent), so to avoid confusion we prefer the terminology of Frostman dimension from the fractal geometry community.} Using large deviations obtained by perturbation theory of the above operators, it was shown by Guivarc'h \cite{Gui90} that \(\dim_F \nu_F>0\) when \(\mu\) satisfies SIP and has finite exponential moment. Aoun and Sert attained an explicit (not sharp) lower bound for the Frostman dimension in a related Gromov hyperbolic setting \cite[Proposition 4.10]{AS22}, which includes the \(\SL(2,\R)\), SIP setting.

Le Page's work has been significantly generalised in work by Goldsheid, Guivarc'h, Le Page \cite{GG96,GP04,Gui08,GP16}, Benoist, Quint \cite{BQ16}, and Hautec{\oe}ur \cite{Hau25}. Notably, it is also known that \(P_t\) has a spectral gap on H\"older functions for \(t \in [0,\infty)\) when \(\mu\) is compactly supported \cite{GP04,GP16}. Despite the wide interest in and applicability of these results, several fundamental questions remain unanswered. For example, it is characteristic of Le Page's result and its generalisations that the following quantities have not been determined:
\begin{enumerate}[label=(\roman*)]
    \item the value \(\zeta\) for which \(P_t\) has a spectral gap on \(C^{\zeta}(\RP^{d-1})\);
    \item the left neighbourhood of 0 for which \(P_t\) has a spectral gap;
    \item the Frostman dimension of \(\nu_F\).
\end{enumerate}

In this paper we describe the quantities (i)-(iii) in the \(\SL(2,\R)\), SIP setting where \(\mu\) is compactly supported. In this case the Frostman dimension of the Furstenberg measure is precisely the supremum of the \(\zeta\) for which \(P_0\) has a spectral gap on \(C^{\zeta}(\RP^1)\). In fact, we can say a lot more. Before we present the results, let us first discuss two main motivations for this work. 

\subsubsection*{Motivation 1: Large deviation theory}
Large deviation theory for random matrix products has been used in the study of one-dimensional Ising models \cite{CPV93}, Schr\"odinger operators \cite{BL85,Bou12b}, affine stochastic recursions \cite{Kes73,Gol91,BDM16}, Mandelbrot percolations \cite{OS24}, and the ARCH-GARCH models in financial mathematics \cite{BDM16}. It also plays a central role in renewal theorems \cite{Kes73,Gol91,GP16,Li18,Li22}, equidistribution results \cite{BFLM11}, and in establishing Fourier decay and Frostman dimension estimates for the Furstenberg measure \cite{Gui90,Bou12b,Li18,Li22}.

The first motivation for the present paper is an open question in the probability community concerning large deviations. Besides its intrinsic interest, the transfer operator machinery needed to address this question can be used to attain precise large deviation estimates; see \cite{BM16,XGL20,XGL22,XGL23}. 
 
Let \(\mu^n\) denote the \(n\)-th convolution power of \(\mu\), i.e. \(\mu^n=\mu * \ldots * \mu\). For \(\alpha \in \R\), \(\epsilon>0\), and \(n \in \N\), we define the sets
\[E(\alpha,\epsilon,n):=\left\{g \in \GL(d,\R) :\left|\frac{1}{n}\log |g|-\alpha \right| \leq \epsilon \right\},\]
where \(|g|:=\sup_{x \in \BP} |gx|\) is the spectral norm. We also define the rate function
\[I(\alpha) := \lim_{\epsilon \rightarrow 0} \varlimsup_{n \rightarrow \infty} \frac{1}{n} \log \mu^n (E(\alpha,\epsilon,n)).\]
Note the rate function is usually defined to be the negative of this. In this paper \(I(\alpha) \leq 0\).

We say \(\mu\) is \textit{irreducible} if there does not exist any proper subspace of \(\R^d\) invariant under \(S_\mu\). When \(\mu\) is irreducible and compactly supported (more generally, when it has strong exponential moment), it was shown in \cite{Ser19} that \(I(\alpha)\) is equal to the negative of the Legendre transform of the limit
\[k(t):= \lim_{n \rightarrow \infty} \frac{1}{n} \log \int |g|^t \diff \mu^n(g).\]
That is, 
\begin{equation}\label{eqn:Ialpha}
    I(\alpha)=\inf_{t \in \R}\{k(t)-\alpha t\}. 
\end{equation}
We let 
\(\mathrm{A}:= \{\alpha \in \R: I(\alpha)>-\infty\}\). Then, by the above result, \(\mathrm{A}\) is an interval.

We let \(L(\mu)\) be the \textit{Lyapunov exponent} of \(\mu\) defined by 
\[L(\mu):= \lim_{n \rightarrow \infty} \frac{1}{n} \int \log |g| \diff \mu^n(g).\]
The limit exists by subadditivity and is equal to \(\inf_{n \in \N}  \frac{1}{n} \int \log |g| \diff \mu^n(g)\). It is known by the results of Le Page and the extensions that \(I(\alpha)\) has a unique 0 at \(\alpha=L(\mu)\) when \(\mu\) is irreducible. We are interested in the following natural question (cf. \cite[Question 1.10]{Ser19}):

\begin{question}\label{question2}
    When is \(\alpha \mapsto I(\alpha)\) analytic on \(\interior(\mathrm{A})\)?
\end{question}

There is an analogous open question in the fractal geometry community regarding analyticity of level sets defined in terms of Lyapunov exponent \cite{Fen09,DGR19}, and there has been a particular emphasis on studying associated equilibrium states and invariant measures \cite{FK11,BR18,Pir20,DGR22,DGR23,Rus24,Hau25}. 

It is known that \(I(\alpha)\) can fail to be analytic when \(\mu\) is reducible, e.g. when the matrices are diagonal (cf. \cite[\S 8.1]{Hau25}). We are specifically interested in Question \ref{question2} when \(\mu\) is irreducible.

Hautec{\oe}ur showed in \cite{Hau25} that \(k(t)\) is analytic on \((0,\infty)\) when \(\mu \in \MCP(\GL(d,\R))\) is proximal and irreducible. In the irreducible but \textit{not} strongly irreducible \(\SL(2,\R)\) setting, it is a classical result that \(L(\mu)=0\) (see \cite{GR85} and \cite[Theorem 6.1]{BL85}). This implies that \(k(t)=0\) for all \(t<0\). By a standard argument, we further have from the results of \cite{Hau25} that the right derivative of \(k(t)\) exists and is equal to 0.\footnote{\textit{Sketch}: With the notation in \cite[\S 7.2]{Hau25}, \( \mathbb{Q} \mapsto h_{\mu}(\mathbb{Q})\) is upper semi-continuous in the weak* topology, so, with \(\mathbb{Q}_t\) as given by \cite[Theorem 7.8]{Hau25}, any weak* limit of the measures \(\mathbb{Q}_t\) as \(t \downarrow 0\) must be equal to \(\mu\). By upper semi-continuity of the Lyapunov exponent, we have \(\varlimsup_{t \rightarrow 0^+} k'(t)=\varlimsup_{t \rightarrow 0^+} L(\mathbb{Q}_t) \leq L(\mu)=0 \).} Thus, by (\ref{eqn:Ialpha}), \(\alpha \mapsto I(\alpha)\) is analytic on \(\interior(\mathrm{A})\) in the irreducible but not strongly irreducible, proximal \(\SL(2,\R)\) setting.

The more difficult case for \(\mu \in \MCP(\SL(2,\R))\) is where \(\mu\) is strongly irreducible and proximal, especially outside of the uniformly hyperbolic setting which is better understood \cite[Theorem 2.9]{BKM20}. By the results of Le Page and Guivarc'h \cite{GP04,GP16}, \(k(t)\) is analytic on \((-\epsilon,\infty)\) which, by the result of Sert (\ref{eqn:Ialpha}), implies that \( I(\alpha)\) is analytic on \((L(\mu)-\epsilon',\sup \mathrm{A})\), where \(\epsilon, \epsilon'>0\) are completely non-explicit. To answer Question \ref{question2} in this setting requires studying the transfer operators \(P_t\) in the subtle \(t<0\) case.

In this paper we describe the value \( t_c<0\) such that \(P_t\) has a spectral gap on some space of \(\zeta\)-H\"older functions \(C^{\zeta}(\BP)\) for \(t \in (t_c, \infty)\), making progress on answering Question \ref{question2} in the \(\SL(2,\R)\), SIP setting. However, for reasons we now discuss, our results fall short of answering the question.

We define
\[ t_0:=\inf\{ t \in \R \cup \{-\infty\}:k(t)> I(0)\}=\sup\{ t \in \R:k(t) = I(0)\} .\]
It is known that \(t_0>-\infty\) when \(\mu\) is finitely supported and one of the matrices in \(\supp \mu\) is an irrational rotation \cite{DGR19} (see \cite[\S 10]{Rus24} for a discussion). This has also been shown for a specific example which has an invariant multicone (which prohibits an irrational rotation being in \(\supp \mu\)) in \cite{MQ25}, where they also give a complete description of the associated equilibrium states for all \(t \in \R\). For the definition of an invariant multicone, we refer the reader to \cite{BKM20}. We note that analyticity of \(\alpha \mapsto I(\alpha)\) holds if and only if \(k(t)\) is analytic on \((t_0,\infty)\) and is differentiable at \(t_0\), i.e. with derivative 0. If \(k(t)\) is not differentiable at \(t_0\), then \(I(\alpha)\) would have a transition point \(\alpha_c\) for which \(\alpha \mapsto I(\alpha)\) is linear on \([0,\alpha_c]\), with derivative in this region equal to \(-t_0\).

We will show that \(k(t)\) is analytic on \((t_c,\infty)\), but in general it is possible that \(t_0<t_c\). However, we emphasise that \(P_t\) failing to have a spectral gap on H\"older spaces for \(t<t_c\) does not imply that \(k(t)\) fails to be analytic at \(t_c\). For instance, in the case where \(S_{\mu}\) has an invariant multicone \(\MCC\), it may be that \(P_t\) has a spectral gap when taken to be acting on some space of H\"older functions on \(\MCC\), since one would expect contraction on average on \(\MCC\) for \(t>t_0\). Even outside the case where \(S_\mu\) has an invariant multicone, \(P_t\) may have a spectral gap for \(t_0<t<t_c\) when taken to be acting on a different space, e.g. other Besov spaces. We conjecture:

\begin{conjecture}
For compactly supported \(\mu\in \mathcal{P}(\SL(2,\R))\) satisfying SIP, \(t \mapsto k(t)\) is analytic everywhere except \(t_0\).
\end{conjecture}

For \(d>2\) we expect that further phase transitions in \(k(t)\) should be possible. For example, it is likely that there can exist \(t_0<t_{1,2}<0\) such that for \(t<t_{1,2}\),
\[k(t)= \lim_{n \rightarrow \infty} \frac{1}{n} \log \int |g^{\wedge 2}|^{t/2} \diff \mu^n(g),\]
where \(g^{\wedge 2}\) is the map corresponding to \(g\) acting on \(\wedge^2 \R^d\) in the natural way. Note this phenomenon has already been observed in the \(d=2\) case (in which case \(t_0 \equiv t_{1,2}\)). This is also inspired by thermodynamic formalism considerations: for \(t>t_{1,2}\) equilibrium states should have distinct top Lyapunov exponent, but not for \(t<t_{1,2}\). Outside of these phase changes, we expect \(k(t)\) to depend analytically when \(\mu\) is \textit{fully strongly irreducible} and \textit{fully proximal}, that is, when the induced action on \(\wedge^i \R^d\) is strongly irreducible and proximal for all \(i \in \{1,\ldots,d-1\}\); see \cite{BT22}. We further conjecture:

\begin{conjecture}
For compactly supported \(\mu\in \mathcal{P}(\GL(d,\R))\), \(d \geq 2\), which is fully strongly irreducible and fully proximal, \(t \mapsto k(t)\) is analytic everywhere except for at most \(d-1\) points.
\end{conjecture}

We define \(t_{i,i+1}\) analogously to \(t_{1,2}\) above. It is not clear whether \(k(t)\) should be differentiable at these transition points \(t_{i,i+1}\). Already in the \(\SL(2,\R)\) case this is delicate: for \(t<0\) \(\mu \mapsto tL(\mu)\) is \textit{lower} semi-continuous, so the argument used, for example, in the irreducible case above to show differentiability at 0 breaks down completely, even if we did have the relevant equilibrium states and spectral gaps for \(t\in (t_0,\infty)\). We ask:

\begin{question}
For compactly supported \(\mu\in \mathcal{P}(\GL(d,\R))\) which is fully strongly irreducible and fully proximal, is \(t \mapsto k(t)\) differentiable at the transition points \(t_{i,i+1}\) for \(1 \leq i \leq d-1\)?
\end{question}

\subsubsection*{Motivation 2: Dimension theory of Furstenberg measure} 
Recall in the SIP setting the Furstenberg measure is the unique measure on the projective space satisfying \(\mu * \nu_F =\nu_F\). The study of the dimension theory and regularity of Furstenberg measure, and stationary measures more generally, has received considerable attention in recent years, e.g. \cite{BPS12, Bou12b,Hoc14, HS17,Shm19a,Var19, Rap21,ST21, Fen23, LL23, Kit25,CJ26,Jur26}. The dimension theory of Furstenberg measure is central to the study of self-affine measures and sets \cite{BK17,BHR19,BJKR21,HR22,Rap24}. 

Strengthening a result of Ledrappier \cite{Led83}, Hochman and Solomyak showed in \cite{HS17} that \(\nu_F\) is \textit{exact dimensional} when \(\mu \in \MCP(\SL(2,\R))\) satisfies SIP and has \(\int \log|g| \diff \mu<\infty\). This means there exists \(D>0\) such that 
\[\lim_{r \rightarrow 0} \frac{- \log \nu_F(B(x,r))}{- \log r} \rightarrow D\]
for \(\nu_F\)-a.e. \(x\). Furthermore, the constant \(D\) is known to be equal to
\[D:= \frac{h_F(\mu)}{2L(\mu)},\]
where
\[h_F(\mu):=\int \int \log \frac{\diff g \nu_F}{\diff \nu_F}(x) \diff g\nu_F(x) \diff \mu(g)\]
is the \textit{Furstenberg entropy}. This result has been generalised to higher dimensions when \(\mu\) is finite/discrete in the papers \cite{Rap21,LL23}. Under the Diophantine property, Hochman and Solomyak improved the above result by showing that \(h_F(\mu)\) is equal to the random walk entropy \(h_{\mathrm{RW}}(\mu)\) (see \cite[Theorem 1.1]{HS17}), which is easier to work with in practice. This result has recently been extended to \(\SL(2,\C)\) matrices by Rapaport and Ren in \cite{RR25}.

There has also been significant work towards understanding when \(\nu_F\) is absolutely continuous and has smooth density. It was conjectured in \cite{KP11} that \(\nu_F\) is singular whenever \(\mu\) is finitely supported, but this was disproved in \cite{BPS12}, though a version of this conjecture is still open for discrete groups. Bourgain gave examples of \(\mu \in \MCP(\SL(2,\R))\) with finite support such that the Furstenberg measure is arbitrarily smooth in \cite{Bou12a}, and this was extended to higher dimensions in \cite{BQ18}. In \cite{Kit25}, Kittle gave a sufficient condition for \(\nu_F\) to be absolutely continuous and used this to construct a broad class of examples for which this holds. 

In recent work \cite{Usu25}, Usuki studied \(L^q\) dimension of \(\nu_F\) in the \(\SL(2,\R)\) setting assuming a strong Diophantine condition and uniform hyperbolicity, showing that previously unseen behaviour is possible as \(q \rightarrow \infty\). In particular, it can occur that for \(q\) sufficiently large the \(L^q\)-dimension is equal to the Frostman dimension. Usuki further proved a detailed description of the \(L^q\)-dimensions for \(q>1\) below this transition point and, when the transition does not occur, for all \(q>1\). Phase transitions in \(L^q\)-dimension have also been exhibited in the self-affine literature (via different mechanisms); see \cite{Fal99,FLMY21}. For a survey on and applications of \(L^q\) spectrum in related settings, we refer the reader to \cite{Shm19a,Shm19b,JR21}.

As mentioned above, it is a classical result of Guivarc'h \cite{Gui90} that \(\dim_F \nu_F>0\), and this was used alongside large deviation theory to show Fourier decay of \(\nu_F\) in \cite{Li18,Li22}. However, as we have already discussed, these results are not explicit, and the strengthening of the result of Guivarc'h in a related setting in \cite{AS22} is not sharp. Obtaining an explicit and refined dimension theory (e.g. multifractal formalism results) for \(\nu_F\) is a notoriously difficult problem, and much remains unknown.

The second motivation for this paper is to introduce a method to obtain dimension-theoretic results for Furstenberg measure which hold generally, that is without needing additional technical assumptions beyond the natural SIP condition (consequently, giving less strong results than when stronger conditions are assumed). In particular, in this paper we prove a formula for the Frostman dimension of \(\nu_F\) in the \(\SL(2,\R)\), SIP setting via transfer operators and Fourier analysis. We remark that this result is new even under the additional assumptions of the strong Diophantine condition and uniform hyperbolicity when \(\dim_F \nu_F<1\).

Since the Fourier decay of \(\nu_F\) is intimately related to the large deviation theory previously discussed (see \cite[Remark 1.8]{Li22}), it is likely that our results can be used to obtain more explicit Fourier decay estimates as well. The author further expects that the methods in this paper can be extended to study the \(L^q\)-dimensions of \(\nu_F\) and plans to revisit this in a subsequent paper.

\subsection{Main results}
For now and hereafter, let \(\mu\) be a compactly supported Borel probability measure on \(\SL(2,\R)\) satisfying SIP. To ease the notation we let \(\BP:=\RP^1\). For \(x=\R u, y=\R v \in \BP\), we will write
\begin{equation}\label{eqn:innerproduct}
    |\innerproduct{x}{y}|:=\frac{|\innerproduct{u}{v}|}{|u| |v|},
\end{equation}
 and recall we defined the metric on \(\BP\) by 
 \[ d_{\BP}(x,y):= \frac{| u \wedge v|}{|u| |v|} \equiv \sqrt{ 1-|\innerproduct{x}{y}|^2}.\]
 
Recall we also defined the sets
\[E(\alpha,\epsilon,n):=\left\{g \in \SL(2,\R) :\left|\frac{1}{n}\log |g|-\alpha \right| \leq \epsilon \right\}.\]
We will denote by \(\omega_-(g)\) the element in \(\BP\) corresponding to the smallest singular value in the preimage of \(g\) (see \S \ref{subsec:tc'zetat'}). For \(\alpha \geq 0\), \(\epsilon>0\), \(n \in \N\),  \(\gamma \in [0,2]\), and \(x \in \BP\), we define
\[ E_{\gamma}(\alpha,\epsilon,n,x):=\left\{g \in E(\alpha,\epsilon,n):d_{\BP}(x,\omega_-(g)) \leq e^{-n(\gamma-\epsilon) \alpha} \right\}. \]
We further define
\[I_{\gamma}(\alpha,\epsilon,n):= \frac{1}{n} \log  \left(  \sup_{x \in \BP} \mu^n(E_{\gamma}(\alpha,\epsilon,n,x)) \right) ,\]
\[I_{\gamma}(\alpha,\epsilon):= \varlimsup_{n \rightarrow \infty} I_{\gamma}(\alpha,\epsilon,n) ,\]
\[I_{\gamma}(\alpha):= \lim_{\epsilon \rightarrow 0} I_{\gamma}(\alpha,\epsilon) .\]
Note that \(I_0(\alpha) \equiv I(\alpha)\) defined in the introduction and that \(I_\gamma(0) \equiv I_0(0)\) for any \(\gamma \in [0,2]\). 

Let us give a brief justification of why these are natural in the study of the spectral properties of \(P_t\). Following \cite[\S 4]{Mau14} we define 
\[k^+(t):= \varlimsup_{n \rightarrow \infty} \frac{1}{n}\log \sup_{x \in \BP} P_t^n 1(x), \qquad k^-(t):=\varliminf_{n \rightarrow \infty} \frac{1}{n}\log \inf_{x \in \BP} P_t^n 1(x).\] 
Using only elementary arguments, one can show that 
\[k^+(t)=\sup_{\alpha \in \mathrm{A}} \sup_{\gamma \in [0,2]} \{I_{\gamma}(\alpha)+t(1-\gamma)\alpha\}, \quad \forall t \in  \R,\] 
and \(k^-(t) =k(t)\) for all \(t \in (-1,0]\) (see \S\ref{subsec:ktk+tk-t}). Moreover, the crucial proposition \cite[Proposition 2.3]{BL85} used to prove the spectral gap at \(t=0\) is equivalent to the statement that \(\sup_{\alpha \in \mathrm{A}} \sup_{\gamma \in [1,2]}I_\gamma(\alpha)<0\) (see Proposition \ref{prop:Igamma<0forgammain12}). Already implicit in the work of Le Page is the following phenomenon: estimates of \(I_\gamma(\alpha)\) for larger \(\gamma\) implies non-trivial estimates of \(I_{\gamma}(\alpha)\) for smaller \(\gamma\).

We define
\[t_c:=  \inf\{t \in \R: k(t)> \sup_{\alpha \in \mathrm{A}} \{I_{2}(\alpha)-t \alpha\} \} \]
and note that \(-1 \leq t_c<0\). For \(t \in (t_c,0]\) we define \(\zeta_t\) to be the unique solution of 
\[ \sup_{\alpha \in \mathrm{A}}  \{ I_{2}(\alpha)-t \alpha+2\zeta_t \alpha\}=k(t).\]  
This is equivalently equal to
\begin{equation*}
    \zeta_t:=\sup \{\zeta \in \R:\sup_{\alpha \in \mathrm{A}} \{I_2(\alpha)-t\alpha+2\zeta \alpha \}<k(t)\}.
\end{equation*}
Necessarily, \(\zeta_t \in (0,1+t]\). 

The following theorem describes the region in which \(P_t\) has a spectral gap on a H\"older space \(C^{\zeta}(\BP)\).

\begin{theo}\label{theo:spectralgap}
The map \(t \mapsto k(t)\) is analytic on \((t_c,0)\). Moreover, for all \(t_c<t \leq 0\), there exists a unique probability measure \(\nu_t \in \MCP(\BP)\) such that 
        \[(P_t)_{\Cdot} \nu_t=e^{k(t)} \nu_t, \]
        and a unique H\"older continuous function \(h_t:\BP \rightarrow \R\) satisfying \(\nu_t(h_t)=1\) and 
        \[P_t h_t=e^{k(t)} h_t.\]
        For all \(0<\zeta<\zeta_t\), \(P_t:C^{\zeta}(\BP) \rightarrow C^{\zeta}(\BP)\) has a spectral gap:
        \[P_t=e^{k(t)}(\nu_t \otimes  h_t+S_t)\]
        where \((\nu_t \otimes  h_t) f:=\nu_t(f) h_t\) and \(S_t:C^{\zeta}(\BP) \rightarrow C^{\zeta}(\BP)\) satisfies \(\rho(S_t)<1\).
\end{theo}

We also define the transfer operators \({}^*\! P_t\) by
\[{}^*\! P_t \psi(w):=\int |g^* w|^t \psi(g^* w) \diff \mu(g).\]
where \(g^*\) is the adjoint (i.e. transpose) of \(g\). For a probability measure \(\nu \in \MCP(\BP)\) we write \(({}^*\! P_t)_{\Cdot} \nu\) to be the measure defined by 
\[({}^* \! P_t)_{\Cdot} \nu (f)= \nu ({}^* \! P_t f) \]
for all measurable \(f\). As in the \(t>0\) case \cite[Theorem 2.6]{GP16} and \((-\epsilon,0)\) case \cite{GQX22}, we can describe the eigenfunctions \(h_t\) in terms of eigenmeasures of \({}^*\! P_t\). A new development in this work is that we further relate the H\"older regularity of \(h_t\) with the Frostman dimension of the eigenmeasures, and both of these with our \(\zeta_t\).

\begin{theo}\label{theo:measure}
    For all \(t \in (t_c,0]\) there exists a unique probability measure \({}^{*} \! \nu_t \in \MCP(\BP)\) satisfying \(({}^{*} \!  P_t )_{\Cdot} {}^{*} \! \nu_t=e^{k(t)} {}^{*} \! \nu_t\) and \(\dim_F{}^{*} \! \nu_t>-t\). This satisfies \(\dim_F{}^{*} \! \nu_t=-t+\zeta_t\). Moreover, the eigenfunction \(h_t(x)\) in Theorem \ref{theo:spectralgap} is equal to a scalar multiple of \begin{equation}\label{eqn:eigenfunction}
            x \mapsto \int |\innerproduct{x}{w}|^t \diff {}^{*} \! \nu_t(w).
        \end{equation}
        For \(t \not=0\) we further have \(\sup\{\zeta>0: h_t \in C^{\zeta}(\BP) \}=\zeta_t\).
\end{theo}
Observe that for \(t \in (t_c,0)\) \(P_t\) has a spectral gap on \(C^{\zeta}(\BP)\) for all \(\zeta>0\) for which the natural eigenfunction (\ref{eqn:eigenfunction}) is in \(C^{\zeta}(\BP)\). We also observe that either \(t_c=t_0\) or \(\sup\{\zeta>0: h_t \in C^{\zeta}(\BP) \} \rightarrow 0\) as \(t \downarrow t_c\) (or both). That is, \(t_c\) is indeed the point at which \(P_t\) ceases to have a spectral gap on a space of H\"older functions on \(\BP\).

It is noteworthy that the Frostman dimension of the eigenmeasures is controlled by the regularity at the very smallest scales, that is, the quantities \(I_{\gamma}(\alpha)\) for \(\gamma=2\). The next theorem gives quantitative bounds for the regularity (the \(I_{\gamma}(\alpha)\)) at larger scales (\(\gamma<2\)) from the quantities \(I_2(\alpha)\). Again it says that in some sense irregularity cannot be gained at larger scales.

\begin{theo}\label{theo:Igammaestimates}
    Let \(t \in (t_c,0]\). For all \(\alpha \in \mathrm{A} \) and \(\gamma \in [0,2]\),
        \[I_{\gamma}(\alpha)+t(1-\gamma) \alpha \leq k(t)-\zeta_t \gamma \alpha.\]
\end{theo}

Let us highlight the following corollary of Theorem \ref{theo:measure}. We further define the sets
\[ {}^* \! E_{2}(\alpha,\epsilon,n,x):=\left\{g \in E(\alpha,\epsilon,n):d_{\BP}(x,\upsilon_+(g)) \leq e^{-n(2-\epsilon)\alpha} \right\}, \]
where \(\upsilon_+(g)\) denotes the element in \(\BP\) corresponding to the largest singular value in the image of \(g\). Let
\[{}^* \! I_{2}(\alpha)=\lim_{\epsilon \rightarrow 0} \varlimsup_{n \rightarrow \infty}  \frac{1}{n} \log \sup_{x \in \BP}   \mu^n({}^* \! E_{2}(\alpha,\epsilon,n,x)) .\] 

\begin{cor}\label{corr}
   The Frostman dimension of the Furstenberg measure \(\nu_F \equiv \nu_0\) satisfies \(\dim_F \nu_F=\zeta\), where \(\zeta \in (0,1]\) is given by the unique solution of
   \[ \sup_{\alpha \in \mathrm{A}}  \{ {}^*\! I_{2}(\alpha)+2 \zeta \alpha\}=0. \]
\end{cor}

As is often the case in dimension-theoretic results, the upper bound of \(\dim_F \nu_F\) is relatively straightforward, but the lower bound is much more involved.

\subsection{Sketch of the proof}
We now give a sketch of the proofs of Theorems \ref{theo:spectralgap}-\ref{theo:Igammaestimates} as well as a layout of the paper. Theorems \ref{theo:spectralgap}-\ref{theo:Igammaestimates} are proved concurrently. By this we mean, we first prove them for some \(t_c' \in [-1,0)\) and \(\zeta_t' \in (0,1+t]\). They are then proved by showing that \(\zeta_t'=\zeta_t\) and \(t_c'=t_c\).

In the next section, Section \ref{sec:elementary}, we prove some basic results about the \(I_{\gamma}(\alpha)\) and some expansion estimates relating to the sets \(E_{\gamma}(\alpha,\epsilon,n,x)\). Recall the definitions of \(k(t), k^+(t)\), and \(k^-(t)\). We also show, using only elementary methods (no irreducibility assumptions are needed), that 
\[k(t)= \sup_{\alpha \in \mathrm{A}} \{I_0(\alpha)+t\alpha\},  \quad \forall t \in \R ,\]
\[k^+(t)=\sup_{\alpha \in \mathrm{A}} \sup_{\gamma \in [0,2]} \{I_{\gamma}(\alpha)+t(1-\gamma)\alpha\}, \quad \forall t \in \R,\]
and \(k(t)=k^-(t)\) for all \(-1<t \leq 0\). Observe that if \(-1<t<0\) and \(P_t\) has a strictly positive, continuous eigenfunction \(h_t\) satisfying \(P_t h_t= e^{k^+(t)} h_t\), then immediately \(k^+(t)=k(t)\).

In Section \ref{section:spectralgap} we describe the region for which \(P_t\) has a spectral gap, proving Theorem \ref{theo:spectralgap} with \(t_c'\), \(\zeta_t'\) in place of \(t_c\), \(\zeta_t\). We define 
\[t_c':= \inf\left\{t \geq -1:\sup_{\alpha \in \mathrm{A}} \sup_{\gamma \in [1,2]} \{I_{\gamma}(\alpha)+t(1-\gamma)\alpha\}<k^+(t) \right\}.\]
We show, using the crucial proposition \cite[Proposition 2.3]{BL85}, that \(t_c'<0\). We note that \((t_c',\infty)\) is the region where \(k^+(t)\) is strictly increasing in \(t\).

For \(t \in (t_c',0]\) we define \(\zeta_t'\) by 
\[\zeta_t':=\sup \left\{\zeta>0: \sup_{\alpha \in \mathrm{A}} \sup_{\gamma \in [1,2]} \{ I_\gamma(\alpha)+t(1-\gamma)\alpha -2\zeta (1-\gamma) \alpha\}< k^+(t) \right\} .\]
Necessarily, \(0<\zeta_t' \leq 1+t\). For \(t_c'<t\leq 0\) we will show in Lemma \ref{lem:quasi-compact} that \(P_t\) is quasi-compact on \(C^{\zeta}(\BP)\) for all \(0<\zeta<\zeta_t'\). We further show that there is a spectral gap for these \(\zeta\) and that the leading eigenfunction is strictly positive (Proposition \ref{prop:spectralgap}). By the above relations, this implies \(k(t)=k^+(t)\); that is, for these \(t\), 
\[\sup_{\alpha \in \mathrm{A}} \sup_{\gamma \in [0,2]}\{I_{\gamma}(\alpha)+t(1-\gamma)\alpha\}= \sup_{\alpha \in \mathrm{A}} \{I_0(\alpha)+t\alpha\}.\]
Already, using only estimates for \(I_{\gamma}(\alpha)\) with \(\gamma \in [1,2]\), we will have obtained the non-trivial estimates that \(I_{\gamma}(\alpha)+t(1-\gamma)\alpha \leq k(t)\) on \((t_c',0]\). It is possible to get better estimates via the Frostman dimension of \({}^* \! \nu_t\), a natural eigenmeasure of \({}^* \! P_t\).

Hence, in Section \ref{subsec:eigenfunction}, we describe the eigenfunction in terms of an eigenmeasure of \({}^* \! P_t\). We note that our argument does not require that \({}^* \! P_t\) has a spectral gap at \(t\) (there is no problem being in the region \(t_c<t<{}^* \! t_c\), should this be non-empty, where \({}^*\! t_c \) is defined analogously to \(t_c\)). Let \(\lambda\) be the Lebesgue measure on \(\BP \simeq \R /\pi \Z \) normalised so that \(\lambda(\BP)=1\). For \(t_c'<t<0\) we let \({}^* \! \nu_t \in \MCP(\BP)\) be a weak* limit of the sequence \( {}^* \! \eta_n:= \frac{({}^* \! P_t^n)_{\Cdot} \lambda}{({}^* \! P_t^n)_{\Cdot} \lambda(1)} \). Using a novel argument, in Lemma \ref{lem:lebesgueaeequality} we deduce that, up to a scalar multiple, \(h_t(x)=\int |\innerproduct{x}{w}|^t \diff {}^* \! \nu_t(w)\) for \(\lambda\)-a.e. \(x \in \BP\). 

In the appendix, for \(-1<t<0\) and \(\nu \in \MCM(\BP)\) we will consider Riesz potentials of the form  
\[\MCI_{t,\nu}(x) :=\int d_{\BP}(x,y)^t \diff \nu(y) \equiv \int |\innerproduct{x}{y^{\perp}}|^t\diff \nu(y), \]
where \(y^\perp\) is the element in \(\BP\) perpendicular to \(y\), and relate the H\"older regularity of these potentials with the Frostman dimension of \(\nu\). Since these results are generalisations of classical results \cite{Car63,Wal66} for singular potentials on \(\R^d\) and require different methods from the rest of the paper, we defer the proofs until the appendix. 

We show in Proposition \ref{prop:hisholder} that for \(\nu\) with \(\dim_F \nu>-t\), \(\MCI_{t,\nu}\) is \(\zeta\)-H\"older continuous for all \(0<\zeta<\min\{1,-t+\dim_F \nu\}\). In Proposition \ref{prop:Holdergivesdimensionbound} we prove the following converse: If \(\MCI_{t,\nu}\) is \(\lambda\)-almost everywhere equal to some \(\zeta\)-H\"older continuous \(F:\BP \rightarrow \R\), then \(\dim_F \nu \geq -t+\zeta\) and \(\MCI_{t,\nu} \equiv F\). This is proved using Fourier analysis and we use a recent result of Roncal and Stinga \cite[Theorem 1.5]{RS16}. In Proposition \ref{prop:Inu1=Inu2impliesnu1=nu2} we also show that if \(\lambda\)-almost everywhere we have \(\MCI_{t,\nu_1}=\MCI_{t,\nu_2} \), then \(\nu_1=\nu_2\). Hence, applying these results, we are able to show that 
\( \int |\innerproduct{x}{w}|^t \diff {}^* \! \nu_t(w) =h_t(x)\) for all \(x \in \BP\) (again up to a scalar multiple), \(\dim_F{}^* \! \nu_t \geq -t+\zeta_t'\), and \(({}^* \! P_t)_{\Cdot} {}^*\! \nu_t=e^{k(t)}{}^* \! \nu_t\). 

In Section \ref{sec:proofoftheorems}, the next step in the proof is to show the equality of \(\zeta_t=\zeta_t'\) for \(t_c'<t<0\). This is the crux of our proof, which we give a sketch of now. We show in Lemma \ref{lem:boundsforQgamma}, using only that \(({}^* \! P_t)_{\Cdot} {}^*\! \nu_t=e^{k(t)} {}^* \! \nu_t\), that for all \(\alpha \in \mathrm{A}\) and \(\gamma \in [0,2]\),
\begin{equation}\label{eqn:insketchstandardinequality}
    I_{\gamma}(\alpha)+t(1-\gamma)\alpha \leq k(t)-(t+\dim_F{}^* \! \nu_t) \gamma \alpha.
\end{equation}
By the definition of \(\zeta_t\) this easily implies \(\dim_F{}^* \! \nu_t \leq -t+\zeta_t\). Moreover, combining the inequality  \(\dim_F{}^* \nu_t \geq -t+\zeta_t' \) with (\ref{eqn:insketchstandardinequality}), we have that for all \(\alpha \in \mathrm{A}\) and \(\gamma \in [0,2]\),
\begin{equation}\label{eqn:insketchinequalitywithzeta}
    I_{\gamma}(\alpha)+t(1-\gamma)\alpha \leq k(t)-\zeta_t' \gamma \alpha.
\end{equation}
By the definition of \(\zeta_t'\) and upper semi-continuity of \(I_{\cdot}(\cdot)\) (see Lemmas \ref{lem:USC} and \ref{lem:zetatunique}), there exists \(\beta \in \mathrm{A} \setminus \{0\}\) and \(\gamma \in (1,2]\) such that 
\[I_{\gamma}(\beta)+t(1-\gamma)\beta-2\zeta_t'(1-\gamma)\beta=k(t).\]
Putting this into (\ref{eqn:insketchinequalitywithzeta}), we have
\[2\zeta_t'(1-\gamma)\beta \leq - \zeta_t' \gamma \beta,\]
so \(\gamma \geq 2\), that is \(\gamma =2\). Hence, it follows that \(\zeta_t' =\zeta_t\) for all \(t_c'< t<0\), and furthermore for \(t=0\) by continuity of \(\zeta_t\) and \(\zeta_t'\). From this it is not hard to also show that \(t_c'=t_c\).

In this way, we will have proved Theorems \ref{theo:spectralgap}-\ref{theo:Igammaestimates} except for the statement that \(\dim_F{}^* \!  \nu_0=\zeta_0\). This final statement can be proved using the estimates in Theorem \ref{theo:Igammaestimates}, and this will be done in Section \ref{subsec:Frostmandimensionofnu0}. We remark that, in contrast to the \(t<0\) case, we require the spectral gap of \({}^* \! P_0\) for this.

\section{Preliminaries}\label{sec:elementary}
\subsection{The quantities \(I_{\gamma}(\alpha)\)}\label{subsec:tc'zetat'}

For each \(g \in \SL(2,\R)\) with \(|g|>1\) we choose a decomposition \(g=K D L\), where \(D=\mathrm{diag}(\kappa,\kappa^{-1}) \in \SL(2,\R)\) is a diagonal matrix with \(\kappa > 1\) and \(K,L \in O_2(\R)\). We denote by \(\omega_+(g)\) (resp. \(\omega_-(g)\)) the element in \(\BP\) corresponding to the largest (resp. smallest) singular value in the preimage of \(g\); that is,
\[\omega_+(g):=\R L^{-1}e_1, \qquad \omega_-(g):=\R L^{-1}e_2 \]
where \(e_1=(1,0)\) and \(e_2=(0,1)\). Similarly we define
\(\upsilon_+(g)\) (resp. \(\upsilon_-(g)\)) to be the element in \(\BP\) corresponding to the largest (resp. smallest) singular value in the image of \(g\); that is,
\[\upsilon_+(g):=\R K e_1, \qquad \upsilon_-(g):=\R K e_2.\]
We note that for \(g \in \SL(2,\R)\) we always have \(\omega_+(g)=\upsilon_+(g^*)\), \(\omega_-(g)=\upsilon_-(g^*)\), \(\omega_-(g)=\upsilon_+(g^{-1})\), \(\omega_+(g)=\upsilon_-(g^{-1})\), and \(d_{\BP}(\omega_-(g),\omega_+(g))=d_{\BP}(\upsilon_-(g),\upsilon_+(g))=1\). We will use these facts routinely.  If \(|g|=1\) we can choose \(\omega_{\pm}(g), \upsilon_{\pm}(g)\) arbitrarily so that the above relations also hold. 

It will be necessary to sometimes work with modified versions of the \(I_{\gamma}(\alpha)\). For  \(0 \leq \gamma < 2\) we define the sets
\[ \widetilde{E}_{\gamma}(\alpha,\epsilon,n,x):=\left\{ g \in E(\alpha,\epsilon,n): e^{-n(\gamma+\epsilon) \alpha} \leq d_{\BP}(x,\omega_-(g)) \leq e^{-n(\gamma -\epsilon)\alpha} \right\}\]
and we define \(\widetilde{E}_2(\alpha,\epsilon,n, x):= E_2(\alpha,\epsilon,n,x)\). We remark that \(\widetilde{E}_2(0,\epsilon,n,x)=E(0,\epsilon,n)\) and \(\widetilde{E}_\gamma(0,\epsilon,n,x) = \emptyset\) for all \(\gamma<2\) (we will consider the \(\alpha=0\) case separately in our analysis where appropriate). Outside of this, we always have \(\widetilde{E}_\gamma(\alpha,\epsilon,n,x) \subseteq E_\gamma(\alpha,\epsilon,n,x) \). We also note that
\[\widetilde{E}_{0}(\alpha,\epsilon,n, x)=\left\{g \in E(\alpha,\epsilon,n): d_{\BP}(x,\omega_-(g)) \geq  e^{-n\epsilon \alpha} \right\}. \]
For \(0 \leq \gamma \leq 2\) we define
\[\widetilde{I}_{\gamma}(\alpha,\epsilon,n):= \frac{1}{n} \log \sup_{x \in \BP}  \mu^n \left( \widetilde{E}_\gamma(\alpha,\epsilon,n,x) \right),\]
\[\widetilde{I}_{\gamma}(\alpha,\epsilon):=\varlimsup_{n \rightarrow \infty} \widetilde{I}_{\gamma}(\alpha,\epsilon,n),\]
\[\widetilde{I}_{\gamma}(\alpha)=\lim_{\epsilon \rightarrow 0} \widetilde{I}_{\gamma}(\alpha,\epsilon).\]

We let \(\overline{\alpha}>0\) be such that \(\supp \mu \subseteq \{g \in \SL(2,\R): |g| \leq e^{\overline{\alpha}}\}\). For \(\epsilon>0\), we define the sets
\[\mathrm{A}(\epsilon):= \left\{i \epsilon:i \in \N \cup \{0\}, 0\leq i\epsilon \leq \overline{\alpha} \right\} \]
and
\[\Gamma(\epsilon):= \left\{i \epsilon:i \in \N \cup \{0\}, 0 \leq i\epsilon<2 \right\}\cup\{2 \}.\]

\begin{remark}
    We are not in general assuming that \(0 \in \mathrm{A}\), but we will sometimes consider the sets \(E(0,\epsilon,n)\) as \(\alpha=0\) can require a slightly different argument. There is no problem if \(0 \not \in \mathrm{A}\), though the argument for \(\alpha=0\) will be unnecessary in this case, and similarly if \(\alpha>\overline{\alpha}\).
\end{remark}

\begin{lemma}\label{lem:gammaiscover}
     Let \(\alpha \geq 0\), \(\epsilon>0\), \(n \in \N\), and \(x \in \BP\). Then,
     \(E(\alpha,\epsilon,n)= \bigcup_{\gamma \in \Gamma(\epsilon)} \widetilde{E}_{\gamma}(\alpha,\epsilon,n,x).\)
\end{lemma}

\begin{proof}
The inclusion `\(\supseteq\)' is clear. To show the other inclusion, fix \(0<\epsilon<1\), \(n \in \N\), and \(x \in \BP\). The \(\alpha=0\) case is trivial since \(E(\alpha,\epsilon,n)=\widetilde{E}_2(\alpha,\epsilon,n,x)\), so assume \(\alpha>0\). Let \(g \in E(\alpha,\epsilon,n)\). If \(d_{\BP}(x,\omega_-(g)) \leq e^{-2n\alpha}\), then \(g \in \widetilde{E}_{2}(\alpha, \epsilon,n, x)\). Suppose then that \(d_{\BP}(x,\omega_-(g))> e^{-2n\alpha}\) and let \(0 \leq \gamma <2\) be such that \(d_{\BP}(x,\omega_-(g)) = e^{-n \gamma \alpha}\). Let \(\gamma' \in \Gamma(\epsilon) \setminus \{2\}\) be such that \(|\gamma-\gamma'|\leq \epsilon\). We have 
\[d_{\BP}(x,\omega_-(g))=e^{-n \gamma \alpha} \leq e^{-n (\gamma'-\epsilon)\alpha}\]
and, similarly, 
\[d_{\BP}(x,\omega_-(g))=e^{-n \gamma \alpha} \geq e^{-n (\gamma'+\epsilon)\alpha} .\]
Thus, \(g \in \widetilde{E}_{\gamma'}(\alpha,\epsilon,n, x)\).
\end{proof}

\begin{lemma}\label{lem:alphagammaiscover}
Let \(\epsilon>0\) and \(n \in \N\) and \(x \in \BP\). We have
 \[\{g \in \SL(2,\R) :|g| \leq e^{n\overline{\alpha}} \} \subseteq \bigcup_{\alpha \in \mathrm{A}(\epsilon)} \bigcup_{\gamma \in \Gamma(\epsilon)} \widetilde{E}_{\gamma}(\alpha,\epsilon,n,x).\]
\end{lemma}

\begin{proof}
Let \(g \in \SL(2,\R)\) be such that \(|g| \leq e^{n\overline{\alpha}}\). Write \(\alpha:=\frac{1}{n}\log|g|\) and let \(\alpha' \in \mathrm{A}(\epsilon)\) be such that \(|\alpha'-\alpha| \leq \epsilon\). Clearly we have \( g \in E(\alpha',\epsilon,n)\). Hence, the lemma follows by Lemma \ref{lem:gammaiscover}.
\end{proof}

The next two lemmas are straightforward properties of the quantities \(I_{\gamma}(\alpha)\).

\begin{lemma}\label{lem:elementaryQgammainequality}
For all \(\alpha \in \mathrm{A}\) and \(0 \leq \gamma' \leq \gamma \leq 2\), 
\[I_{\gamma}(\alpha) \leq I_{\gamma'}(\alpha) \leq I_{\gamma}(\alpha)+(\gamma-\gamma')\alpha.\]
\end{lemma}

\begin{proof}
    The inequality \( I_{\gamma}(\alpha) \leq I_{\gamma'}(\alpha)\) is trivial. To show the other inequality, let \(\alpha \in \mathrm{A}\), \(0 \leq \gamma' \leq \gamma \leq 2 \), and \(x \in \BP\). Note we can cover \(B(x,e^{-n(\gamma'-\epsilon) \alpha})\) by \(Ce^{n((\gamma-\gamma')\alpha)} \) balls of radius \(e^{-n(\gamma-\epsilon)\alpha}\), where \(C>0\) can be taken to be independent of \(n \in \N\) and \(x \in \BP\). It follows that 
    \[\mu^n(E_{\gamma'}(\alpha,\epsilon,n,x)) \leq  Ce^{n((\gamma-\gamma')\alpha+I_{\gamma}(\alpha,\epsilon,n))}.\]
    Since \(x \in \BP\) was arbitrary, this implies
    \[\frac{1}{n} \log \sup_{x \in \BP} \mu^n(E_{\gamma'}(\alpha,\epsilon,n,x))  \leq (\gamma-\gamma')\alpha+I_{\gamma}(\alpha,\epsilon,n)+\frac{1}{n}\log C.\]
    Taking limits as \(n \rightarrow \infty\) then \(\epsilon \rightarrow 0\), the lemma follows.
\end{proof}

\begin{lemma}\label{lem:USC}
Suppose that \((\alpha_k,\gamma_k) \in [0,\infty) \times [0,2] \) is such that \((\alpha_k,\gamma_k) \rightarrow (\alpha,\gamma) \in [0,\infty) \times [0,2]\) in \((\R^2,|\cdot|_{\infty})\). Then, for any \(\epsilon_k \rightarrow 0\), 
\[\varlimsup_{k \rightarrow \infty} I_{\gamma_k}(\alpha_k) \leq  \varlimsup_{k \rightarrow \infty} I_{\gamma_k}(\alpha_k,\epsilon_k) \leq I_{\gamma}(\alpha),\]
and, further assuming that \(\alpha>0\),
\[\varlimsup_{k \rightarrow \infty} \widetilde{I}_{\gamma_k}(\alpha_k) \leq  \varlimsup_{k \rightarrow \infty} \widetilde{I}_{\gamma_k}(\alpha_k,\epsilon_k) \leq \widetilde{I}_{\gamma}(\alpha). \]
In particular, the maps \(I_{\cdot}(\cdot):[0,\infty) \times [0,2] \rightarrow \R \cup\{-\infty\} \) and \(\widetilde{I}_{\cdot}(\cdot):(0,\infty) \times [0,2] \rightarrow \R \cup \{-\infty\}\) are upper semi-continuous.
\end{lemma}

\begin{proof} 
 Let \(\epsilon'>0\). If \(\alpha=0\), we can let \(\epsilon>0\) be such that 
    \[I_0(0,2\epsilon) \leq I_0(0)+\epsilon'.\]
    Then for all \(k \in \N\) sufficiently large such that \(\epsilon_k<\epsilon\) and \((\alpha_k,\gamma_k)\) has \(|\alpha_k-0|<\epsilon\), we have for any \(n \in \N\) and \(x \in \BP\) that \(E_{\gamma_k}(\alpha_k,\epsilon_k,n,x) \subseteq E(\alpha_k,\epsilon_k,n) \subseteq E(0,2\epsilon,n)\), so
    \[I_{\gamma_k}(\alpha_k) \leq I_{\gamma_k}(\alpha_k,\epsilon_k) \leq I_{0}(0,2\epsilon) \leq I_{0}(0)+\epsilon'. \]
    Hence, \(\varlimsup_{k \rightarrow \infty} I_{\gamma_k}(\alpha_k) \leq \varlimsup_{k \rightarrow \infty} I_{\gamma_k}(\alpha_k,\epsilon_k)  \leq I_\gamma(0)\).

    Now suppose \(\alpha \not=0\). We prove the statement for \(I_{\cdot}(\cdot)\); the argument for \(\widetilde{I}_{\cdot}(\cdot)\) can be proved in a similar way. If \(\gamma=0\), then the argument above holds, so assume \(\gamma>0\). Let \(\epsilon'>0\). We let \(0<\epsilon<\min\{\alpha,\gamma\}\) be such that 
    \[I_{\gamma}(\alpha,\epsilon) \leq I_{\gamma}(\alpha)+\epsilon'.\]
    We let \(0<\delta<\epsilon/2\) be such that \(\epsilon\alpha-(2\alpha+\gamma-2\delta)\delta>0\). We claim that for \(k \in \N\) sufficiently large so that \(|(\alpha_k,\gamma_k)-(\alpha,\gamma)|_{\infty}<\delta\) we have \(E_{\gamma_k}(\alpha_k,\delta,n,x) \subseteq E_{\gamma}(\alpha,\epsilon,n,x)\) for any \(n \in \N\) and \(x \in \BP\). Since \(\delta<\epsilon/2\) we have \(E(\alpha_k,\delta,n) \subseteq E(\alpha,\epsilon,n)\). Moreover, for any \(g \in E_{\gamma_k}(\alpha_k,\delta,n,x)\) we have
    \begin{align*}
        d_{\BP}(\omega_-(g),x) &\leq e^{-n(\gamma_k-\delta)\alpha_k}  \\  &\leq e^{-n(\gamma-2\delta)(\alpha-\delta)} \\
        &=e^{-n((\gamma-\epsilon)\alpha +\epsilon\alpha-(2\alpha+\gamma-2\delta)\delta)} \\
        &\leq e^{-n(\gamma-\epsilon)\alpha}.
    \end{align*}
    Thus, \(E_{\gamma_k}(\alpha_k,\delta,n,x) \subseteq E_{\gamma}(\alpha,\epsilon,n,x)\). Hence, it follows that for all \(k \in \N\) sufficiently large so that \(\epsilon_k \leq \delta\) and \(|(\alpha_k,\gamma_k)-(\alpha,\gamma)|_{\infty}<\delta\), we have
    \[I_{\gamma_k}(\alpha_k) \leq I_{\gamma_k}(\alpha_k,\epsilon_k) \leq I_{\gamma}(\alpha,\epsilon) \leq I_{\gamma}(\alpha)+\epsilon',\]
    which implies 
    \[\varlimsup_{k \rightarrow \infty} I_{\gamma_k}(\alpha_k) \leq  \varlimsup_{k \rightarrow \infty} I_{\gamma_k}(\alpha_k,\epsilon_k) \leq I_{\gamma}(\alpha).\]
\end{proof}

\subsection{Expansion bounds}
In this section we prove some growth and expansion estimates for the maps \(g \in \SL(2,\R)\). An advantage of the sets \(\widetilde{E}_{\gamma}(\alpha,\epsilon,n,x)\) is that they give us more precise growth and contraction/expansion estimates compared to previous work.

The following lemma follows directly from \cite[Lemma 13.2]{BQ16}.
\begin{lemma}\label{lem:BQlemma}
    For every \(g \in \SL(2,\R)\) and \(x = \R u \in \BP\), 
    \[d_{\BP}(\omega_-(g),x) \leq \frac{|gu|}{|g| |u|} \leq d_{\BP}(\omega_-(g),x) +|g|^{-2}.\]
\end{lemma}

As an immediate consequence, we have:
\begin{lemma}\label{lem:essentialinequality}
    For all \(\alpha \geq 0\), \(\epsilon>0\), \(n \in \N\), \(x =\R u\in \BP\), \(0 \leq \gamma \leq 2\), and \(g \in \widetilde{E}_\gamma(\alpha,\epsilon,n, x)\) we have 
    \[e^{n((1-\gamma)\alpha-(1+\alpha)\epsilon)} \leq \frac{|gu|}{|u|} \leq 2 e^{n((1-\gamma)\alpha+(1+\alpha)\epsilon)}.\]
\end{lemma}

To get contraction/expansion estimates, we require the following lemma from \cite{BL85}.

\begin{lemma}[{\cite[Lemma 4.2]{BL85}}]\label{lem:bougerollemma}
    For all \(g \in \SL(2,\R)\) and \(x=\R u,y=\R v \in \BP\), \(x \not=y\),
    \[\frac{d_{\BP}(gx,gy)}{d_{\BP}(x,y)} \leq \frac{|u||v|}{|gu| |gv|}.\]
\end{lemma}

Ignoring error terms and problems arising from \(\alpha\) or \(n\) being too small, the following lemma says that all \(g \in \widetilde{E}_\gamma(\alpha,\epsilon,n,x)\) (i.e. \(d_{\BP}(x, \omega_-(g)) \approx e^{-n\gamma \alpha} \)) has \(d_{\BP}(gx, \upsilon_+(g)) \approx  e^{-n(2-\gamma)\alpha} \). 

\begin{lemma}\label{lem:umapsclosetoupsilon+}
    For all \(\alpha \geq 0\), \(\epsilon>0\), \(n \in \N\), \(x=\R u \in \BP\), \(0 \leq \gamma \leq 2\), and \(g \in \widetilde{E}_\gamma(\alpha,\epsilon,n, x)\) we have 
    \[d_{\BP}(gx,\upsilon_+(g)) \leq e^{-n((2-\gamma)\alpha-(2+\alpha)\epsilon)}.\]
    If \(e^{-n(\gamma-\epsilon) \alpha} \leq \frac{1}{2}\), we further have
    \[d_{\BP}(gx,\upsilon_+(g))  \geq \frac{1}{8} e^{-n((2-\gamma) \alpha +(2+\alpha)\epsilon)}.\]
\end{lemma}

\begin{proof}
Let \(u,v \in S^1\) such that \(x=\R u\) and \(\omega_+(g)=\R v\).
By Lemma \ref{lem:essentialinequality} and \ref{lem:bougerollemma},
    \begin{align*}
        d_{\BP}(gx,\upsilon_+(g)) &\leq  \frac{1}{|gx| |g v|} \\
        &\leq e^{-n((1-\gamma)\alpha-(1+\alpha)\epsilon)} e^{-n(\alpha-\epsilon)} \\
        &=e^{-n((2-\gamma)\alpha-(2+\alpha)\epsilon)},
    \end{align*}
where in the first inequality we are using that \(d_{\BP}(\cdot,\cdot) \leq 1\). 

Using the same lemmas again, we also have 
\begin{align*}
    d_{\BP}(gx,\upsilon_+(g)) \geq d_{\BP}(g^{-1} g x,\omega_+(g)) \frac{|g^{-1} gu| |g^{-1}gv| }{|gu| |gv|} \geq \frac{1}{4}(1-d_{\BP}(x,\omega_-(g)))e^{-n(\alpha+\epsilon)} e^{-n((1-\gamma)\alpha+(1+\alpha)\epsilon)},
\end{align*}
where we have used that \(1=d_{\BP}(\omega_-(g),\omega_+(g)) \leq d_{\BP}(\omega_-(g),x)+d_{\BP}(x,\omega_+(g))\). Hence, the second statement follows.

\end{proof}

The next lemma follows immediately from Lemmas \ref{lem:essentialinequality} and \ref{lem:bougerollemma}. Note that \(g \in \widetilde{E}_{\gamma}(\alpha,\epsilon,n,x) \cap  \widetilde{E}_{\gamma}(\alpha,\epsilon,n,y)\) implies either \(x \approx y\) or \(d_{\BP}(x,y)\) is of the order approximately \(e^{-n\gamma\alpha}\). The bound holds in both cases.

\begin{lemma}\label{lem:expansionbound}
For all  \(\alpha \geq 0\), \(\epsilon>0\), \(\gamma \in [0,2]\), \(x,y \in \BP\) with \(x \not=y\), \(n \in \N\), and \(g \in \widetilde{E}_{\gamma}(\alpha,\epsilon,n,x) \cap  \widetilde{E}_{\gamma}(\alpha,\epsilon,n,y)\),
\[ \frac{d_{\BP}(gx,gy)}{d_{\BP}(x,y)} \leq 2e^{-n(2(1-\gamma)\alpha-2(1+\alpha)\epsilon)} .\]
\end{lemma}

We also have the following contraction/expansion bounds in the case \(g \in \widetilde{E}_{\gamma}(\alpha,\epsilon,n,x) \cap  \widetilde{E}_{\delta}(\alpha,\epsilon,n,y)\) where \(\delta \not= \gamma\).

\begin{lemma}\label{lem:distortionbound1}
    Let \(\epsilon>0\), \(\alpha \geq \epsilon\), and \(n \in \N\) large enough so that \(1-e^{-n\epsilon^2}>2^{-1}\). Suppose that \(g \in  \widetilde{E}_{\gamma}(\alpha,\epsilon,n,x) \cap \widetilde{E}_{\delta}(\alpha,\epsilon,n,y)\) where \(\gamma,\delta \in \Gamma(\epsilon)\) satisfy \(\gamma-\delta \geq 3\epsilon\). Then,
    \[\frac{d_{\BP}(gx,gy)}{d_{\BP}(x,y)} \leq 4e^{-n(2(1-\gamma)\alpha-2\epsilon)}.\]
\end{lemma}

\begin{proof}
    By the triangle inequality and Lemma \ref{lem:umapsclosetoupsilon+},
\begin{align*}
    d_{\BP}\left(gx,gy\right) &\leq d_{\BP} \left(gx,\upsilon_+(g)\right) +d_{\BP} \left(\upsilon_+(g), gy \right) \\
    &\leq e^{-n((2-\gamma)\alpha-(2+\alpha)\epsilon)}+e^{-n((2-\delta)\alpha-(2+\alpha)\epsilon)} \\
    &\leq  2e^{-n((2-\gamma)\alpha-(2+\alpha)\epsilon)}.
\end{align*}
Since we are assuming \(n \in \N\) is large enough so that 
\(1-e^{-n \epsilon^2}>2^{-1}\), we also have
\begin{align*}
    d_{\BP}(x,y) &\geq e^{-n(\delta+\epsilon) \alpha }-e^{-n(\gamma -\epsilon) \alpha } \\
    &\geq  e^{-n(\gamma -2\epsilon)\alpha}-e^{-n(\gamma-\epsilon) \alpha } \\
    &= (e^{-n(\gamma-2\epsilon)\alpha })(1-e^{-n\epsilon\alpha}) \\
    &\geq 2^{-1} e^{-n(\gamma-2\epsilon )\alpha}.
\end{align*}
Thus, the lemma follows.
\end{proof}

\subsection{The functions \(k(t)\), \(k^+(t)\), and \(k^-(t)\)}\label{subsec:ktk+tk-t}
Recall for \(t \in \R\) we defined
\[ k(t)= \lim_{n \rightarrow \infty} \frac{1}{n} \log \int |g|^t \diff \mu^n(g),\]
\[k^+(t):= \varlimsup_{n \rightarrow \infty}  \frac{1}{n} \log \left( \sup_{x \in \BP} \int |g x|^t \diff \mu^n(g)  \right),  \]
and
\[k^-(t):= \varliminf_{n \rightarrow \infty}  \frac{1}{n} \log \left( \inf_{x \in \BP} \int |g x|^t \diff \mu^n(g)  \right) .\]
In this section we derive some relations between the functions \(k(t),k^+(t),k^-(t)\) and the previously defined \(I_{\gamma}(\alpha)\). The proofs in this section are elementary; we do not require any irreducibility assumptions whatsoever. For completeness we also consider the case where \(t>0\), but in this paper we are generally only concerned with the \(t \leq 0\) case (though, we expect our theorems can be extended to the \(t>0\) case).
 
Note the limit in \(k(t)\) exists by subadditivity when \(t \geq 0\) and superadditivity when \(t<0\). Trivially, when \(t \geq 0\) we have \(k^-(t) \leq k^+(t) \leq k(t)\) and for \(t<0\) we have \(k(t) \leq k^-(t) \leq k^+(t)\). 

The following lemma is easy to prove and will be used regularly.

\begin{lemma}\label{lem:limsupmaxismaxlimsup}
    Let \(F\) be a finite set, and for each \(i \in F\) let \((a_n^{(i)})_{n \in \N}\) be a sequence of strictly positive real numbers. Then, 
    \[\varlimsup_{n \rightarrow \infty} \frac{1}{n} \log \sum_{i \in F} a_n^{(i)}= \max_{i \in F} \varlimsup_{n \rightarrow \infty} \frac{1}{n} \log a_n^{(i)}. \]
\end{lemma}

The following proposition is well known. We include a proof here for the convenience of the reader. 

\begin{prop}\label{prop:PtopissupMCE}
    For all \(t \in \R\), \(k(t)=\sup_{\alpha \in \mathrm{A}} \left\{ I_0(\alpha)+t\alpha \right\}.\)
\end{prop}

\begin{proof}[Proof of Proposition \ref{prop:PtopissupMCE}]
The equality \(`\geq'\) is straightforward, so we show the other inequality. Assume \(t<0\); the proof when \(t \geq 0\) follows similarly. Let \(\epsilon_k>0\) be a sequence converging to 0. For any \(n \in \N\) we have
\begin{align*}
     \frac{1}{n} \log \int |g|^t \diff \mu^n(g) &\leq \frac{1}{n} \log \sum_{\alpha \in \mathrm{A}(\epsilon_k)} \int_{E(\alpha,\epsilon_k,n)} |g|^t \diff \mu^n(g) \\
     &\leq  \frac{1}{n} \log \sum_{\alpha \in \mathrm{A}(\epsilon_k)}  e^{n I_0(\alpha,\epsilon_k,n)} e^{nt(\alpha-\epsilon_k)}.
\end{align*}
By Lemma \ref{lem:limsupmaxismaxlimsup}, for each \(k \in \N\) there exists \(\alpha_k \in \mathrm{A}(\epsilon_k)\) such that 
\[\varlimsup_{n \rightarrow \infty} \frac{1}{n} \log \int |g|^t \diff \mu^n(g) \leq I_0(\alpha_k,\epsilon_k)+t(\alpha_k-\epsilon_k).\]
Letting \(\alpha \geq 0\) be a limit point of the sequence \((\alpha_k)_{k \in \N}\). By Lemma \ref{lem:USC},
\[\varlimsup_{n \rightarrow \infty} \frac{1}{n} \log \int |g|^t \diff \mu^n(g) \leq I_0(\alpha)+t\alpha,\]
which further implies that \(\alpha \in \mathrm{A}\).
\end{proof}

For \(\alpha \geq 0\) we define
\[k_\alpha^+(t):=\lim_{\epsilon \rightarrow 0} \varlimsup_{n \rightarrow \infty} \frac{1}{n} \log \left( \sup_{x \in \BP} \int_{E(\alpha,\epsilon,n)} |g x|^t \diff \mu^n(g)  \right).  \]
We note that it is easy to see that \(k_0^+(t) \equiv I_0(0)\) and \(k_{\alpha}^+(t)=-\infty\) for \(\alpha \not \in \mathrm{A}\).

\begin{lemma}\label{lem:t<0,Palpha(t)<=}
    For all \(t \in \R\) and \(\alpha \in \mathrm{A} \setminus\{ 0\} \), \(k_{\alpha}^+(t) \leq \sup_{\gamma \in [0,2]} \{I_{\gamma}(\alpha)+t(1-\gamma) \alpha \}.\)
\end{lemma}

\begin{proof}
Assume \(t \leq 0\); the proof when \(t > 0\) follows similarly. Let \(\epsilon_k>0\) be a sequence converging to 0. By Lemmas \ref{lem:gammaiscover} and \ref{lem:essentialinequality}, for all \(n \in \N\), 
\begin{align*}
        \sup_{x \in \BP} \int_{E(\alpha,\epsilon,n)} |gx|^t \diff \mu^n(g) & \leq \sup_{x \in \BP} \sum_{\gamma \in \Gamma(\epsilon_k)} \int_{\widetilde{E}_{\gamma}(\alpha,\epsilon_k, n, x)} |gx|^t \diff \mu^n(g) \\
        &\leq  \sum_{\gamma \in \Gamma(\epsilon_k)}  \sup_{x \in \BP}   \int_{\widetilde{E}_{\gamma}(\alpha, \epsilon_k, n,x)} |gx|^t \diff \mu^n(g) \\
        &\leq  \sum_{\gamma \in \Gamma(\epsilon_k)}   e^{n \widetilde{I}_{\gamma}(\alpha,\epsilon_k,n)} e^{n(t(1-\gamma)\alpha-t(1+\alpha)\epsilon_k)} \\
        &\leq \sum_{\gamma \in \Gamma(\epsilon_k)}   e^{n I_{\gamma}(\alpha,\epsilon_k,n)} e^{n(t(1-\gamma)\alpha-t(1+\alpha)\epsilon_k)}.
\end{align*}
Hence, by Lemma \ref{lem:limsupmaxismaxlimsup}, for each \(k \in \N\) there exists \(\gamma_k \in \Gamma(\epsilon_k)\) such that 
\begin{equation}\label{eqn:inproofoPalphat<=}
    k_{\alpha}^+(t) \leq I_{\gamma_k}(\alpha,\epsilon_k)+t(1-\gamma_k)\alpha-t(1+\alpha)\epsilon_k.
\end{equation}
The lemma then follows by Lemma \ref{lem:USC}.
\end{proof}

\begin{lemma}\label{lem:t<0,Palpha(t)>=}
    For \(\alpha \in \mathrm{A} \setminus \{0\}\) and all \(t \in \R\),
    \[k^+_\alpha(t) \geq \sup_{\gamma \in [0,2]} \{ \widetilde{I}_{\gamma}(\alpha)+t(1-\gamma) \alpha \}.\]
\end{lemma}

\begin{proof}
We consider the case where \(t<0\); the case when \(t \geq 0\) can be proved similarly. Fix \(\alpha \in \mathrm{A} \setminus \{0\}\) and \(0 \leq \gamma \leq 2\). Let \(\epsilon>0\). By Lemma \ref{lem:essentialinequality}, for all \(n \in \N\),
\begin{align*}
    \widetilde{I}_{\gamma}(\alpha,\epsilon,n)+ t (1-\gamma)\alpha
    &=\frac{1}{n} \log \left( \sup_{x \in \BP} \int_{ \widetilde{E}_{\gamma}(\alpha,\epsilon,n, x)} e^{tn (1-\gamma)\alpha} \diff \mu^n(g)  \right) \\
    &\leq \frac{1}{n} \log  \left( \sup_{x \in \BP} \int_{ \widetilde{E}_{\gamma}(\alpha,\epsilon,n, x)} |gx|^t \diff \mu^n(g) \right) -\frac{t}{n} \log 2- t(1+\alpha)\epsilon \\
    &\leq \frac{1}{n} \log  \left( \sup_{x \in \BP} \int_{E(\alpha,\epsilon,n)} |gx|^t \diff \mu^n(g) \right) -\frac{t}{n} \log 2- t(1+\alpha)\epsilon.
\end{align*}
Hence, taking the limsup as \(n \rightarrow \infty\), then letting \(\epsilon \rightarrow 0 \) we have 
\[ \widetilde{I}_{\gamma}(\alpha)+t (1-\gamma)\alpha  \leq k_{\alpha}^+(t).\]
\end{proof}

\begin{lemma}\label{lem:QsupequalstildeQsup}
For all \(t \leq 0\) and \(\alpha \in \mathrm{A} \setminus \{0\}\),
     \[\sup_{\gamma \in [0,2]} \{ \widetilde{I}_{\gamma}(\alpha)+t(1-\gamma) \alpha \} = \sup_{\gamma \in [0,2]} \{I_{\gamma}(\alpha)+t(1-\gamma) \alpha \}.\]
\end{lemma}

\begin{proof}
The inequality `\(\leq\)' is clear. To prove the other inequality, let \(\epsilon'>0\) and let \(\gamma \in [0,2]\) be such that
\begin{equation}\label{eqn:inprooftakegammaclosetosurpremum}
     I_{\gamma}(\alpha)+t(1-\gamma) \alpha \geq \sup_{ \gamma' \in [0,2]} \{I_{\gamma'}(\alpha)+t(1-\gamma') \alpha \}-\epsilon'.
 \end{equation}
 
Let \(\epsilon_k>0\) be a sequence converging to 0. We let 
\[\Gamma_{\geq \gamma}(\epsilon_k):=\{\gamma, 2\} \cup \left\{j \epsilon_k :j \in \N \cup \{0\}, \gamma < j \epsilon_k  < 2 \right\}.\]
By an argument similar to the one in Lemma \ref{lem:gammaiscover}, for each \(x \in \BP\) we have \(E_\gamma(\alpha,\epsilon_k,n,x)=\cup_{\gamma' \in \Gamma_{\geq \gamma}(\epsilon_k)} \widetilde{E}_{\gamma'}(\alpha,\epsilon_k,n,x)\), so
\begin{align*}
    I_{\gamma}(\alpha,\epsilon_k) &= \varlimsup_{n \rightarrow \infty} \frac{1}{n} \log \sup_{x \in \BP} \mu^n( E_\gamma(\alpha,\epsilon_k,n,x)) \\
    &\leq  \varlimsup_{n \rightarrow \infty} \frac{1}{n} \log \sup_{x \in \BP} \sum_{\gamma' \in \Gamma_{\geq \gamma}(\epsilon_k)} \mu^n(\widetilde{E}_{\gamma'}(\alpha,\epsilon_k,n,x)).
\end{align*}
By Lemma \ref{lem:limsupmaxismaxlimsup}, for each \(k \in \N\) there exists \(\gamma_k \in \Gamma_{\geq \gamma}(\epsilon_k)\) such that 
\begin{align}
    I_{\gamma}(\alpha,\epsilon_k) &\leq \widetilde{I}_{\gamma_k}(\alpha,\epsilon_k). \label{eqn:inproofQgammalessthanwidetildeQgammak}
\end{align}
Let \(\gamma_{\infty} \in [\gamma,2]\) be a limit point of \((\gamma_k)_{k \in \N}\). Without loss of generality, we assume that \(\gamma_k \rightarrow \gamma_{\infty}\). Using (\ref{eqn:inproofQgammalessthanwidetildeQgammak}), we have
\begin{align*}
     \sup_{\gamma' \in [0,2]} \{I_{\gamma'}(\alpha)+t(1-\gamma') \alpha \} &\leq I_{\gamma}(\alpha)+t(1-\gamma) \alpha +\epsilon' \\
     &\leq \varlimsup_{k \rightarrow \infty} \widetilde{I}_{\gamma_k}(\alpha,\epsilon_k)+t(1-\gamma) \alpha +\epsilon' \\
     &\leq \widetilde{I}_{\gamma_{\infty}}(\alpha)+t(1-\gamma_{\infty})\alpha+\epsilon',
\end{align*}
where the final inequality follows from Lemma \ref{lem:USC} and the fact that \(\gamma_{\infty} \geq \gamma\) (using that \(t \leq 0\)).
\end{proof}

\begin{lemma}\label{lem:oPalpha}
For \(\alpha \in \mathrm{A}\) and for all \(t \in \R\),
    \[k^+_\alpha(t) = \sup_{\gamma \in [0,2]} \{ I_{\gamma}(\alpha)+t(1-\gamma) \alpha \}.\]
In particular, \(t \mapsto k_{\alpha}^+(t)\) is convex on \(\R\). Moreover, \(k_{\alpha}^+(t)=I_0(\alpha)+t\alpha\) for all \(t\geq 0\) and \(k_{\alpha}^+(t)=I_2(\alpha)-t \alpha\) for all \(t \leq -1\).
\end{lemma}

\begin{proof}

Note that for all \(t \in \R\) we trivially have \(k_0^+(t) \equiv I_0(0) \equiv I_{\gamma}(0)\) for any \(\gamma \in [0,2]\). Assume \(\alpha>0\). For \(t \leq 0\), the equality 
\[k_{\alpha}^+(t) = \sup_{\gamma \in [0,2]} \{ I_{\gamma}(\alpha)+t(1-\gamma) \alpha \}\]
follows immediately from Lemmas \ref{lem:t<0,Palpha(t)<=}, \ref{lem:t<0,Palpha(t)>=}, and \ref{lem:QsupequalstildeQsup}. 

We now show that \(k_{\alpha}^+(t) =I_2(\alpha)-t\alpha \) for \(t \leq -1\). By Lemma \ref{lem:elementaryQgammainequality} we have that for all \(\gamma \in [0,2]\),
\[I_{\gamma}(\alpha) \leq I_{2}(\alpha)+(2-\gamma)\alpha.\]
So, when \(t \leq -1\),
\[I_{\gamma}(\alpha) +t(1-\gamma)\alpha \leq I_{2}(\alpha)-t\alpha+(1+t)(2-\gamma)\alpha \leq I_{2}(\alpha)-t\alpha.\]
In particular,
\[k_{\alpha}^+(t)=\sup_{\gamma \in [0,2]} \{ I_{\gamma}(\alpha)+t(1-\gamma) \alpha \} \leq I_{2}(\alpha)-t\alpha,\]
so we must have equality.

For \(t \geq 0\) we have \(|gx|^t \leq |g|^t\) which implies that  \(k_{\alpha}^+(t) \leq I_{0}(\alpha)+ t\alpha\). The other inequality can be shown by taking any \(x\) and using Lemma \ref{lem:BQlemma} to show that either \(x\) or \(x^{\perp}\) must have 
\[\varlimsup_{n \rightarrow \infty} \frac{1}{n}\log \int_{E(\alpha,\epsilon,n)} |g \cdot|^t \diff \mu^n(g)=I_0(\alpha)+t\alpha.\]
In particular,
\[ k_{\alpha}^+(t) = I_{0}(\alpha)+ t\alpha= \sup_{\gamma \in [0,2]} \{ I_{\gamma}(\alpha)+t(1-\gamma)\alpha\}.\]
where the last equality is since \(t \geq 0\) and \(I_{\gamma}(\alpha) \leq I_0(\alpha)\) for any \(\gamma \in [0,2]\).
\end{proof}

\begin{lemma}\label{lem:oPtissupalphaoPalphat}
For all \(t \in \R\), \(k^+(t) =\sup_{\alpha \in \mathrm{A}} k_{\alpha}^+(t).\)
\end{lemma}

\begin{proof}
    The inequality \(k^+(t) \geq \sup_{\alpha \in \mathrm{A}} k_{\alpha}^+(t)\) is clear. For the other inequality, assume \(t<0\); the proof when \(t \geq 0\) follows similarly. Let \(\epsilon_k>0\) be a sequence converging to 0. For any \(n \in \N\) we have
\begin{align*}
     \frac{1}{n} \log \sup_{x \in \BP} \int |gx|^t \diff \mu^n(g) &\leq \frac{1}{n} \log \sum_{\alpha \in \mathrm{A}(\epsilon_k)} \sup_{x \in \BP} \int_{E(\alpha,\epsilon_k,n)} |gx|^t \diff \mu^n(g).
\end{align*}
By Lemma \ref{lem:limsupmaxismaxlimsup}, for each \(k \in \N\) there exists \(\alpha_k \in \mathrm{A}(\epsilon_k)\) such that 
\[k^+(t)= \varlimsup_{n \rightarrow \infty} \frac{1}{n} \log \sup_{x \in \BP} \int |gx|^t \diff \mu^n(g) \leq \varlimsup_{n \rightarrow \infty} \frac{1}{n} \log  \sup_{x \in \BP} \int_{E(\alpha,\epsilon_k,n)} |gx|^t \diff \mu^n(g).\] 
Let \(\alpha\) be a limit point of \((\alpha_k)_{k \in \N}\). Without loss of generality we assume that \(\alpha_k \rightarrow \alpha\).

Let \(\epsilon'>0\) and let \(\epsilon>0\) be such that 
\[ \varlimsup_{n \rightarrow \infty} \frac{1}{n} \log \sum_{\alpha \in \mathrm{A}(\epsilon)} \sup_{x \in \BP} \int_{E(\alpha,2\epsilon,n)} |gx|^t \diff \mu^n(g) \leq k^+_{\alpha}(t)+\epsilon'.\]
For \(K \in \N\) large enough such that \(\epsilon_K<\epsilon\), we have \(E(\alpha_K,\epsilon_K,n) \subseteq E(\alpha,2\epsilon,n)\). Thus,
\begin{align*}
    k^+(t) &\leq \varlimsup_{n \rightarrow \infty} \frac{1}{n} \log \sup_{x \in \BP} \int_{E(\alpha_K,\epsilon_K,n)} |gx|^t \diff \mu^n(g) \\
    &\leq \varlimsup_{n \rightarrow \infty} \frac{1}{n} \log \sup_{x \in \BP} \int_{E(\alpha,2\epsilon,n)} |gx|^t \diff \mu^n(g) \\
    &\leq k^+_\alpha(t)+\epsilon'. 
\end{align*}
Note, since \(k^+(t)>-\infty\), we must have \(\alpha \in \mathrm{A}\).
\end{proof}

The following proposition follows immediately from Lemma \ref{lem:oPalpha} and \ref{lem:oPtissupalphaoPalphat}.

\begin{prop}\label{prop:oP}
For all \(t \in \R\),
\begin{align*}
    k^+(t) =\sup_{\alpha \in \mathrm{A}} \sup_{\gamma \in [0,2]} \{ I_{\gamma}(\alpha)+t(1-\gamma) \alpha \}
\end{align*}
In particular, \(t \mapsto k^+(t)\) is convex on \(\R\). Moreover, \(k^+(t)=\sup_{\alpha \in \mathrm{A}}\{ I_0(\alpha)+t \alpha \}=k(t)\) for all \(t \geq 0\) and \(k^+(t)=\sup_{\alpha \in \mathrm{A}}\{ I_2(\alpha)-t \alpha \}\) for all \(t \leq -1\).
\end{prop} 

\begin{prop}\label{prop:uP}
For all \(t \in (-1,0]\), the limit in \(k^-(t)\) exists and is equal to \(k(t)\).
\end{prop}

\begin{proof}
For \(t \in (-1,0]\) and any \(x \in \BP\), let 
    \[C_{\lambda,t}:=\int d_{\BP}(x,y)^t \diff \lambda(y).\]
    We note that \(C_{\lambda,t}\) does not depend on the choice of \(x\) and is finite since \(t>-1\). By Lemma \ref{lem:BQlemma}, for any \(g \in \SL(2,\R)\),
    \[ \int |g x|^t \diff \lambda(x) \leq |g|^t \int d_{\BP}(\omega_-(g),x)^t \diff \lambda(x) = C_{\lambda,t} |g|^t. \]
    Hence, by Tonelli's theorem,
    \begin{align*}
        k(t) &\leq \varliminf_{n \rightarrow \infty} \frac{1}{n} \log \left( \inf_{x \in \BP} \int |g x|^t \diff \mu^n(g)  \right) \\
        &\leq \varlimsup_{n \rightarrow \infty} \frac{1}{n} \log \left( \inf_{x \in \BP} \int |g x|^t \diff \mu^n(g)  \right) \\
        &\leq \varlimsup_{n \rightarrow \infty} \frac{1}{n} \log \left( \int \int |g x|^t \diff \mu^n(g) \diff \lambda(x) \right) \\
        &= \varlimsup_{n \rightarrow \infty} \frac{1}{n} \log \left( \int \int |g x|^t \diff \lambda(x)  \diff \mu^n(g)  \right) \\
        &\leq \varlimsup_{n \rightarrow \infty} \frac{1}{n} \log \left(C_{\lambda,t} \int |g|^t \diff \mu^n(g) \right) \\
        &= k(t),
    \end{align*}
    which implies the proposition.
\end{proof}

\section{The spectral gap}\label{section:spectralgap}
For \(0 < \zeta \leq 1\), let \(C^{\zeta}(\BP)\) denote the set of real-valued \(\zeta\)-H\"older continuous maps \(f: \BP \rightarrow \R\). We endow \(C^{\zeta}(\BP)\) with the usual norm: for \(f \in C^{\zeta}(\BP)\) let
\[ [f]_{\zeta}:=\sup_{\substack{x,y \in \BP \\ x \not= y}} \frac{|f(x)-f(y)|}{d_{\BP}(x,y)^\zeta},\]
then the norm is defined by
\[|f|_{\zeta}:=[f]_{\zeta}+|f|_{\infty}.\]

The following proposition, combined with \cite[Proposition 2.5]{BL85}, are the central ingredients in showing that \(P_t: C^{\zeta}(\BP) \rightarrow C^{\zeta}(\BP)\) is quasi-compact for \(t \in (t_c',0]\) and \(\zeta \in (0,\zeta_t')\). 

\begin{prop}\label{prop:zetabougerol}
For all \(t \leq 0\) and \(\zeta>0\),
\[ \lim_{n \rightarrow \infty} \frac{1}{n} \log \sup_{\substack{x,y\in \BP \\ x \not=y}} \int |gx|^t \left(\frac{d_{\BP}(g x,g y)}{d_{\BP}(x,y)} \right)^{\zeta} \diff \mu^n(g) = \sup_{\alpha \in \mathrm{A}} \sup_{\gamma \in [0,2]} \{I_\gamma(\alpha)+t(1-\gamma)\alpha -2\zeta (1-\gamma) \alpha\}.\]
\end{prop}

In the following proof, by \(O(\epsilon)\) we mean a term of the form \(C\epsilon\), where \(C>0\) is only allowed to depend on \(\overline{\alpha}, |t|,\) and \(\zeta\).

\begin{proof}[Proof of the `\(\leq\)' inequality]
Let \(\epsilon>0\). For all \(n \in \N\), by Lemmas \ref{lem:gammaiscover} and \ref{lem:essentialinequality},
    \begin{align*}
        \int  |gx|^t \left( \frac{d_{\BP}(gx,gy )}{d_{\BP}(x,y)} \right)^{\zeta} \diff \mu^n(g)\\ 
        \leq \int_{E(0,\epsilon,n)} e^{-n t\epsilon} \left( \frac{d_{\BP}(gx,gy )}{d_{\BP}(x,y)} \right)^{\zeta} \diff \mu^n(g) \\
        +\sum_{\alpha \in \mathrm{A}(\epsilon) \setminus \{0\} } \sum_{\substack{\gamma,\delta \in \Gamma(\epsilon), \\ |\gamma-\delta| \leq 3\epsilon} } \int_{\widetilde{E}_{\gamma}(\alpha,\epsilon,n, x) \cap \widetilde{E}_{\delta}(\alpha,\epsilon,n, y)}  e^{n(t(1-\gamma)\alpha+O(\epsilon))}  \left( \frac{d_{\BP}(gx,gy )}{d_{\BP}(x,y)} \right)^{\zeta} \diff \mu^n(g) \\
        +\sum_{\alpha \in \mathrm{A}(\epsilon)  \setminus \{0\}} \sum_{\substack{\gamma,\delta \in \Gamma(\epsilon), \\ \gamma-\delta>3\epsilon}} \int_{\widetilde{E}_{\gamma}(\alpha,\epsilon,n, x) \cap \widetilde{E}_{\delta}(\alpha,\epsilon,n, y)}  e^{n(t(1-\gamma)\alpha+O(\epsilon))}  \left( \frac{d_{\BP}(gx,gy )}{d_{\BP}(x,y)} \right)^{\zeta} \diff \mu^n(g) \\
        +\sum_{\alpha \in \mathrm{A}(\epsilon)  \setminus \{0\}} \sum_{\substack{\gamma,\delta \in \Gamma(\epsilon), \\  \delta-\gamma>3\epsilon}} \int_{\widetilde{E}_{\gamma}(\alpha,\epsilon,n, x) \cap \widetilde{E}_{\delta}(\alpha,\epsilon,n, y)}  e^{n(t(1-\gamma)\alpha+O(\epsilon))} \left( \frac{d_{\BP}(gx,gy )}{d_{\BP}(x,y)} \right)^{\zeta} \diff \mu^n(g).
    \end{align*}
By Lemma \ref{lem:bougerollemma} the first term is bounded by \(2e^{n(I_0(0,\epsilon,n)+O(\epsilon))}.\)

Now fix \(\alpha \in \mathrm{A}(\epsilon)\setminus \{0\} \). Suppose that \(g \in  \widetilde{E}_{\gamma}(\alpha,\epsilon,n,x) \cap \widetilde{E}_{\delta}(\alpha,\epsilon,n,y)\) with \(|\gamma-\delta| \leq 3\epsilon\). Observe that
\[g \in  \widetilde{E}_{\gamma}(\alpha,4\epsilon,n,x) \cap \widetilde{E}_{\delta}(\alpha,4\epsilon,n,y).\]
By Lemma \ref{lem:expansionbound} we have 
 \[\frac{d_{\BP}(gx,gy)}{d_{\BP}(x,y)} \leq 2e^{-n(2(1-\gamma)\alpha-O(\epsilon))},\]
so
\begin{align*}
    \int_{\widetilde{E}_{\gamma}(\alpha,\epsilon,n,x) \cap \widetilde{E}_{\delta}(\alpha,\epsilon,n,y)} e^{n (t(1-\gamma)\alpha+O(\epsilon))} \left( \frac{d_{\BP}(gx,gy )}{d_{\BP}(x,y)} \right)^{\zeta} \diff \mu^n(g) \\
    \leq 2^{\zeta} \mu^n( \widetilde{E}_{\gamma} (\alpha,4\epsilon,n,x)) e^{n((t-2\zeta)(1-\gamma)\alpha+O(\epsilon))}.
\end{align*}
Therefore,
\begin{align*}
    \sum_{\alpha \in \mathrm{A}(\epsilon)} \sum_{\substack{\gamma,\delta \in \Gamma(\epsilon), \\  |\gamma-\delta| \leq 3\epsilon}} \int_{\widetilde{E}_{\gamma}(\alpha,\epsilon,n, x) \cap \widetilde{E}_{\delta}(\alpha,\epsilon,n, y)}  e^{n(t(1-\gamma)\alpha+O(\epsilon))} \left( \frac{d_{\BP}(gx,gy )}{d_{\BP}(x,y)} \right)^{\zeta} \diff \mu^n(g) \\
    \leq 2^{\zeta} \left(\frac{2}{\epsilon} +2 \right) \sum_{\alpha \in \mathrm{A}(\epsilon)} \sum_{\substack{\gamma \in \Gamma(\epsilon)}} e^{n I_{\gamma}(\alpha,4\epsilon,n)} e^{n ((t-2\zeta)(1-\gamma)\alpha+O(\epsilon))}.
\end{align*}
We note the \((\frac{2}{\epsilon}+2)\) term is appearing because \(|\Gamma(\epsilon)| \leq (\frac{2}{\epsilon}+2)\).

Now let \(g \in  \widetilde{E}_{\gamma}(\alpha,\epsilon,n,x) \cap \widetilde{E}_{\delta}(\alpha,\epsilon,n,y)\) with \(\gamma-\delta>3\epsilon\). We assume \(n \in \N\) large enough so that \(1-e^{-n\epsilon^2}>2^{-1}\), so that by Lemma \ref{lem:distortionbound1},
\begin{align*}
    \int_{\widetilde{E}_{\gamma}(\alpha,\epsilon,n,x) \cap \widetilde{E}_{\delta}(\alpha,\epsilon,n,y)}   e^{n (t(1-\gamma)\alpha+O(\epsilon))} \left( \frac{d_{\BP}(gx,gy )}{d_{\BP}(x,y)} \right)^{\zeta} \diff \mu^n(g) \\
    \leq 4^{\zeta} \mu^n(\widetilde{E}_{\gamma}(\alpha,\epsilon,n,x)) e^{n ((t-2\zeta)(1-\gamma)\alpha+O(\epsilon))}.
    \end{align*}
Thus,
\begin{align*}
    \sum_{\alpha \in \mathrm{A}(\epsilon)} \sum_{\substack{\gamma,\delta \in \Gamma(\epsilon), \\ \gamma-\delta>3\epsilon}} \int_{\widetilde{E}_{\gamma}(\alpha,\epsilon,n, x) \cap \widetilde{E}_{\delta}(\alpha,\epsilon,n, y)}  e^{n(t(1-\gamma)\alpha+O(\epsilon))}  \left( \frac{d_{\BP}(gx,gy )}{d_{\BP}(x,y)} \right)^{\zeta} \diff \mu^n(g) \\
    \leq  4^{\zeta} \left(\frac{2}{\epsilon} +2 \right) \sum_{\alpha \in \mathrm{A}(\epsilon)} \sum_{\substack{\gamma \in \Gamma(\epsilon)}} e^{n I_{\gamma}(\alpha,\epsilon,n)} e^{n ((t-2\zeta)(1-\gamma)\alpha+O(\epsilon))}
\end{align*}

Similarly, for \(g \in  \widetilde{E}_{\gamma}(\alpha,\epsilon,n,x) \cap \widetilde{E}_{\delta}(\alpha,\epsilon,n,y)\) with \(\delta-\gamma>3\epsilon\),
\begin{align*}
    \int_{ \widetilde{E}_{\gamma}(\alpha,\epsilon,n,x) \cap \widetilde{E}_{\delta}(\alpha,\epsilon,n,y)}   e^{n (t(1-\gamma)\alpha+O(\epsilon))} \left( \frac{d_{\BP}(gx,gy )}{d_{\BP}(x,y)} \right)^{\zeta} \diff \mu^n(g) \\
    \leq 4^{\zeta} \int_{ \widetilde{E}_{\gamma}(\alpha,\epsilon,n,x) \cap \widetilde{E}_{\delta}(\alpha,\epsilon,n,y)}   e^{n (t(1-\gamma)\alpha-2\zeta(1-\delta)\alpha +O(\epsilon))} \diff \mu^n(g) 
    \\
    \leq 4^{\zeta} \int_{ \widetilde{E}_{\gamma}(\alpha,\epsilon,n,x) \cap \widetilde{E}_{\delta}(\alpha,\epsilon,n,y)}   e^{n (t(1-\delta)\alpha-2\zeta(1-\delta)\alpha +O(\epsilon))} \diff \mu^n(g) \\
    \leq 4^{\zeta} \mu^n( \widetilde{E}_{\delta}(\alpha,\epsilon,n,y)) e^{n( (t-2\zeta)(1-\delta)\alpha+O(\epsilon))},
\end{align*}
where in the second inequality we are using that \(\gamma<\delta\) implies that \(t(1-\gamma) \leq t(1-\delta)\) since \(t \leq 0\). Hence, we have
\begin{align*}
    \sum_{\alpha \in \mathrm{A}(\epsilon)} \sum_{\substack{\gamma,\delta \in \Gamma(\epsilon), \\  \delta-\gamma>3\epsilon}} \int_{\widetilde{E}_{\gamma}(\alpha,\epsilon,n, x) \cap \widetilde{E}_{\delta}(\alpha,\epsilon,n, y)}  e^{n(t(1-\gamma)\alpha+O(\epsilon))} \left( \frac{d_{\BP}(gx,gy )}{d_{\BP}(x,y)} \right)^{\zeta} \diff \mu^n(g) \\
    \leq 4^{\zeta} \left(\frac{2}{\epsilon} +2 \right) \sum_{\alpha \in \mathrm{A}(\epsilon)} \sum_{\substack{\delta \in \Gamma(\epsilon)}} e^{n I_{\delta}(\alpha,\epsilon, n)}  e^{n ((t-2\zeta)(1-\delta)\alpha+O(\epsilon))}.
\end{align*}

In particular, we have shown that 
\begin{align*}
    \int  |gx|^t \left( \frac{d_{\BP}(gx,gy )}{d_{\BP}(x,y)} \right)^{\zeta} \diff \mu^n(g) \\
     \leq  4 \cdot 4^{\zeta} \left(\frac{2}{\epsilon} +2 \right) \sum_{\alpha \in \mathrm{A}(\epsilon)} \sum_{\substack{\gamma \in \Gamma(\epsilon)}} e^{n I_{\gamma}(\alpha,4\epsilon, n)}  e^{n ((t-2\zeta)(1-\gamma)\alpha+O(\epsilon))},
\end{align*}
Using Lemma \ref{lem:limsupmaxismaxlimsup}, this implies that
\begin{align*}
     \varlimsup_{n \rightarrow \infty} \frac{1}{n} \log \int  |gx|^t \left( \frac{d_{\BP}(gx,gy )}{d_{\BP}(x,y)} \right)^{\zeta} \diff \mu^n(g) \leq \max_{\alpha \in \mathrm{A}(\epsilon)} \max_{\gamma \in \Gamma(\epsilon)} \{ I_{\gamma}(\alpha,4\epsilon) +(t-2\zeta) (1-\gamma)\alpha +O(\epsilon)\}.
\end{align*}
Thus, by an application of Lemma \ref{lem:USC}, 
\[\lim_{\epsilon \rightarrow 0} \varlimsup_{n \rightarrow \infty}\frac{1}{n} \log \sup_{\substack{x,y \in \BP \\ x \not= y }} \int  |gx|^t \left( \frac{d_{\BP}(gx,gy )}{d_{\BP}(x,y)} \right)^{\zeta} \diff \mu^n(g) \leq \sup_{\alpha \in \mathrm{A}} \sup_{\gamma \in [0,2]} \{I_{\gamma}(\alpha)+(t-2\zeta) (1-\gamma)\alpha\}.\]
\end{proof}

\begin{proof}[Proof of the `\(\geq\)' inequality]
First suppose \(\alpha>0\) and \(\gamma \in (0, 2)\). Let \(\epsilon'>0\) and let \(0<\epsilon<\min\{\frac{\gamma}{3}, \frac{2-\gamma}{3(2+\alpha)} \}\). For each \(n \in \N\), let \(x_n \in \BP\) be such that
    \[\frac{1}{n} \log  \mu^n(\widetilde{E}_\gamma(\alpha,\epsilon,n,x_n)) >\widetilde{I}_{\gamma}(\alpha,\epsilon,n)-\epsilon'. \]
    Let \(y_{n,1},y_{n,2} \in \BP\) be the two distinct elements of \(\BP\) such that \(d_{\BP}(x_n,y_{n,i})=e^{-n(\gamma-2\epsilon) \alpha}\). Let \(g \in \widetilde{E}_\gamma(\alpha,\epsilon,n,x_n)\). For both \(i \in \{1,2\}\), 
    \[e^{-n(\gamma-2\epsilon) \alpha} -e^{-n(\gamma-\epsilon)\alpha} \leq d_{\BP}(y_{n,i},\omega_-(g)) \leq 2e^{-n(\gamma-2\epsilon) \alpha},\]
    which, assuming that \(n\) is large enough so that \(e^{n\epsilon\alpha} \geq 2\), implies that \(g \in \widetilde{E}_{\gamma}(\alpha,3\epsilon,n,y_{n,i})\).
    Hence, assuming \(n\) is large enough so that \(e^{-n(\gamma-3\epsilon)\alpha} \leq \frac{1}{2}\), by Lemma \ref{lem:umapsclosetoupsilon+} for both \(i\) we must further have 
    \[ \frac{1}{8}e^{-n((2-\gamma)\alpha+3(2+\alpha)\epsilon)} \leq d_{\BP}(g y_{n,i},\upsilon_+(g)) \leq e^{-n((2-\gamma)\alpha-3(2+\alpha)\epsilon)}.\]
    
    By considering the action of \(g\) on \(y_{n,1}\) and \(y_{n,2}\), it is clear that
    \(d_{\BP}(g y_{n,1},g y_{n,2}) \geq d_{\BP}(g y_{n,i},\upsilon_+(g))\) whenever \(n \in \N\) is sufficiently large. In particular, this holds when moreover \(e^{-n((2-\gamma)\alpha-3(2+\alpha)\epsilon)} \leq 1/2\), because when this holds we have \(d_{\BP}(g y_{n,i},\upsilon_+(g)) \leq 1/2\) and \(d_{\BP}(g y_{n,i},\upsilon_-(g)) \geq 1-d_{\BP}(g y_{n,i},\upsilon_+(g)) \geq 1/2\). It therefore follows that
    \begin{align*}
        \sup_{x, y \in \BP} \int |gx|^t \left(\frac{d_{\BP}(gx,gy)}{d_{\BP}(x, y)} \right)^{\zeta} \diff \mu^n(g) &\geq \int_{\widetilde{E}_{\gamma}(\alpha,3\epsilon,n,y_{n,1})} |g y_{n,1}|^t \left(\frac{d_{\BP}(g y_{n,1},g y_{n,2})}{d_{\BP}(y_{n,1}, y_{n,2})} \right)^{\zeta} \diff \mu^n(g) \\
        &\geq 2^t  e^{n(t(1-\gamma)\alpha+3t(1+\alpha)\epsilon)}\left(\frac{\frac{1}{8}e^{-n((2-\gamma)\alpha+3(2+\alpha)\epsilon)}}{2e^{-n(\gamma-2\epsilon)\alpha}} \right)^{\zeta} \mu^n(\widetilde{E}_\gamma(\alpha,\epsilon,n,x_n) ) \\
        &\geq \frac{2^t}{16^\zeta} e^{n(\widetilde{I}_\gamma(\alpha,\epsilon,n) + t(1-\gamma)\alpha-2\zeta(1-\gamma)\alpha)-O(\epsilon)-\epsilon')} 
    \end{align*}
Since \(\epsilon'\) and \(\epsilon\) can be made arbitrarily small, this implies
\[\varliminf_{n \rightarrow \infty} \frac{1}{n} \log \sup_{x, y \in \BP} \int |gx|^t \left(\frac{d_{\BP}(g x,g y)}{d_{\BP}(x, y)} \right)^{\zeta} \diff \mu^n(g) \geq \widetilde{I}_\gamma(\alpha)+t(1-\gamma)\alpha -2\zeta (1-\gamma) \alpha.\]

Now suppose \(\alpha>0\) and \(\gamma=0\). By an application of Lemma \ref{lem:bougerollemma} (i.e. using \(g^{-1}\)), for all \(g \in E(\alpha,\epsilon,n)\) and all \(x,y \in \BP\) with \(x\not=y\) we have 
\[\frac{d_{\BP}(gx,gy)}{d_{\BP}(x,y)} \geq e^{-2n(\alpha+\epsilon)} .\]
Let \(\epsilon',\epsilon>0\) and for each \(n \in \N\), let \(x_n \in \BP\) be such that
    \[\frac{1}{n} \log  \mu^n(\widetilde{E}_0(\alpha,\epsilon,n,x_n)) >\widetilde{I}_{0}(\alpha,\epsilon,n)-\epsilon'. \]
Using Lemma \ref{lem:essentialinequality}, for any \(n \in \N\) we have
\begin{align*}
 \sup_{x, y \in \BP} \int_{\widetilde{E}_0(\alpha,\epsilon,n,x_n)} |g x_n|^t \left(\frac{d_{\BP}(g x_n,g y)}{d_{\BP}(x, y)} \right)^{\zeta} \diff \mu^n(g) & \geq \int_{\widetilde{E}_0(\alpha,\epsilon,n,x_n) } |g x_n|^t e^{-2n\zeta(\alpha+\epsilon)} \diff \mu^n(g) \\
 &\geq 2^t e^{n(\widetilde{I}_0(\alpha,\epsilon,n)-\epsilon')} e^{nt(\alpha+(1+\alpha)\epsilon)} e^{-2n\zeta(\alpha+\epsilon)}.
\end{align*}
It follows that
\[\varliminf_{n \rightarrow \infty} \frac{1}{n} \log \sup_{x, y \in \BP} \int |gx|^t \left(\frac{d_{\BP}(g x,g y)}{d_{\BP}(x, y)} \right)^{\zeta} \diff \mu^n(g) \geq \widetilde{I}_0(\alpha)+t\alpha -2\zeta  \alpha.\]

Thus, for any \(\alpha>0\) we have shown 
\begin{align*}
    \lim_{\epsilon \rightarrow 0} \varliminf_{n \rightarrow \infty} \frac{1}{n} \log \sup_{x, y \in \BP} \int |gx|^t \left(\frac{d_{\BP}(g x,g y)}{d_{\BP}(x, y)} \right)^{\zeta} \diff \mu^n(g) &\geq \sup_{\gamma \in [0,2) } \{ \widetilde{I}_\gamma(\alpha)+t(1-\gamma)\alpha -2\zeta (1-\gamma) \alpha \}.
\end{align*}
It is easy to show from the definitions that \(\widetilde{I}_2(\alpha) \leq \widetilde{I}_\gamma(\alpha)\) for all \(\gamma \in [0,2]\) and \(\alpha> 0\). In particular, 
\begin{align*}
    \widetilde{I}_2(\alpha)-t\alpha +2\zeta \alpha &=\lim_{\gamma \rightarrow 2} \widetilde{I}_2(\alpha)+t(1-\gamma)\alpha -2\zeta(1-\gamma) \alpha \\
    &\leq \varlimsup_{\gamma \rightarrow 2} \widetilde{I}_\gamma(\alpha)+t(1-\gamma)\alpha -2\zeta(1-\gamma) \alpha \\
    &\leq \sup_{\gamma \in [0,2) } \{ \widetilde{I}_\gamma(\alpha)+t(1-\gamma)\alpha -2\zeta (1-\gamma) \alpha \}.
\end{align*}
That is,
\[ \sup_{\gamma \in [0,2) } \{ \widetilde{I}_\gamma(\alpha)+t(1-\gamma)\alpha -2\zeta (1-\gamma) \alpha \}= \sup_{\gamma \in [0,2]} \{ \widetilde{I}_\gamma(\alpha)+t(1-\gamma)\alpha -2\zeta (1-\gamma) \alpha\}. \]
Hence,
\begin{align*}
    \lim_{\epsilon \rightarrow 0} \varliminf_{n \rightarrow \infty} \frac{1}{n} \log \sup_{\substack{x, y \in \BP \\ x \not= y}} \int |gx|^t \left(\frac{d_{\BP}(g x,g y)}{d_{\BP}(x, y)} \right)^{\zeta} \diff \mu^n(g) & \geq 
    \sup_{\gamma \in [0,2] } \{ \widetilde{I}_\gamma(\alpha)+t(1-\gamma)\alpha -2\zeta (1-\gamma) \alpha\} \\
    &= \sup_{\gamma \in [0,2] }\{ I_\gamma(\alpha)+t(1-\gamma)\alpha -2\zeta (1-\gamma) \alpha \},
\end{align*}
where the last equality is by Lemma \ref{lem:QsupequalstildeQsup} (noting that \(t-2\zeta \leq 0\)).

Now suppose \(\alpha=0\). Let \(\epsilon>0\), \(n \in \N\), and \(g \in E(0,\epsilon,n)\). Again by an application of Lemma \ref{lem:bougerollemma}, for all \(x,y \in \BP\) with \(x\not=y\) we have 
\[\frac{d_{\BP}(gx,gy)}{d_{\BP}(x,y)} \geq e^{-2n\epsilon} .\] 
Thus,
\[\varliminf_{n \rightarrow \infty} \frac{1}{n} \log \sup_{x, y \in \BP} \int |gx|^t \left(\frac{d_{\BP}(g x,g y)}{d_{\BP}(x, y)} \right)^{\zeta} \diff \mu^n(g) \geq I_0(0).\]
\end{proof}

\begin{prop}\label{prop:Igamma<0forgammain12}
    \(\sup_{\alpha \in \mathrm{A}} \sup_{\gamma \in [1,2]} I_\gamma(\alpha)<0.\)
\end{prop}

\begin{proof}
Combining Proposition \ref{prop:zetabougerol} and \cite[Proposition 2.3]{BL85}, for \(\zeta>0\) sufficiently small we have
\[\sup_{\alpha \in \mathrm{A}} \sup_{\gamma \in [0,2]} \{I_{\gamma}(\alpha)-2\zeta(1-\gamma)\alpha\}<0,\]
which implies the proposition.
\end{proof}

Our only need for the `\(\geq\)' inequality in Proposition \ref{prop:zetabougerol} is to prove Proposition \ref{prop:Igamma<0forgammain12}. It is also possible to prove Proposition \ref{prop:Igamma<0forgammain12} using perturbation theory and the results of Le Page (i.e. by perturbation theory one can show \(k^+(t)=k(t)\) in a neighbourhood of 0, so \(k^+(t)\) must be strictly increasing in some neighbourhood). We preferred to give the more direct proof here. 

\subsection{Introducing \(t_c'\) and \(\zeta_t'\)}
We define
\[t_c':= \inf\left\{t \geq -1 :\sup_{\alpha \in \mathrm{A}} \sup_{\gamma \in [1,2]} \{I_{\gamma}(\alpha)+t(1-\gamma)\alpha\}<k^+(t)  \right\}, \]
and, for \(t>t_c'\),
\[\zeta_t':=\sup \left\{\zeta>0: \sup_{\alpha \in \mathrm{A}} \sup_{\gamma \in [1,2]} \{ I_\gamma(\alpha)+t(1-\gamma)\alpha -2\zeta (1-\gamma) \alpha\}< k^+(t) \right\} .\]
We will later show that \(t_c' = t_c\) and \(\zeta_t' = \zeta_t\) for \(t \in (t_c,0]\).
For now we only state and prove some basic facts about these quantities.

\begin{lemma}\label{lem:tc<0}
    \(t_c',t_c<0\).
\end{lemma}

\begin{proof}
We show \(t_c'<0\); the proof for \(t_c\) follows similarly. By Proposition \ref{prop:Igamma<0forgammain12} we have \(\sup_{\alpha \in \mathrm{A}} \sup_{\gamma \in [1,2]} I_{\gamma}(\alpha)<0=k^+(0)\). Since the supremum of affine functions is convex, 
\[t \mapsto \sup_{\gamma \in [1,2]} \{ I_\gamma(\alpha)+t(1-\gamma)\alpha \}\]
is continuous on \(\R\), as is \(k^+(t)\) by Proposition \ref{prop:oP}. Hence, for all \(t<0\) sufficiently close to 0 we have 
\[\sup_{\alpha \in \mathrm{A}} \sup_{\gamma \in [1,2]} \{ I_\gamma(\alpha)+t(1-\gamma)\alpha\}<k^+(t). \]
\end{proof}

\begin{lemma}\label{lem:zetatupperbound}
    For any \(t >t_c'\), \(\zeta_t' \in (0,1+t]\).
\end{lemma}

\begin{proof}
The inequality \(\zeta_t'>0\) is clear. By Lemma \ref{lem:elementaryQgammainequality}, for any \(\alpha \in \mathrm{A}\) and \(\gamma \in [0,2]\), 
\begin{align*}
    I_2(\alpha)-t\alpha+2(1+t) \alpha &\geq I_{\gamma}(\alpha)-(2-\gamma)\alpha -t\alpha +2(1+t) \alpha \\
    &= I_{\gamma}(\alpha)+t(1-\gamma)\alpha +(1+t)\gamma \alpha \\
    &\geq  I_{\gamma}(\alpha)+t(1-\gamma)\alpha.
\end{align*}
Hence,
\begin{align*}
\sup_{\alpha\in \mathrm{A}} \{ I_2(\alpha)-t\alpha+2(1+t) \alpha\}
&\geq
\sup_{\alpha \in \mathrm{A}} \sup_{\gamma \in [0,2]} \{ I_{\gamma}(\alpha)+t(1-\gamma)\alpha\} \\
&= k^+(t),   
\end{align*}
so
\[\sup_{\alpha\in \mathrm{A}} \sup_{\gamma \in [1,2]} \{ I_\gamma(\alpha)+t(1-\gamma)\alpha-2(1+t)(1-\gamma) \alpha\} 
\geq k^+(t).\]
This implies \(\zeta_t' \leq 1+t\).
\end{proof}

\begin{lemma}\label{lem:zetat'withgammain02}
 For all  \(t>t_c'\) and \(0<\zeta<\zeta_t'\),
    \[\sup_{\alpha \in \mathrm{A}} \sup_{\gamma \in [0,2]} \{ I_\gamma(\alpha)+t(1-\gamma)\alpha -2\zeta (1-\gamma) \alpha\}< k^+(t).\]
\end{lemma}

\begin{proof}
It suffices to show that
\[\sup_{\alpha \in \mathrm{A}} \sup_{\gamma \in [0,1]} \{I_\gamma(\alpha)+t(1-\gamma)\alpha -2\zeta (1-\gamma) \alpha\}< k^+(t) .\]
Let \(\delta>0\) be such that 
\[\sup_{\alpha \in \mathrm{A}} \sup_{\gamma \in [1,2]} \{I_\gamma(\alpha)+t(1-\gamma)\alpha -2\zeta (1-\gamma) \alpha\} \leq k^+(t) -\delta.\]
Let \(K=1/2\). By Lemma \ref{lem:elementaryQgammainequality}, for \(\alpha \in \mathrm{A}\)  and \(\gamma \in [0,1]\) with \((1-\gamma)\alpha \leq \frac{K\delta}{1+t}\),
\begin{align*}
    I_{\gamma}(\alpha)+t(1-\gamma)\alpha-2\zeta(1-\gamma)\alpha & \leq I_1(\alpha)+(1+t-2\zeta)(1-\gamma)\alpha \\  &\leq I_1(\alpha)+(1+t)(1-\gamma)\alpha \\
    &\leq k^+(t)-\delta+(1+t)(1-\gamma)\alpha \\
    &\leq k^+(t)-\frac{\delta}{2}.
\end{align*}
Moreover, for \(\alpha \in \mathrm{A}\)  and \(\gamma \in [0,1]\) with \((1-\gamma)\alpha \geq \frac{K\delta}{1+t}\),
\begin{align*}
    I_{\gamma}(\alpha)+t(1-\gamma)\alpha-2\zeta(1-\gamma)\alpha &\leq  I_{\gamma}(\alpha)+t(1-\gamma)\alpha-\frac{2\zeta K\delta}{1+t} \\
    &\leq k^+(t) -\frac{2\zeta K\delta}{1+t}.
    \end{align*}
Hence, we have shown
\[\sup_{\alpha \in \mathrm{A}} \sup_{\gamma \in [0,1]} \{I_\gamma(\alpha)+t(1-\gamma)\alpha -2\zeta (1-\gamma) \alpha\} <k^+(t).\]
\end{proof}

\begin{lemma}\label{lem:zetatunique}
    For all \(t >t_c'\), \(\zeta_t'\) is the unique \(\zeta>0\) such that 
    \[\sup_{\alpha \in \mathrm{A}} \sup_{\gamma \in [1,2]} \{I_{\gamma}(\alpha)+t(1-\gamma)\alpha-2\zeta(1-\gamma)\alpha \}=k^+(t).\]
    Similarly, for \(t >t_c\), \(\zeta_t\) is the unique \(\zeta>0\) such that 
    \[\sup_{\alpha \in \mathrm{A}} \{I_{2}(\alpha)-t\alpha+2\zeta\alpha \}=k(t).\]
\end{lemma}

\begin{proof}
First note that \[\zeta \mapsto \sup_{\alpha \in \mathrm{A}} \sup_{\gamma \in [1,2]} \{I_{\gamma}(\alpha)+t(1-\gamma)\alpha-2\zeta(1-\gamma)\alpha \}\]
is continuous by convexity. Hence, \(\zeta_t'\) satisfies
\[\sup_{\alpha \in \mathrm{A}} \sup_{\gamma \in [1,2]} \{I_{\gamma}(\alpha)+t(1-\gamma)\alpha-2\zeta_t'(1-\gamma)\alpha \}=k^+(t).\]
By Lemma \ref{lem:USC}, for each \(t> t_c'\) there exists \(\beta_t \in \mathrm{A}\) and \(\gamma_t \in [1,2]\) such that 
\begin{equation}\label{eqn:betatgammat}
    I_{\gamma_t}(\beta_t)+t(1-\gamma_t)\beta_t-2\zeta_t'(1-\gamma_t)\beta_t =\sup_{\alpha \in \mathrm{A}} \sup_{\gamma \in [1,2]} \{I_{\gamma}(\alpha)+t(1-\gamma)\alpha-2\zeta_t'(1-\gamma)\alpha \}=k^+(t).
\end{equation}
Since \(t>t_c'\), \(\sup_{\alpha \in \mathrm{A}} \sup_{\gamma \in [1,2]} \{I_{\gamma}(\alpha)+t(1-\gamma)\alpha\}<k^+(t)\), so we must have \(\beta_t>0\) and \(\gamma_t >1\). Hence, for \(\zeta>\zeta_t'\),
\[\sup_{\alpha \in \mathrm{A}} \sup_{\gamma \in [1,2]} \{I_{\gamma}(\alpha)+t(1-\gamma)\alpha-2\zeta(1-\gamma)\alpha \} \geq I_{\gamma_t}(\beta_t)+t(1-\gamma_t)\beta_t-2\zeta(1-\gamma_t)\beta_t  > k^+(t),\]
which proves the first statement. The proof for \(\zeta_t\) follows similarly. 
\end{proof}

\begin{lemma}\label{lem:continuity}
The map \(t \mapsto \zeta_t'\) is continuous on \((t_c',\infty)\) and the map \(t \mapsto \zeta_t\) is continuous on \((t_c,\infty)\).
\end{lemma}

\begin{proof}
We show the first statement; the second follows similarly. Fix some \(t>t_c'\) and let \(\epsilon>0\). We let 
\[F(t,\zeta):= \sup_{\alpha \in \mathrm{A}} \sup_{\gamma \in [1,2]} \{I_{\gamma}(\alpha)+t(1-\gamma)\alpha-2\zeta(1-\gamma)\alpha \}-k^+(t).\]
By Lemma \ref{lem:zetatunique}, we have
\[F(t,\zeta_t'+\epsilon)>0\]
and
\[F(t,\zeta_t'-\epsilon)<0.\]
Note that \(s \mapsto F(s,\zeta_t'+\epsilon)\) and \(s \mapsto F(s,\zeta_t'-\epsilon)\) are continuous. In particular, there exists \(\delta>0\) such that for all \(s \in (t_c',\infty)\) with \(|s-t|<\delta\), 
\[|F(t,\zeta_t'+\epsilon)-F(s,\zeta_t'+\epsilon)|<\frac{F(t,\zeta_t'+\epsilon)}{2} \]
and
\[|F(t,\zeta_t'-\epsilon)-F(s,\zeta_t'-\epsilon)|<\frac{-F(t,\zeta_t'-\epsilon)}{2}.\]
Hence, for all \(s \in (t_c',\infty)\) with \(|s-t|<\delta\),
\[ F(s,\zeta_t'+\epsilon)>F(t,\zeta_t'+\epsilon)/2>0 \]
and 
\[ F(s,\zeta_t'-\epsilon)<F(t,\zeta_t'-\epsilon)/2<0.\]
Since \(\zeta \mapsto F(s,\zeta)\) is increasing and \(\zeta_s'\) is the unique \(\zeta\) such that \(F(s,\zeta)=0\), we must have \(|\zeta_s'-\zeta_t'|<\epsilon\).
\end{proof}

\subsection{Quasi-compactness}
In this section we show \(P_t: C^{\zeta}(\BP) \rightarrow C^{\zeta}(\BP)\) is quasi-compact for \(t \in (t_c',0]\) and \(\zeta \in (0,\zeta_t')\) by proving a Lasota--Yorke type inequality (Proposition \ref{prop:lasotayorke}).

The following H\"older estimate is much stronger than we require, but we remark that the exponent \(2\zeta-t\) is optimal. 

\begin{lemma}\label{lem:lipschitzbound}
   For all \(-1<t<0\) and \(0<\zeta \leq 1\), there exists \(C>0\) such that for any \(g \in \SL(2,\R)\) and \(x,y \in \BP\),
    \[ \left| |gx|^t- |gy| ^t \right| \leq C|g|^{2\zeta-t} d_{\BP} (x,y)^{\zeta}. \]
\end{lemma}

\begin{proof}
We let \(C>1\) be as given by \cite[Lemma 12.2]{BQ16}. Let \(x ,y \in \BP\). Then, by \cite[Lemma 12.2]{BQ16}, for any \(g \in \SL(2,\R)\) we have 
\[\left|\log|g x|-\log|g y| \right| \leq C |g|^2 d_{\BP}(x,y).\]
By the mean value theorem, for \(-1 <t<0\) and \(a,b \in \R\),
\[|e^{ta}-e^{tb}| \leq |t|e^{t\min\{a,b\}} |a-b|,\]
so it follows that
\[\left||g x|^t-|g y|^t \right|= \left| e^{t \log |gx|}-e^{t \log |gy|} \right| \leq C|g|^{2-t}d_{\BP}(x,y). \]
Thus, using the trivial inequality that \(\min\{a,b\} \leq a^{1-\zeta} b^{\zeta}\) for \(a,b \geq 0\), 
\begin{align*}
    \left||g x|^t-|g y|^t \right| &\leq \min\{ 2|g|^{-t},C |g|^{2-t} d_{\BP}(x,y)\}  \\
    &\leq 2^{1-\zeta}C^{\zeta} |g|^{2\zeta-t} d_{\BP}(x,y)^{\zeta}.
\end{align*}

\end{proof}

\begin{prop}\label{prop:lasotayorke}
    For all \(0<\zeta<\zeta_t'\) and \(0<r<1\) with
    \begin{equation*}
        \log r < \sup_{\alpha \in \mathrm{A}}\sup_{\gamma \in [0,2]} \{ I_{\gamma}(\alpha)+t (1-\gamma)\alpha-2\zeta(1-\gamma)\alpha \} -k^+(t),
    \end{equation*}
    there exists \(C>0\) and \(n \in \N\) such that for all \(f \in C^{\zeta}(\BP)\),
    \[| e^{-n k^+(t)} P_t^n f |_{\zeta} \leq r^n |f|_{\zeta} +C|f|_{\infty}.\]
\end{prop}

\begin{proof}
    For any \(f \in C^{\zeta}(\BP)\) and \(x ,y \in \BP\) with \(x \not= y\),
    \begin{align*}
        \frac{|P_t^n f(x)- P_t^n f(y)|}{d_{\BP}(x,y)^\zeta} \\
        =  d_{\BP}(x,y)^{-\zeta} \left| \int  |gx|^t f(gx)-|gy|^t f(gy) \diff \mu^n(g) \right|\\
        = d_{\BP}(x,y)^{-\zeta} \left| \int  |gx|^t (f(gx)-f(gy))+\left(|gx|^t-|gy|^t \right) f(gy) \diff \mu^n(g) \right| \\
        \leq  [f]_{\zeta} \int  |gx|^t \frac{d_{\BP}(gx,gy)^\zeta}{d_{\BP}(x,y)^\zeta} \diff \mu^n(g)  + |f|_{\infty} d_{\BP}(x,y)^{-\zeta} \int  \left||gx|^t-|gy|^t \right| \diff \mu^n(g).
    \end{align*}
    Let \(C>0\) be as in Lemma \ref{lem:lipschitzbound}. For each \(n \in \N\), 
    \[ \sup_{ \substack{x,y \in \BP \\  x \not=y }} d_{\BP}(x,y)^{-\zeta} \int \left||gx|^t-|gy|^t \right| \diff \mu^n(g) \leq C \int |g|^{2\zeta-t} \diff \mu^n(g)<\infty.\]
    Hence, using Proposition \ref{prop:zetabougerol}, for all \(0<r<1\) with
    \[\log r < \sup_{\alpha \in \mathrm{A}}\sup_{\gamma \in [0,2]} \{ I_{\gamma}(\alpha)+t (1-\gamma)\alpha-2\zeta(1-\gamma)\alpha \} -k^+(t),\]
    there exists \(C'>0\) and \(n \in \N\) such that for all \(f \in C^{\zeta}(\BP)\),
    \begin{align*}
        |e^{-n k^+(t)}  P_t^n f|_{\zeta} &\leq r^n |f|_{\zeta} +C' |f|_{\infty}.
    \end{align*}
\end{proof}

\begin{lemma}\label{lem:spectralradius} 
\(\log \rho(P_t) \geq k^+(t)\).
\end{lemma}

\begin{proof}
We have \( | P_t^n 1|_{\zeta} \geq |P_t^n 1|_{\infty} \), so by Proposition \ref{prop:oP} and Gelfand's formula, 
\[\log \rho(P_t) \geq \varlimsup_{n \rightarrow \infty} \frac{1}{n}\log |P_t^n 1|_{\infty} \equiv k^+(t). \]
\end{proof}

We use the following well known result of Hennion: 

\begin{prop}[{\cite[Theorem XIV.3]{HH01}}] \label{prop:hennion}
Let \((X,|\cdot|_1)\) be a Banach space and \(T:(X, |\cdot|_1) \rightarrow (X, |\cdot|_1)\)
a bounded linear operator. Suppose there exists a norm \(|\cdot|_2\)
such that:
\begin{enumerate}[label=(\roman*)]
    \item every sequence \(f_n \in \{ f \in X: |f|_1 \leq 1\}\) contains a subsequence which is Cauchy in \((X,| \cdot|_2)\),
    \item there exists a constant \(M>0\)  such that for all \(f \in X\), \(|T f|_2 \leq M |f|_2\),
    \item there exist \(n \in \N\) and \(r,C>0\) such that \(r<\rho(T)\) and for every \(f \in X\),
        \[|T^n f|_1 \leq r^n |f|_1+C |f|_2.\]
\end{enumerate}
Then, \(T\) is quasi-compact and the essential spectral radius of \(T\) is less than or equal to r.
\end{prop}

\begin{lemma}\label{lem:quasi-compact}
For all \(0<\zeta<\zeta_t'\), \(e^{-k^+(t)} P_t: C^{\zeta}(\BP) \rightarrow C^{\zeta}(\BP)\) is quasi-compact and its essential spectral radius is less than or equal to 
\[\exp(\sup_{\alpha \in \mathrm{A}} \sup_{\gamma \in [0,2]} \{  I_{\gamma}(\alpha)+t (1-\gamma)\alpha-2\zeta(1-\gamma)\alpha \}-k^+(t))<1.\]
\end{lemma}

\begin{proof}
This follows by applying Proposition \ref{prop:hennion} with \(|\cdot|_1:=|\cdot|_{\zeta}\) and \(|\cdot|_2:=|\cdot|_\infty\). Note \((i)\) is satisfied by the Arzel\`a--Ascoli theorem and \((ii)\) is clearly satisfied. Lemma \ref{lem:spectralradius} implies that \(\rho(e^{-k^+(t)} P_t) \geq 1\) and Lemma \ref{lem:zetat'withgammain02} implies that 
\[\exp\left(\sup_{\alpha \in \mathrm{A}} \sup_{\gamma \in [0,2]} \{  I_{\gamma}(\alpha)+t (1-\gamma)\alpha-2\zeta(1-\gamma)\alpha \}-k^+(t) \right)<1.\]
Hence, by Proposition \ref{prop:lasotayorke}, the conditions in \((iii)\) are satisfied for all 
\[ 0< r <\exp\left(\sup_{\alpha \in \mathrm{A}} \sup_{\gamma \in [0,2]} \{  I_{\gamma}(\alpha)+t (1-\gamma)\alpha-2\zeta(1-\gamma)\alpha \}-k^+(t) \right).\]
\end{proof}

\subsection{Proof of the spectral gap} 
To prove the spectral gap, we use the Ruelle-Perron-Frobenius theorem in \cite[\S 6]{PP22}. This theorem is a modification of the results \cite{Sas64} which can be applied in our case. 

We fix some \(t_c'<t \leq 0\) and \(0<\zeta<\zeta_t'\). We define 
\[ \MCC_+:= \{ f \in C^{\zeta}(\BP): f(x) \geq 0, \forall x \in \BP\},\]
\[\interior(\MCC_+):= \{ f \in C^{\zeta}(\BP): f(x) > 0, \forall x \in \BP\}, \] 
\[\partial \MCC_+:=\MCC_+ \setminus \interior(\MCC_+) .\]
Then, \(\MCC_+\) is a closed proper cone in the sense of \cite[Definition 6.10]{PP22} and \cite[\S 3]{Sas64}.

\begin{lemma}\label{lem:ifPtnfx=0fgoesto0}
    For all \(0<r<1\) such that
    \[\log r < \sup_{\alpha \in \mathrm{A}} \sup_{\gamma \in [0,2]} \{I_{\gamma}(\alpha)+t(1-\gamma)\alpha-2\zeta(1-\gamma)\alpha\}-k^+(t), \] 
    there exists \(C>0\) such that for all \(n \in \N \) and for all \(f \in \MCC_+\) such that \(P_t^n f\in \partial \MCC_+\), we have
    \[|e^{-nk^+(t)} P_t^n f|_{\infty} \leq C r^n |f|_{\zeta}.\]
\end{lemma}
\begin{proof}
Let \(f\) be as in the lemma and let \(x_0 \in \BP\) be such that \(P_t^n f(x_0)=0\). We have
    \[0=P_t^n f(x_0) =\int |g x_0|^t f(g x_0 ) \diff \mu^n(g),\]
    which implies that \(f(g x_0)=0 \) for \(\mu^n\)-a.e. \(g \in \SL(2,\R)\).
    Hence, for any \(x\in \BP\),
    \begin{align*}
        P_t^n f(x) &=\int |gx|^t f(g x) \diff \mu^n(g) \\
        &\leq \int |gx|^t \left(f(g x_0)+[f]_{\zeta} d_{\BP}(gx, g x_0 )^{\zeta} \right) \diff \mu^n(g) \\
        &=  [f]_{\zeta} \int |gx|^t d_{\BP}(gx,g x_0  )^\zeta \diff \mu^n(g) \\
        &\leq  |f|_{\zeta} \sup_{\substack{x',y' \in \BP \\ x' \not= y'}} \int |gx'|^t \frac{d_{\BP}(gx',g y'  )^\zeta} {d_{\BP}(x',y')^{\zeta}} \diff \mu^n(g),
    \end{align*}
    where in the last line we are using that \(d_{\BP}(\cdot, \cdot) \leq 1\). Hence, the lemma follows by Proposition \ref{prop:zetabougerol}. 
\end{proof}

\begin{lemma}\label{lem:strictlypositiveeigenfunction}
    There exists \(h_t \in \interior(\MCC_+)\) such that \(P_t h_t = \rho(P_t) h_t\).
\end{lemma}

\begin{proof}
    Theorem 3 in \cite{Sas64} implies the existence of \(h_t \in \MCC_+ \setminus\{x \mapsto 0\}\) such that \( P_t h_t = \rho(P_t) h_t\). We now show that \(h_t \in \interior(\MCC_+)\). To see this, assume for a contradiction that \(h_t \in \partial \MCC_+\), then \(P_t^n h_t \in \partial \MCC_+\) for all \(n \in \N\). Hence, by Lemmas \ref{lem:spectralradius} and \ref{lem:ifPtnfx=0fgoesto0}, we have
    \[|h_t|_{\infty}=|\rho(P_t)^{-n} P_t^n h_t|_{\infty} \leq |e^{-n k^+(t)} P_t^n h_t|_{\infty} \xrightarrow[n \rightarrow \infty]{} 0,\]
    which is a contradiction.
\end{proof}

\begin{lemma}\label{lem:ktk+}
\(k(t)=k^+(t)=\log \rho(P_t)\).
\end{lemma}

\begin{proof}
For all \(x \in \BP\) and \(n \in \N\), 
\[P_t^n 1(x) \geq \sup(h_t)^{-1} P_t^n h_t(x) = \sup(h_t)^{-1} \rho(P_t)^n  h_t(x) \geq  \sup(h_t)^{-1} \inf(h_t) \rho(P_t)^n\]
and, similarly,
\[ P_t^n 1(x) \leq \sup(h_t) \inf(h_t)^{-1} \rho(P_t)^n,\]
where \(\inf(h_t)>0\) and \(\sup (h_t) <\infty\). Hence, \(\log \rho(P_t) \leq k^-(t) \leq k^+(t) \leq \log \rho(P_t)\), so we have equality throughout. By Proposition \ref{prop:uP}, we further have that \(k(t)=k^-(t)\).
\end{proof}

We note that by Lemma \ref{lem:ktk+} we can replace \(k^+(t)\) with \(k(t)\) in the definition of \(\zeta_t'\) and write 
\begin{equation}\label{eqn:zetat'improved}
    \zeta_t' \equiv \sup \left\{ \zeta>0: \sup_{\alpha \in \mathrm{A}} \sup_{\gamma \in [1,2]} \{I_{\gamma}(\alpha)+t(1-\gamma)\alpha-2\zeta(1-\gamma)\alpha\}<k(t) \right\}.
\end{equation}

\begin{lemma}\label{lem:existenceofnut}
There exists \(\nu_t \in \MCP(\BP)\) such that  \((P_t)_{\Cdot} \nu_t = e^{k(t)} \nu_t\).
\end{lemma}

\begin{proof}
    By the Schauder fixed point theorem, there exists \(\nu_t\) such that \(\nu_t = \frac{(P_t)_{\Cdot}\nu_t}{(P_t)_{\Cdot} \nu_t (1)}\), that is \((P_t)_{\Cdot}\nu_t= (P_t)_{\Cdot} \nu_t (1) \nu_t\). We have
    \[(P_t)_{\Cdot} \nu_t (1) \nu_t(h_t)=(P_t)_{\Cdot}\nu_t(h_t)=\nu_t(P_t h_t) =e^{k(t)} \nu_t(h_t),\]
    so \((P_t)_{\Cdot} \nu_t (1)=e^{k(t)}\).
\end{proof}

Scaling \(h_t\) if necessary, from now on we will always assume that \(\nu_t(h_t)=1\).

\begin{prop}\label{prop:spectralgap}
For all \(t_c'<t \leq 0\), there exists a unique probability measure \(\nu_t \in \MCP(\BP)\) such that 
        \((P_t)_{\Cdot} \nu_t=e^{k(t)} \nu_t\), and a unique H\"older continuous function \(h_t\) satisfying \(\nu_t(h_t)=1\) and 
        \(P_t h_t=e^{k(t)} h_t.\) The function \(h_t\) is strictly positive and  \(h_t \in C^{\zeta}(\BP)\) for all \(0<\zeta<\zeta_t'\). Moreover, for all \(0<\zeta<\zeta_t'\), we can write
    \[ P_t= e^{k(t)} (\nu_t \otimes  h_t+ S_t),\]
    where \((\nu_t \otimes  h_t) f:=\nu_t(f) h_t\) and \(S_t:C^{\zeta}(\BP) \rightarrow C^{\zeta}(\BP)\) satisfies \(\rho(S_t)<1\).
\end{prop}

\begin{proof}
We check that the conditions of \cite[Theorem 6.13]{PP22} are satisfied. Note it is clear that the set of point mass measures \(\MCS:=\{\delta_x:x \in \BP\}\) is a sufficient set for \(\MCC_+\) (\cite[Definition 6.11]{PP22}). Moreover, by Lemmas \ref{lem:ifPtnfx=0fgoesto0} and \ref{lem:ktk+}, \(e^{-k(t)} P_t\) is semi-positive (\cite[Definition 6.2]{PP22}) with respect to \(\MCC_+\) and \(\MCS\), since if \(e^{-nk(t)}P_t^n f \in \partial \MCC_+\) for all \(n \in \N\) then we have \(|e^{-nk(t)}P_t^n f|_{\infty} \rightarrow 0\). By Lemma \ref{lem:quasi-compact}, \(e^{-k(t)}P_t\) is also quasi-compact. Thus, the proposition follows by \cite[Theorem 6.13]{PP22}. In particular, we note that the \(u\), \(u^*\) in \cite[Theorem 6.14.(1)]{PP22} must be unique by \cite[Theorem 6.14.(4)]{PP22}, so by Lemmas \ref{lem:strictlypositiveeigenfunction} and \ref{lem:existenceofnut} they are given by \(h_t\) and \(\nu_t\).
\end{proof}

\subsection{The leading eigenfunction}\label{subsec:eigenfunction}
Fix some \(t_c'<t < 0\). Let \({}^* \! \eta_n:= \frac{({}^* \! P_t^{n})_{\Cdot}\lambda }{({}^* \! P_t^{n})_{\Cdot}\lambda(1)} \in \MCP(\BP) \). By weak* compactness the sequence \({}^* \! \eta_n \in \MCP(\BP)\) has a weak* limit point in \(\MCP(\BP)\), \({}^* \! \nu_t\) say. In this section we relate the eigenfunction \(h_t\) to the measure \({}^* \! \nu_t\).

\begin{lemma}\label{lem:Ptdelta*Pt}
    For any \({}^* \! \nu \in \MCP(\BP)\), \(n \in \N\), and \(x \in \BP\),
    \[P_t^n \left(\int | \innerproduct{\cdot}{w}|^t \diff {}^* \! \nu(w) \right)(x) =\int | \innerproduct{x}{w}|^t \diff ({}^* \! P_t^n)_{\Cdot} {}^* \! \nu(w). \]
\end{lemma}
\begin{proof}
For \(x, w \in \BP\) let \(u,v \in S^1\) be such that \(x=\R u, w=\R v\). Then, recalling our notation (\ref{eqn:innerproduct}), 
\[|gx|^t | \innerproduct{gx}{w}|^t = |gu|^t \frac{|\innerproduct{gu}{v}|^t}{|gu|^t}= |g^* v|^t \frac{|\innerproduct{u}{g^* v}|^t}{|g^* v|^t} =|g^* w|^t |\innerproduct{x}{g^* w}|^t.  \]
Hence, 
    \begin{align*}
        P_t^n \left(\int | \innerproduct{\cdot}{w}|^t \diff {}^* \! \nu(w) \right)(x)
        &=\int \int |gx|^t | \innerproduct{gx}{w}|^t \diff {}^* \! \nu(w) \diff \mu^n(g) \\
        &= \int  \int |g^* w|^t | \innerproduct{x}{g^* w}|^t \diff {}^* \! \nu(w) \diff \mu^n(g)  \\
        &= \int  \int |g^* w|^t | \innerproduct{x}{g^* w}|^t \diff \mu^n(g) \diff {}^* \! \nu(w) \\
        &= \int | \innerproduct{x}{w}|^t \diff ({}^* \! P_t^n)_{\Cdot}{}^* \! \nu(w),
    \end{align*}
where the penultimate equality is by Tonelli's theorem.
\end{proof}

We recall the definition of \(C_{\lambda,t}\) in the proof of Proposition \ref{prop:uP}. This is equivalently equal to
\[C_{\lambda,t} \equiv \int |\innerproduct{x}{w}|^t \diff \lambda(w) \]
for any \(x \in \BP\).

\begin{lemma}\label{lem:lambdaht}
    For any \(n \in \N\), \(({}^* \! P_t^{n})_{\Cdot}\lambda(1)=(P_t^{n})_{\Cdot}\lambda(1)\). In particular,
    \begin{equation}\label{eqn:e-nktblah}
         \lim_{n \rightarrow \infty} e^{-n k(t)} ( {}^* \!  P_t^{n})_{\Cdot}\lambda(1)=\lambda(h_t).
    \end{equation}
\end{lemma}

\begin{proof}
By Lemma \ref{lem:Ptdelta*Pt}, for any \(n \in \N\) we have
    \begin{align}
        P_t^n 1 (x)&=C_{\lambda,t}^{-1} P_t^n \left(\int | \innerproduct{\cdot}{w}|^t \diff \lambda(w) \right)(x) \nonumber \\
        &= C_{\lambda,t}^{-1} ({}^* \! P_t^{n})_{\Cdot}\lambda(1) \int |\innerproduct{x}{w}|^t \diff {}^* \! \eta_n(w).\label{eqn:longseriesofequalities}
    \end{align}
    Integrating over \(\lambda\) and applying Tonelli's theorem,
    \begin{align*}
         \int P_t^n 1 (x) \diff \lambda(x)&= C_{\lambda,t}^{-1} ({}^* \! P_t^{n})_{\Cdot}\lambda(1) \int \left( \int |\innerproduct{x}{w}|^t \diff{}^* \! \eta_n (w) \right) \diff \lambda(x) \\
         &= C_{\lambda,t}^{-1} ({}^* \! P_t^{n})_{\Cdot}\lambda(1) \int \left( \int |\innerproduct{x}{w}|^t \diff \lambda(x) \right) \diff {}^* \! \eta_n (w) \\
         &= ({}^* \! P_t^{n})_{\Cdot}\lambda(1),
    \end{align*}
    which proves the first statement. The second statement follows from the first using that \( e^{-n k(t)} P_t^n 1(x) \rightarrow h_t(x) \) uniformly on \(\BP\), by Proposition \ref{prop:spectralgap}.
\end{proof}

\begin{lemma}\label{lem:inthtinequality}
For all \(x \in \BP\),
    \(\int |\innerproduct{x}{w}|^t \diff {}^* \! \nu_t(w) \leq C_{\lambda,t} \lambda(h_t)^{-1} h_t(x)\).
\end{lemma}

\begin{proof}
    Let \((n_k)_{k \in \N}\) be such that \({}^* \! \eta_{n_k} \rightarrow {}^* \! \nu_t\) weak*. Combining (\ref{eqn:longseriesofequalities}) with (\ref{eqn:e-nktblah}) implies that
    \begin{equation}\label{eqn:limitintegralink}
        \lim_{k \rightarrow \infty} \int |\innerproduct{x}{w}|^t \diff {}^* \! \eta_{n_k} (w)= C_{\lambda,t} \lambda(h_t)^{-1} h_t(x)
    \end{equation}
    for all \(x \in \BP\).
    
Fix \(x \in \BP\) and let \(M>0\). We let
\[f_{x,M}(w) := \min\{ |\innerproduct{x}{w}|^t,M  \}.\]
Then, the function \(w \mapsto f_{x,M}(w)\) is continuous and bounded. Hence, since \({}^* \! \eta_{n_k} \rightarrow {}^*\! \nu_t\) weak*, 
\[ \int f_{x,M}(w) \diff  {}^*\! \nu_t = \lim_{k \rightarrow \infty} \int f_{x,M}(w) \diff {}^* \! \eta_{n_k}(w).  \]
Thus, by the monotone convergence theorem and (\ref{eqn:limitintegralink}),
\begin{align*}
    \int |\innerproduct{x}{w}|^t \diff {}^*\! \nu_t(w) &= \sup_{M>0} \int f_{x,M}(w) \diff  {}^*\! \nu_t(w) \\
    &= \sup_{M>0} \lim_{k \rightarrow \infty}  \int f_{x,M}(w) \diff {}^* \! \eta_{n_k}(w) \\
    &\leq \varlimsup_{k \rightarrow \infty} \sup_{M>0} \int f_{x,M}(w) \diff {}^* \! \eta_{n_k}(w) \\
    &=\varlimsup_{k \rightarrow \infty} \int |\innerproduct{x}{w}|^t \diff {}^* \! \eta_{n_k}(w) \\
    &= C_{\lambda,t} \lambda(h_t)^{-1} h_t(x).
\end{align*}
\end{proof}

\begin{lemma}\label{lem:lebesgueaeequality}
 For \(\lambda\)-a.e. \(x \in \BP\),  \(\int |\innerproduct{x}{w}|^t \diff {}^* \! \nu_t(w) =C_{\lambda,t} \lambda(h_t)^{-1} h_t(x)\).
\end{lemma}

\begin{proof}
With an application of Tonelli's theorem, we have 
\[\int \left( \int|\innerproduct{x}{w}|^t \diff {}^* \! \nu_t(w) \right) \diff \lambda(x) = C_{\lambda,t}.\]
We also have 
\[ \int C_{\lambda,t} \lambda(h_t)^{-1} h_t(x) \diff \lambda(x)=C_{\lambda,t}.\]
Hence, the lemma follows from Lemma \ref{lem:inthtinequality} and the simple fact that if \(f \leq g\) everywhere on \(\BP\) and \(\int f \diff \lambda=\int g \diff \lambda\), then \(f=g\) for \(\lambda\)-a.e. \(x \in \BP\).
\end{proof}

\begin{lemma}\label{lem:dimlowerbound}
\(h_t(x) \equiv C_{\lambda,t}^{-1} \lambda(h_t) \int |\innerproduct{x}{w}|^t \diff {}^* \! \nu_t(w) \) and \(\dim_F {}^* \! \nu_t \geq -t+\zeta_t' \).
\end{lemma}

\begin{proof}
By Lemma \ref{lem:lebesgueaeequality} and Proposition \ref{prop:Holdergivesdimensionbound} we have \(h_t(x)= C_{\lambda,t}^{-1} \lambda(h_t)  \int |\innerproduct{x}{w}|^t \diff {}^* \! \nu_t(w) \) for all \(x \in \BP\) and \(\dim_F {}^* \! \nu_t \geq -t+ \zeta_t'\). We note that we can apply Proposition \ref{prop:Holdergivesdimensionbound}  since \(|\innerproduct{x}{w}| \equiv d_{\BP}(x,w^{\perp})\).
\end{proof}

\begin{prop}
    \( ({}^* \! P_t)_{\Cdot}  {}^* \!  \nu_t = e^{k(t)}  {}^* \! \nu_t  \).
\end{prop}

\begin{proof}
Since \(x \mapsto \int |\innerproduct{x}{w}|^t \diff {}^* \! \nu_t(w)\) is a scalar multiple of \(h_t\), we have for every \(x \in \BP\),
\begin{align*}
     e^{k(t)}\int |\innerproduct{x}{w}|^t \diff {}^* \! \nu_t(w) &=  P_t \left(\int |\innerproduct{\cdot}{w}|^t \diff {}^* \! \nu_t(w) \right)(x)  \\
     &= \int |\innerproduct{x}{w}|^t \diff (({}^* \! P_t)_{\Cdot} {}^* \! \nu_t)(w),
\end{align*}
where the last equality is by Lemma \ref{lem:Ptdelta*Pt}. Hence, applying Proposition \ref{prop:Inu1=Inu2impliesnu1=nu2} with \(\nu_1 :={}^*\! \nu_t\) and \(\nu_2 :=e^{-k(t)} ({}^* \! P_t)_{\Cdot} {}^*\! \nu_t\) gives that \(e^{-k(t)} ({}^* \! P_t)_{\Cdot} {}^*\! \nu_t={}^* \! \nu_t\). 
\end{proof} 

\begin{remark}
By the same argument as in the proof of Proposition \ref{prop:uP}, for all \(t \in (-1,0]\) we have 
 \[ \lim_{n \rightarrow \infty} \frac{1}{n} \log ({}^* \! P_t^n)_{\Cdot}\lambda(1)=k(t).\]
Therefore, a general theorem of Maucourant \cite[Theorem IV.6]{Mau14} implies the existence of a probability measure \({}^* \! \nu_t \in \MCP(\BP)\) satisfying \(({}^* \! P_t)_{\Cdot} {}^* \! \nu_t=e^{k(t)} {}^* \! \nu_t \) for \textit{all} \(t \in (-1, 0]\). However, the proof does not show that \( {}^* \! \nu_t\) can be taken to be a weak* limit of the sequence \(\frac{({}^* \! P_t^n)_{\Cdot} \lambda}{({}^* \! P_t^n)_{\Cdot}\lambda (1)}\), but instead the sequence \(\frac{\sum_{n \leq N}  ({}^* \! P_t^n)_{\Cdot}\lambda}{\sum_{n \leq N} ({}^* \! P_t^n)_{\Cdot }\lambda(1)} \) (i.e. after using Lemma \ref{lem:lambdaht} above to rule out the second case in the proof of \cite[Theorem IV.6]{Mau14}). We preferred to show \({}^* \! \nu_t\) can be taken to be a weak* limit of \(\frac{({}^* \! P_t^n)_{\Cdot} \lambda}{({}^* \! P_t^n)_{\Cdot}\lambda (1)}\), and it also simplified the exposition and the proofs in this section to do so.
\end{remark}

\section{Concluding the proof of the main theorems}\label{sec:proofoftheorems}
\subsection{Estimates for \(I_{\gamma}(\alpha)\)}
In the next two lemmas we show that lower bounds for \(\dim {}^* \! \nu_t\) imply upper bounds for the quantities \(I_{\gamma}(\alpha,\epsilon,n)\) and \(I_{\gamma}(\alpha)\).

\begin{lemma}\label{lem:tightboundsforQgamma}
 For all \(t_c'<t \leq 0\) and \(0<D<\dim_F{}^*\!\nu_t\) there exists \( K>0\) such that for all \(\alpha \in \mathrm{A}\), \(\gamma \in (0,2]\), \(\epsilon>0\), and \(n \in \N\), we have
    \[I_{\gamma}\left(\alpha,\epsilon,n  \right)+t(1-\gamma)\alpha  \leq k(t)- (D+t)\gamma\alpha+\frac{K}{n}+(3+\alpha)\epsilon.\]
\end{lemma}

\begin{proof}
Let \(t_c'<t \leq 0\) and \(0<D<\dim_F{}^*\!\nu_t\). Let \(C>0\) be such that for all \(r>0\),
\begin{equation}\label{eqn:dimboundapplied}
    \sup_{w \in \BP }{}^*\!\nu_t(B(w,r))  <Cr^{D}.
\end{equation}
We also let \(R_0>0\) be such that 
\begin{equation*}
    \sup_{w \in \BP}{}^*\!\nu_t(B(w,R_0))  <\frac{1}{2}.
\end{equation*}

Fix some \(x\in \BP\). For all \(n \in \N \) and \(g \in E_{\gamma}(\alpha,\epsilon,n,x)\) we have
\[d_{\BP}(x^\perp, \upsilon_+(g^*))=d_{\BP}(x, \omega_-(g)) \leq e^{-n (\gamma-\epsilon)\alpha}.\]
Moreover, using Lemma \ref{lem:BQlemma} twice, for all \(w=\R v \in \BP\), \(v \in S^1\), such that \(d_{\BP}(w,\omega_-(g^*)) \geq R_0\),
\begin{align} 
    d_{\BP}(g^* w, \upsilon_+(g^*)) &=  d_{\BP}(g^* w, \omega_-((g^*)^{-1})) \nonumber \\
    &\leq \frac{|(g^*)^{-1} g^* v|}{|g| \left|g^* v \right|} \nonumber \\
     &\leq \frac{1}{|g|^2 R_0 }  \nonumber \\
     &\leq R_0^{-1} e^{-2n(\alpha-\epsilon)}. \nonumber 
\end{align}
In particular, for all such \(w \in \BP\),
\begin{align}
    d_{\BP}(g^* w,x^\perp) &\leq d_{\BP}(g^* w,\upsilon_+(g^*))+d_{\BP}(\upsilon_+(g^*),x^{\perp}) \\
    &\leq  (1+R_0^{-1})e^{-n(\gamma\alpha-(2+\alpha)\epsilon)} \nonumber \\
    &\leq 2R_0^{-1}e^{-n(\gamma\alpha-(2+\alpha)\epsilon)}. \label{eqn:R^-1e2nalphaminusepsilon}
\end{align}
Using that \(({}^* \! P_t)_{\Cdot} {}^*\! \nu_t= e^{k(t)} {}^* \! \nu_t\) and Tonelli's theorem, it follows that 
\begin{align}
   {}^*\!\nu_t(B(x^\perp,2R_0^{-1}e^{-n(\gamma\alpha-(2+\alpha)\epsilon)})) &= e^{-nk(t)}  \int \int |g^*w|^t 1_{B(x^\perp,2R_0^{-1}e^{-n(\gamma\alpha-(2+\alpha)\epsilon)} )} \left(g^* w \right) \diff{}^*\!\nu_t(w) \diff \mu^n(g) \nonumber  \\
   &\geq e^{-nk(t)} e^{nt(\alpha+\epsilon)} \int_{E_{\gamma}(\alpha,\epsilon,n,x)} \int    1_{B(x^{\perp},2R_0^{-1}e^{-n(\gamma\alpha-(2+\alpha)\epsilon)} )} \left(g^* w \right) \diff{}^*\!\nu_t(w) \diff \mu^n(g) \nonumber \\
    &\geq \frac{1}{2} e^{-nk(t)} e^{nt(\alpha+\epsilon)} \int_{E_{\gamma}(\alpha,\epsilon,n,x)}  \diff \mu^n(g), \label{eqn:lastoneinsomealign}
\end{align}
where the last inequality is because all but at most a ball of radius \(R_0\) is mapped by \(g^*\) into \(B(x^\perp,2R_0^{-1} e^{-n(\gamma\alpha-(2+\alpha)\epsilon)})\) by (\ref{eqn:R^-1e2nalphaminusepsilon}). By (\ref{eqn:dimboundapplied}), \({}^*\!\nu_t(B(x^\perp,2R_0^{-1}e^{-n(\gamma\alpha-(2+\alpha)\epsilon)})) \leq 2CR_0^{-1} e^{-nD(\gamma\alpha-(2+\alpha)\epsilon)}\). Hence,
\[\frac{1}{2} e^{-nk(t)} e^{nt(\alpha+\epsilon)} \mu^n(E_{\gamma}(\alpha,\epsilon,n,x)) \leq 2CR_0^{-1} e^{-nD(\gamma\alpha-(2+\alpha)\epsilon)}. \]
Since \(x \in \BP\) was arbitrary, we therefore have
\[I_{\gamma}(\alpha,\epsilon,n)-k(t)+t(\alpha+\epsilon) \leq \frac{1}{n} \log 4CR_0^{-1} -D(\gamma\alpha-(2+\alpha)\epsilon).\]
Rearranging then subtracting \(t\gamma\alpha\) from both sides gives
\begin{align*}
    I_{\gamma}\left(\alpha,\epsilon, n \right)+t(1-\gamma)\alpha \leq k(t)- (D+t)\gamma\alpha +\frac{1}{n}\left(\log 4CR_0^{-1} \right) -t\epsilon+(2+\alpha)D\epsilon,
\end{align*}
so 
\[I_{\gamma}\left(\alpha,\epsilon, n \right)+t(1-\gamma)\alpha \leq k(t)-(D+t)\gamma\alpha+\frac{K}{n}+(3+\alpha)\epsilon,\] 
where 
\(K:=\log 4CR_0^{-1}\).
\end{proof}

The following is a straightforward consequence of Lemma \ref{lem:tightboundsforQgamma}. Note the case \(\gamma=0\) follows by continuity of \(\gamma \mapsto I_{\gamma}(\alpha)\) (Lemma \ref{lem:elementaryQgammainequality}).

\begin{lemma}\label{lem:boundsforQgamma}
For \(t \in (t_c',0]\), we have for all \(\alpha \in \mathrm{A}\) and \(\gamma \in [0,2]\),
\[I_\gamma(\alpha)+t(1-\gamma)\alpha \leq k(t)-(t+\dim_F{}^*\!\nu_t)\gamma \alpha.\]
\end{lemma}

We are now able to prove:

\begin{lemma}\label{lem:zetat'zetat}
For all \(t \in (t_c',0]\), \(\zeta_t'=\zeta_t\).
\end{lemma}

\begin{proof}
    Considering (\ref{eqn:zetat'improved}) it is clear that \(\zeta_t' \leq \zeta_t\) for all \(t \in (t_c',0]\). For the other direction, first assume that \(t \not= 0\). By Lemmas \ref{lem:USC} and \ref{lem:zetatunique} we can let \(\beta_t>0\) and \(\gamma_t \in (1,2]\) be such that 
    \[I_{\gamma_t}(\beta_t)+t(1-\gamma_t)\beta_t-2\zeta_t' (1-\gamma_t)\beta_t=k(t).\]
    By Lemmas \ref{lem:boundsforQgamma} and \ref{lem:dimlowerbound}, for all \(\alpha \geq 0\) and \(\gamma \in [0,2]\) we have 
    \[ I_\gamma(\alpha)+t(1-\gamma)\alpha \leq k(t) -\zeta_t' \gamma \alpha.\]
    Hence,
    \[ 2\zeta_t' (1-\gamma_t)\beta_t \leq -\zeta_t' \gamma_t \beta_t,\]
    so \(\gamma_t \geq 2\), i.e. \(\gamma_t =2\). We have shown for \(t \not= 0\) that 
    \[I_2(\beta_t)-t\beta_t+2\zeta_t'\beta_t=k(t),\]
    which implies that \(\zeta_t \leq \zeta_t'\). By continuity of \(t \mapsto \zeta_t\) and \(t \mapsto \zeta_t'\) (Lemma \ref{lem:continuity}), this further implies that \(\zeta_0 \leq \zeta_0'\). 
    \end{proof}

\begin{lemma}\label{lem:tc'=tc}
    \(t_c'=t_c\).
\end{lemma}

\begin{proof}
    Again it is clear that \(t_c \leq t_c'\), so we show the other inequality. If \(k(t_c')=I_0(0) \equiv I_2(0)\) we have \(t_c \geq t_c'\) immediately, so assume \(I_0(0)<k(t_c')\) (i.e. \(t_c'>t_0\)). 
    
    Let \(t_n \downarrow t_c'\). By Lemmas \ref{lem:zetatunique} and \ref{lem:USC}, for each \(n \in \N\) we can let \(\beta_{t_n}>0\) be such that 
    \[ I_2(\beta_{t_n})-t_n \beta_{t_n}+2\zeta_{t_n} \beta_{t_n} =k(t_n). \]
    Let \(\beta_{t_c'} \in \mathrm{A}\) be a limit point of the sequence \((\beta_{t_n})_{n \in \N}\) and without loss of generality assume that \(\beta_{t_n} \rightarrow \beta_{t_c'}\). By Lemma \ref{lem:USC},
    \begin{align*}
        k(t_c')= \lim_{n \rightarrow \infty} k(t_n) &= \lim_{n \rightarrow \infty} I_2(\beta_{t_n})-t_n \beta_{t_n}+2\zeta_{t_n} \beta_{t_n} \\
        &\leq  I_2(\beta_{t_c'})-t_c' \beta_{t_c'}+ 2(\varlimsup_{n \rightarrow \infty} \zeta_{t_n}) \beta_{t_c'}.
    \end{align*}
    Hence, to finish the proof of the lemma, it suffices to show that \(\varlimsup_{n \rightarrow \infty} \zeta_{t_n}=0\). 
    
    Combining Lemmas \ref{lem:dimlowerbound}, \ref{lem:boundsforQgamma}, and \ref{lem:zetat'zetat} we have for all \(t>t_c'\), all \(\alpha \in \mathrm{A}\), and \(\gamma \in [0,2] \),  
    \[I_{\gamma}(\alpha)+t(1-\gamma)\alpha \leq k(t)-\zeta_t \gamma \alpha.\]
    By continuity of \(k(t)\) and \(\zeta_t\), we further have that for all \(\alpha \in \mathrm{A}\) and \(\gamma \in [0,2] \),  
    \begin{equation}\label{eqn:boundwithlimsupzetat}
        I_{\gamma}(\alpha)+t_c'(1-\gamma)\alpha \leq k(t_c')-(\varlimsup_{n \rightarrow \infty} \zeta_{t_n}) \gamma \alpha.
    \end{equation}
    Since we are assuming that \(I_0(0)<k(t_c')\), for some \(\underline{\alpha}>0\) small enough we have \(I_0(0,\underline{\alpha})-t_c'\underline{\alpha}<k(t_c')\), so
    \[\sup_{\alpha \leq \underline{\alpha}/2}  \sup_{\gamma \in [1,2]} \left\{I_{\gamma}(\alpha)+t_c'(1-\gamma)\alpha \right\}  \leq I_0(0,\underline{\alpha})-t_c' \underline{\alpha}<k(t_c').\]
    If \(\varlimsup_{n \rightarrow \infty} \zeta_{t_n}>0\), by (\ref{eqn:boundwithlimsupzetat}) we also have
    \begin{align*}
        \sup_{\alpha \geq \underline{\alpha}/2}  \sup_{\gamma \in [1,2]} \left\{I_{\gamma}(\alpha)+t_c'(1-\gamma)\alpha \right\} &\leq k(t_c')- \left(\varlimsup_{n \rightarrow \infty} \zeta_{t_n}\right) \underline{\alpha}/2 \\
        &<k(t_c'),
    \end{align*}
    which implies
    \[\sup_{\alpha \in \mathrm{A}}  \sup_{\gamma \in [1,2]} \left\{I_{\gamma}(\alpha)+t_c'(1-\gamma)\alpha \right\} <k(t_c'). \]
    However, this contradicts the definition of \(t_c'\) since this would imply for some \(t<t_c'\),
    \[\sup_{\alpha \in \mathrm{A}} \sup_{\gamma \in [1,2]} \{I_{\gamma}(\alpha)+t(1-\gamma)\alpha\} <k(t).\]
    In particular, we must have \( \varlimsup_{n \rightarrow \infty} \zeta_{t_n}=0\).
\end{proof}

\subsection{Analyticity of \(k(t)\) on \((t_c,0)\)}
 Let \(\overline{C^{\zeta}}(\BP)\) denote the complexification of \(C^{\zeta}(\BP)\), that is, 
    \[\overline{C^{\zeta}}(\BP):= \{f+ig: f,g \in C^{\zeta}(\BP) \}. \]
    The function
    \[ |f+ig |_{\overline{C^{\zeta}}(\BP)}:=\sup\{ |\cos(\theta) f+\sin(\theta)g|_{\zeta}: \theta \in [0,2\pi]\}\]
    defines a norm on \(\overline{C^{\zeta}}(\BP)\), making \( (\overline{C^{\zeta}}(\BP), |\cdot |_{\overline{C^{\zeta}}(\BP)}) \) a Banach space. We define the operators \(\overline{P_z}: \overline{C^{\zeta}}(\BP) \rightarrow \overline{C^{\zeta}}(\BP)\) by
    \[ \overline{P_z}(F)(x)=\int e^{z \log |gx|} F(gx) \diff \mu(g), \quad \forall z\in \C. \]
    Note that for \(t \in \R\), \(\overline{P_t}\) acts on \(\overline{C^{\zeta}}(\BP)\) by
    \[\overline{P_t}(f+ig)=P_t f+i P_t g.\]
    This combined with Proposition \ref{prop:spectralgap}  immediately implies the following lemma.

    \begin{lemma}\label{lem:overlinePtspectralgap}
    For all \(t_c<t \leq 0\) and \(0<\zeta<\zeta_t'\), we can write
    \[ \overline{P_t}:= e^{k(t)} (\nu_t \otimes  h_t+ \overline{S_t}),\]
    where \((\nu_t \otimes  h_t) (f+ig):=(\nu_t(f)+i\nu_t(g)) h_t\) and \(\overline{S_t}:\overline{C^{\zeta}}(\BP) \rightarrow \overline{C^{\zeta}}(\BP)\) satisfies \(\rho(\overline{S_t})<1\).
    \end{lemma}
    
    Proposition V.4.1 in \cite{BL85} proves that the family \(\{\overline{P_z}:z \in \C \}\) is an analytic family of bounded operators on \(\overline{C^{\zeta}}(\BP)\) for any \(\zeta \in (0,1]\) (it is easily seen from the proof that the \(\eta\) in the proposition can be taken to be \(\infty\) when \(\mu\) is compactly supported). Hence, by perturbation theory (see \cite[Lemma V.3.2]{BL85}), we immediately have from Lemma \ref{lem:overlinePtspectralgap}:
    
\begin{prop}\label{prop:ktanalytic}
The map \(t \mapsto k(t)\) is analytic on \((t_c,0)\).    
\end{prop}

\begin{remark}
Combined with \cite[Theorem 1.1]{GP16} and \cite[Theorem V.4.3]{BL85}, this implies that \(k(t)\) is analytic on \((t_c,\infty)\).
\end{remark}

\subsection{Proof of Theorems \ref{theo:spectralgap}-\ref{theo:Igammaestimates}}
\begin{proof}[Proof of Theorem \ref{theo:spectralgap}]
Theorem \ref{theo:spectralgap} follows immediately from Proposition \ref{prop:ktanalytic} and combining Proposition \ref{prop:spectralgap} with Lemmas \ref{lem:zetat'zetat} and \ref{lem:tc'=tc}. 
\end{proof}

\begin{proof}[Proof of Theorem \ref{theo:measure} except for \(t =0\)]
We now prove Theorem \ref{theo:measure} in the case \(t \in (t_c,0)\). We postpone the proof of the \(t=0\) case until after we have proved Theorem \ref{theo:Igammaestimates}.  By Lemmas \ref{lem:dimlowerbound} and \ref{lem:zetat'zetat} we have \(({}^* \! P_t)_{\Cdot} {}^* \! \nu_t=e^{k(t)} {}^* \! \nu_t\) and \(\dim_F {}^* \! \nu_t \geq -t+\zeta_t\). By upper semi-continuity of \(I_2(\cdot)\) and Lemma \ref{lem:zetatunique} there exists \(\beta_t>0\) such that 
\begin{equation}\label{eqn:alphawhichattainsupinQ21}
    I_2(\beta_t)-t\beta_t +2\zeta_t \beta_t= k(t) .
\end{equation}
By Lemma \ref{lem:boundsforQgamma} we have \(I_2(\beta_t)-t\beta_t  \leq k(t) -2(t+\dim_F{}^* \! \nu_t)\beta_t \),
which, combined with (\ref{eqn:alphawhichattainsupinQ21}), implies that \(\dim_F{}^* \! \nu_t \leq -t+ \zeta_t\). Hence, \(\dim_F{}^* \! \nu_t = -t+ \zeta_t\).

The statement that 
\[ h_t(x) \equiv  C_{\lambda,t}^{-1} \lambda(h_t) \int |\innerproduct{x}{w}|^t \diff {}^{*} \! \nu_t(w)\]
is in Lemma \ref{lem:dimlowerbound}. Since \(\dim_F{}^* \! \nu_t=-t+ \zeta_t\), by Propositions \ref{prop:hisholder} and \ref{prop:Holdergivesdimensionbound} we have \(\sup\{\zeta>0:h_t \in C^{\zeta}(\BP)\}=\zeta_t\).

To show the uniqueness statement in the theorem, suppose that \({}^* \! \nu_t' \in \MCP(\BP)\) satisfies \(({}^* \! P_t)_{\Cdot} {}^* \! \nu_t'=e^{k(t)} {}^* \! \nu_t'\) and \(\dim_F {}^* \! \nu_t'>-t\). Then, by Proposition \ref{prop:hisholder}, 
\[h_t'(x):= C_{\lambda,t}^{-1} \lambda(h_t) \int |\innerproduct{x}{w}|^t \diff {}^* \! \nu_t'(w)\]
is H\"older continuous, by Lemma \ref{lem:Ptdelta*Pt} we have that \(P_t h_t'=e^{k(t)}h_t'\), and by Tonelli's theorem  we also have \({}^*\! \nu_t(h_t')=1\). However, for all \(0<\zeta<\zeta_t\), \(h_t\) is the unique H\"older continuous function in \(C^{\zeta}(\BP)\) satisfying \(P_t h_t=e^{k(t)}h_t\) and \(\nu_t(h_t)=1\), so \(h_t'=h_t\). Proposition \ref{prop:Inu1=Inu2impliesnu1=nu2} therefore implies \( {}^* \! \nu_t'=  {}^* \! \nu_t\).
\end{proof}

\begin{proof}[Proof of Theorem \ref{theo:Igammaestimates}]
For \(t \in (t_c,0)\) this follows from Lemma \ref{lem:boundsforQgamma} and the equality \(\dim_F {}^* \! \nu_t = -t+\zeta_t\) which we have just proved. For \(t=0\) this follows by continuity of \(\zeta_t\) and \(k(t)\).
\end{proof}

\subsection{Frostman dimension of \({}^* \! \nu_0\)}\label{subsec:Frostmandimensionofnu0}
Let \({}^* \! \nu_0 \in \MCP(\BP)\)  be the unique \({}^* \! \nu \in \MCP(\BP)\) satisfying \(({}^* \! P_0)_{\Cdot}{}^* \! \nu={}^* \! \nu\), which exists by the classical work of Furstenberg \cite{Fur63}. We complete the proof of Theorem \ref{theo:measure} by proving the following proposition.

\begin{prop}\label{prop:dimnu0}
    \(\dim_F{}^* \! \nu_0 =\zeta_0\).
\end{prop}

Let us give some intuition behind the proof. As in previous work, we require the spectral gap for both the operators \(P_0\) and \({}^* \! P_0\) (cf. \cite[Theorem 13.1]{BQ16}). Roughly speaking, the spectral gap for \(P_0\) is used to control the \(\mu^n\) measure of how close \(\omega_+(g) \equiv \upsilon_+(g^*)\) can be to particular \(x^{\perp} \in \BP\) (hence, the sets \(E_{\gamma}(\alpha,\epsilon,n,x)\) and quantities \(I_{\gamma}(\alpha)\) are relevant). On the other hand, the spectral gap for \({}^* \! P_0\) is used to control, in a \(\mu^n\) sense, how close \(\omega_+(g_1 g_2 \ldots g_n)\) is to the limits \(\xi_+((g_n)_{n \in \N})\) which exist \(\mu^{\N}\)-almost surely by Oseledets' theorem. We note (but do not explicitly use) that \({}^* \! \nu_0\) is equal to the projection of \(\mu^{ \N}\) onto \(\BP\) according to \( \xi_+\).

By the spectral gap result of Le Page \cite[Theorem 2.5]{BL85}, we can let \(0<\zeta<1\), \(C>0\) and \(0<r<1\) be such that \(| {}^* \! P_0^n(f)- {}^* \! \nu_0(f)|_{\zeta} \leq C r^n |f|_{\zeta}\) for all \(f \in C^{\zeta}(\BP)\) and \(n \in \N\). Then, for all \(f \in C^{\zeta}(\BP)\) and all \(n \in \N\), \(|^* \! P_0^n \lambda(f) -{}^* \! \nu_0(f)| \leq Cr^n |f|_{\zeta}\). This is capturing how far \(\omega_+(g_1 g_2 \ldots g_n)\) is from the limits \(\xi_+((g_n)_{n \in \N})\). 

The idea of the proof of Proposition \ref{prop:dimnu0} is to let \(\Delta\) be small enough, relative to \(r\), so that \(\sup_{x \in \BP} {}^* \! \nu_0(B(x,e^{-n \Delta})) \lessapprox e^{n \sup_{\alpha \in \mathrm{A}} I_{\Delta/\alpha}(\alpha)}\), in particular making any errors coming from  \(\omega_+(g_1 g_2 \ldots g_n)\) being far from \(\xi_+((g_n)_{n \in \N})\) negligible. Thus, we can invoke Theorem \ref{theo:Igammaestimates}, which is already proved, to show the more difficult lower bound in the proposition.

We first prove two technical lemmas. We note we will also assume that \(\Delta<2\underline{\alpha}\), where \(\underline{\alpha}>0\) is chosen sufficiently small so that \(g \in E(0,\underline{\alpha},n)\) can be controlled appropriately. Since \(\Delta <2\underline{\alpha}\), for \(\alpha \geq \underline{\alpha}\) we always have \(\Delta/\alpha<2\).

\begin{lemma}\label{lem:technicallemma1}
Let \(\underline{\alpha}>0\), \(0<\Delta<2\underline{\alpha}\), and \(\epsilon>0\). For all \(\alpha \geq \underline{\alpha}\) and \(n \in \N\), 
    \[  \sup_{x \in \BP} \int_{E_{\Delta/\alpha}(\alpha,3\epsilon,n,x)} \int 1_{B(x^{\perp},2e^{-n\Delta})}(g^* w)  \diff \lambda(w)  \diff \mu^n(g) \leq e^{n I_{\Delta/\alpha}(\alpha,3\epsilon,n)}  \] 
\end{lemma}

\begin{proof}
This is straightforward as 
\begin{align*}
    \sup_{x \in \BP} \int_{E_{\Delta/\alpha}(\alpha,3\epsilon,n,x)} \int 1_{B(x^{\perp},2e^{-n\Delta})}(g^* w)  \diff \lambda(w)  \diff \mu^n(g)  & \leq \sup_{x \in \BP} \int_{E_{\Delta/\alpha}(\alpha,3\epsilon,n,x)} 1  \diff \mu^n(g)  \\ &= e^{nI_{\Delta/\alpha}(\alpha,3\epsilon,n)}.
\end{align*}
\end{proof}

For the next lemma, we consider \(g \in \widetilde{E}_{\gamma}(\alpha,\epsilon,n,x)\) where \(\gamma<\Delta/\alpha\), so \(d_{\BP}(x^{\perp}, \upsilon_+(g^*)) \approx e^{-n\gamma \alpha} \gg e^{-n \Delta}\). By the heuristic before Lemma \ref{lem:umapsclosetoupsilon+}, we would expect that a set of Lebesgue measure not more than \(\approx e^{-n(2-\gamma)\alpha}\) to be mapped into \(B(x^{\perp},2e^{-\Delta n})\), since \(x\) with \(d_{\BP}(x, \omega_-(g^*)) \gtrapprox e^{-n(2-\gamma)\alpha} \) will have \(d_{\BP}(g^* x, \upsilon_+(g^*)) \lessapprox  e^{-n \gamma \alpha}\). In the proof we make this sketch rigorous.

\begin{lemma}\label{lem:technicallemma2}
Let \(\underline{\alpha}>0\), \(0<\Delta<2\underline{\alpha}\), and \(\epsilon>0\). For \(n \in \N\) sufficiently large, we have that for any \(\alpha \geq \underline{\alpha}\) and \(\gamma<\Delta/\alpha-2\epsilon\), 
     \[ \sup_{x \in \BP} \int_{\widetilde{E}_{\gamma}(\alpha,\epsilon,n,x)} \int 1_{B(x^{\perp},2e^{-n\Delta})}(g^* w)  \diff \lambda(w)  \diff \mu^n(g) \leq 2e^{nI_{\gamma}(\alpha,\epsilon, n)} e^{-n((2-\gamma) \alpha-(3+\alpha)\epsilon)}. \]
\end{lemma}
\begin{proof}
    Let \(x \in \BP\) and \(g \in \widetilde{E}_{\gamma}(\alpha,\epsilon,n,x)\). If \(w=\R v \in \BP\), \(v \in S^1\), is such that
    \(d_{\BP}(w, \omega_-(g^*) ) \geq e^{-n((2-\gamma)\alpha-(3+\alpha)\epsilon)}\), then using Lemma \ref{lem:BQlemma} twice, 
\begin{align*} 
    d_{\BP}(g^* w, \upsilon_+(g^*)) &=  d_{\BP}(g^* w, \omega_-((g^*)^{-1})) \nonumber \\
    &\leq \frac{|(g^*)^{-1} g^* v|}{|g| \left|g^* v \right|} \nonumber \\
     &\leq \frac{e^{n((2-\gamma)\alpha-(3+\alpha)\epsilon)}}{|g|^2 }  \nonumber \\
     &\leq e^{-n(\gamma\alpha+(1+\alpha)\epsilon)}.
\end{align*}
Assuming that \(n \in \N\) is large enough so that \(e^{-n \epsilon }\leq \frac{1}{2}\) and \(e^{n\epsilon \underline{\alpha}} \geq 4\), we therefore have
\begin{align*}
    d_{\BP}(g^* w, x^{\perp}) &\geq  d_{\BP}(x^{\perp},\upsilon_+(g^*))-d_{\BP}(g^* w,\upsilon_+(g^*) )  \\
    &= d_{\BP}(x,\omega_-(g)) -  d_{\BP}(g^* w,\upsilon_+(g^*) ) \\
    & \geq e^{-n(\gamma+\epsilon)\alpha} - e^{-n(\gamma\alpha+(1+\alpha)\epsilon)} \\
    &\geq \frac{1}{2} e^{-n(\gamma+\epsilon)\alpha} \\
    & \geq \frac{1}{2} e^{-n(\Delta-\epsilon\alpha)} \\
    &\geq 2 e^{-n\Delta},
\end{align*}
where we have used in the penultimate inequality that \(\gamma<\frac{\Delta}{\alpha}-2\epsilon\).

Thus, we have shown that the set \(\{w \in \BP: d_{\BP}(w, \omega_-(g^*) ) \geq e^{-n((2-\gamma)\alpha-(3+\alpha)\epsilon)}\}\) is mapped by \(g^*\) outside of \(B(x^{\perp},2e^{-n\Delta})\). It follows that
\[\int 1_{B(x^{\perp},2e^{-n\Delta})}(g^* w)  \diff \lambda(w)  \leq 2  e^{-n((2-\gamma)\alpha-(3+\alpha)\epsilon)},\]
so
\[\int_{\widetilde{E}_{\gamma}(\alpha,\epsilon, n, x)}  \int 1_{B(x^{\perp},2e^{-n\Delta})}(g^* w)  \diff \lambda(w)  \diff \mu^n(g) \leq 2e^{nI_{\gamma}(\alpha,\epsilon, n)} e^{-n((2-\gamma) \alpha-(3+\alpha)\epsilon)}.\]
\end{proof}

\begin{proof}[Proof of Proposition \ref{prop:dimnu0}]
The proof that \(\dim_F{}^* \! \nu_0 \leq \zeta_0 \) follows in the same way as in the \(t<0\) case. To show the other direction, let \(0<\zeta<1\), \(C>0\) and \(0<r<1\) be such that for all \(f \in C^{\zeta}(\BP)\) and \(n \in \N\),
\[|{}^* \!P^n_0 f-{}^* \! \nu_0(f) |_{\zeta} \leq C r^n |f|_{\zeta}.\]
This exists by the classical results of Le Page \cite{Pag82}. For all \(f \in C^{\zeta}(\BP)\) and \(n \in \N\), we further have
\begin{align*}
    |({}^* \!P_0^n)_{\Cdot} \lambda(f) -{}^* \! \nu_0(f)| 
    &=| \lambda({}^* \!P_0^n f -{}^* \! \nu_0(f)) | \\
    &\leq |({}^* \!P_0^n) f-{}^* \! \nu_0(f) |_{\infty} \\ &\leq  C r^n |f|_{\zeta} .
\end{align*}
Let \(\Delta>0\). We will set some conditions on \(\Delta\) shortly. For \(x^{\perp} \in \BP\) we let
\[f_{x^{\perp},n}(w) = \begin{cases}
    1, & w \in B(x^{\perp},e^{-n\Delta}) \\
    2-e^{n \Delta} d_{\BP}(w,x^{\perp}), & e^{-n\Delta} \leq d_{\BP}(w,x^{\perp}) \leq 2e^{-n\Delta} \\
    0, & \mathrm{otherwise}.
\end{cases}\]
We have \(|f_{x^{\perp},n}|_{\zeta} \leq C'e^{n\zeta \Delta } \) for some constant \(C'>0\) independent of \(x\) and \(n \in \N\). Thus, for all \(x^{\perp} \in \BP\) and \(n \in \N\),
\begin{equation}\label{eqn:nu0lambdainequalitytwoterms}
     {}^* \! \nu_0 (B(x^{\perp},e^{-n\Delta})) \leq  ({}^* \!P_0^n)_{\Cdot} \lambda(f_{x^{\perp},n})+   C C' r^n e^{n\zeta \Delta } .
\end{equation}

Observe that 
\begin{align}
     ({}^* \!P_0^n )_{\Cdot} \lambda (f_{x^{\perp},n}) &=\int \int f_{x^{\perp},n}(g^* w) \diff \mu^n(g) \diff \lambda(w) \nonumber \\ 
    &\leq \int \int 1_{B(x^{\perp},2e^{-n\Delta})}(g^* w) \diff \mu^n(g) \diff \lambda(w) \nonumber \\
    &= \int \int 1_{B(x^{\perp},2e^{-n\Delta})}(g^* w)  \diff \lambda(w)  \diff \mu^n(g). \label{eqn:*P0nlambdainequality} 
\end{align}
We let \(\underline{\alpha}>0\) be such that 
\[I_0(0,\underline{\alpha}) \leq 0,\]
which exists because \(I_0(0)<0\). We assume \(\Delta\) is small enough so that \(\Delta/\underline{\alpha}<2 \), \(I_0(0,\underline{\alpha})/\Delta<-1\), and \(\frac{\log r}{\Delta}+\zeta<-1 \). Let \(\epsilon>0\). By an argument similar to the one in Lemma \ref{lem:gammaiscover}, for any \(\alpha \geq \underline{\alpha}\), \(\epsilon>0\), \(n \in \N\), and \(x \in \BP\),
\[E(\alpha,\epsilon,n)= E_{\Delta/\alpha}(\alpha,3\epsilon,n,x) \cup \bigcup_{\substack{\gamma \in \Gamma(\epsilon), \\ \gamma < \frac{\Delta}{\alpha}-2\epsilon }} \widetilde{E}_{\gamma} (\alpha,\epsilon,n,x).\]
Hence,
\begin{align*}
    \int \int 1_{B(x^{\perp},2e^{-n\Delta})}(g^* w)  \diff \lambda(w)  \diff \mu^n(g) &\leq \mu^n(E(0,\underline{\alpha},n)) \\ & +  \sum_{\substack{\alpha \in \mathrm{A}(\epsilon), \\ \alpha \geq \underline{\alpha}  }}  \int_{E_{\Delta/\alpha}(\alpha,3\epsilon,n,x)} \int 1_{B(x^{\perp},2e^{-n\Delta})}(g^* w)  \diff \lambda(w)  \diff \mu^n(g)   \\
&+ \sum_{\substack{\alpha \in \mathrm{A}(\epsilon), \\ \alpha \geq \underline{\alpha}  }} \sum_{\substack{\gamma \in \Gamma(\epsilon), \\ \gamma < \frac{\Delta}{\alpha}-2\epsilon}}  \int_{\widetilde{E}_{\gamma}(\alpha,\epsilon,n,x)} \int 1_{B(x^{\perp},2e^{-n\Delta})}(g^* w)  \diff \lambda(w)  \diff \mu^n(g).
\end{align*}
Combining this with (\ref{eqn:nu0lambdainequalitytwoterms}), (\ref{eqn:*P0nlambdainequality}) and Lemmas \ref{lem:technicallemma1} and \ref{lem:technicallemma2}, we obtain 
\begin{align*}
    \dim_F {}^* \! \nu_0 = \varliminf_{n \rightarrow \infty} \frac{-\log \sup_{x \in \BP} {}^* \! \nu_0 (B(x^{\perp},2e^{-n\Delta}))}{- \log( 2 e^{-n\Delta})} \geq -\frac{1}{\Delta} \max \{ &I_0(0,\underline{\alpha}), \\  & 
     \max_{\substack{\alpha \in \mathrm{A}(\epsilon) \\ \alpha \geq \underline{\alpha}}} I_{\Delta/\alpha}(\alpha,3\epsilon), \\
    & \max_{\substack{\alpha \in \mathrm{A}(\epsilon) \\ \alpha \geq \underline{\alpha}}} \max_{\substack{\gamma \in \Gamma(\epsilon), \\ \gamma < \frac{\Delta}{\alpha}-2\epsilon }}  I_{\gamma}(\alpha,\epsilon) -(2-\gamma)\alpha+(3+\alpha)\epsilon , \\
    &\log r +\zeta \Delta \}.
\end{align*}
By letting \(\epsilon \rightarrow 0\) and using Lemma \ref{lem:USC}, this further implies that
\begin{align*}
    \dim_F {}^* \! \nu_0 \geq -\frac{1}{\Delta} \max\{ I_0(0,\underline{\alpha}),
     \sup_{\alpha \geq \underline{\alpha} } I_{\Delta/\alpha}(\alpha),
    \sup_{ \alpha \geq \underline{\alpha} } \sup_{\gamma < \Delta/\alpha} I_{\gamma}(\alpha)-(2-\gamma)\alpha, \log r+\zeta \Delta \}. 
\end{align*}

By our assumptions on \(\Delta\), \( I_0(0,\underline{\alpha})/\Delta<-1\) and \((\log r +\zeta \Delta)/\Delta<-1 \). Also, by Theorem \ref{theo:Igammaestimates}, which is already proved, for \(\alpha \geq \underline{\alpha}\) we have
\[ I_{\Delta/\alpha}(\alpha)  \leq -\zeta_0 \Delta,\]
so 
\[\frac{1}{\Delta} \sup_{\alpha \geq \underline{\alpha} } I_{\Delta/\alpha}(\alpha) \leq -\zeta_0.\]
Theorem \ref{theo:Igammaestimates} also implies for \(\gamma \alpha \leq \Delta\),
\[ I_{\gamma} (\alpha) \leq-\zeta_0 \gamma \alpha,\]
so
\begin{align*}
    I_{\gamma} (\alpha)-(2-\gamma)\alpha &\leq  -\zeta_0 \gamma \alpha -(2-\gamma)\alpha \\
    &=-2\alpha +(1-\zeta_0)\gamma \alpha \\
    &\leq -2\alpha+(1-\zeta_0) \Delta \\
    &= (\Delta-2\alpha)-\zeta_0 \Delta \\
    &\leq (\Delta-2\underline{\alpha}) -\zeta_0 \Delta \\
    &< - \zeta_0 \Delta.
\end{align*}
That is,
\[\frac{1}{\Delta}\sup_{ \alpha \geq \underline{\alpha} } \sup_{\gamma < \Delta/\alpha} I_{\gamma}(\alpha)-(2-\gamma)\alpha \leq -\zeta_0\]
Putting this together, we have shown
\(\dim_F{}^* \! \nu_0 \geq \zeta_0.\)
\end{proof}

\subsection*{Acknowledgements}
I am very grateful to Mark Pollicott, \c{C}a\u{g}ri Sert, and Micha{\l} Rams for helpful discussions. I would particularly like to thank \c{C}a\u{g}ri for pointing me towards several related works that I was previously unaware of. This work was supported by the EPSRC grant EP/W033917/1.

\appendix
\section{Riesz potentials on the torus}\label{section:fourier}
For \(\nu \in \MCP(\BP)\) we define \(\MCI_{t,\nu}: \BP \rightarrow [0,\infty]\) by
 \[\MCI_{t,\nu}(x):= \int d_{\BP}(x,y)^t \diff \nu(y).\]
The purpose of this appendix is to prove Propositions \ref{prop:hisholder}, \ref{prop:Holdergivesdimensionbound}, and \ref{prop:Inu1=Inu2impliesnu1=nu2}.

The following proposition generalises the result for Riesz potentials on \(\R^d\) in \cite{Car63} (see also \cite[Theorem A]{Wal66}). 

\begin{prop}\label{prop:hisholder}
For \(-1<t<0\) and \(\nu \in \MCP(\BP)\) suppose that \(\dim_F \nu  >-t\). Then, for all \(0<\zeta<\min\{1,\dim_F \nu +t \}\), \(\MCI_{t,\nu}:\BP \rightarrow \R\) is \(\zeta\)-H\"older continuous.
\end{prop}

The proof in \cite{Car63} does not directly generalise to \((\BP,d_{\BP})\), so we provide a different proof here. We will use the following inequality.
    
\begin{lemma}\label{lem:triviallemma}
Let \(-1<t<0\) and \(0<\zeta \leq 1\). For \(u,v \in (0,1]\),
\[|u^t-v^t| \leq |u-v|^\zeta (u^{t-\zeta}+v^{t-\zeta}) .\]
\end{lemma}

\begin{proof}
By symmetry we may assume that \(v \geq u\). The inequality for \(v=u\) is trivial, so we assume \(v>u\) and let \(\delta:=v-u\).

First assume that \(\delta \leq u\). By the mean value theorem there exists \(w \in [u,v]\) such that
\[|u^t-v^t|= |t| |w|^{t-1} |u-v|, \]
so 
\[|u^t-v^t| \leq \delta u^{t-1}= \delta^\zeta u^{t-\zeta} \left(\frac{\delta}{u} \right)^{1-\zeta} \leq \delta^\zeta u^{t-\zeta} \leq \delta^\zeta (u^{t-\zeta}+v^{t-\zeta}).\]
Now suppose that \(\delta> u\). We have 
\[|u^t-v^t|=u^t-v^t \leq u^t =u^{\zeta} u^{t-\zeta} \leq \delta^{\zeta} u^{t-\zeta} \leq \delta^{\zeta} (u^{t-\zeta}+v^{t-\zeta}). \]
\end{proof}

\begin{lemma}\label{lem:supintisfinite}
For \(-1<t<0\) and \(\nu \in \MCP(\BP)\) suppose that \(\dim_F \nu  >-t\). Then, for all \(0 \leq \zeta<\dim_F \nu +t\),  
\[\sup_{x \in \BP} \int d_{\BP}(x,y)^{t-\zeta} \diff \nu(y) <\infty.\]
\end{lemma}

\begin{proof}
 Fix \(0 \leq \zeta<\dim_F\nu+t\) and choose \(D\) such that \(
\zeta-t<D<\dim_F\nu\). There exists \(C>0\) such that
\[ \nu(B(x,r))\le Cr^D \qquad \forall x\in\BP, \: \forall r>0.\]
For \(x \in \BP\) and \(n \in \N\), we let
\[E_{x,n}:=\{ y \in \BP: 2^{-n}<d_{\BP}(x,y) \leq 2^{-(n-1)} \}.\]
Then, for any \(x \in \BP\),
\begin{align*}
     \int d_{\BP}(x,y)^{t-\zeta} \diff \nu(y) &=  \sum_{n \in \N} \int_{E_{x,n}} d_{\BP}(x,y)^{t-\zeta} \diff \nu(y) \\
    &\leq C \sum_{n \in \N} 2^{-n(t-\zeta)} 2^{-D(n-1)} \\
    &= C 2^{D} \sum_{n \in \N} 2^{-n(t+D-\zeta)}<\infty,
\end{align*}
so
\[\sup_{x \in \BP} \int d_{\BP}(x,y)^{t-\zeta} \diff \nu(y)<\infty.\]
\end{proof}

\begin{proof}[Proof of Proposition \ref{prop:hisholder}]
We first note that Lemma \ref{lem:supintisfinite} implies that \(\sup_{x \in \BP} \MCI_{t,\nu}(x) <\infty\). Let \(0<\zeta<\min\{1,\dim_F\nu+t\}\). By Lemma \ref{lem:triviallemma}, for all \(x, x' \in \BP\), 
\begin{align*}
    |\MCI_{t,\nu}(x)-\MCI_{t,\nu}(x')| &\leq \int \left|d_{\BP}(x,y)^t-d_{\BP}(x',y)^t \right| \diff \nu(y) \\
    &\leq \int \left|d_{\BP}(x,y)-d_{\BP}(x',y) \right|^{\zeta} \left(d_{\BP}(x,y)^{t-\zeta}+d_{\BP}(x',y)^{t-\zeta} \right)  \diff \nu(y) \\
    &\leq d_{\BP}(x,x')^{\zeta}  \int \left(d_{\BP}(x,y)^{t-\zeta}+d_{\BP}(x',y)^{t-\zeta} \right)  \diff \nu(y),
\end{align*}
where in the last line we are using the reverse triangle inequality. Hence, the proposition follows by Lemma \ref{lem:supintisfinite}.
\end{proof}

We have the following converse of Proposition \ref{prop:hisholder}.

\begin{prop}\label{prop:Holdergivesdimensionbound}
    For \(-1<t<0\), \(0<\zeta<1+t\), and \(\nu \in \MCP(\BP)\), suppose that 
    \[\MCI_{t,\nu}(x):= \int d_{\BP}(x,y)^t \diff \nu(y)\]
    is \(\lambda\)-a.e. equal to a \(\zeta\)-H\"older continuous function \(h:\BP \rightarrow \R\). Then, \(\dim_F \nu \geq -t+\zeta\) and \(\MCI_{t,\nu} \equiv h\).
\end{prop}
The analogous result for Riesz potentials on \(\R^d\) is proved in \cite{Wal66}. The proof here roughly follows the argument in \cite{Wal66}, with some changes due to the topology of \(\BP\). We highlight the main difference where it occurs.

We identify \(\BP\) with the torus \(\R /\pi \Z \) in the natural way. By an abuse of notation, for \(f:\BP \rightarrow \R\) we will also denote the corresponding \(\pi\)-periodic function on \(\R\) by \(f\). Recall \(\lambda\) is Lebesgue measure on \(\BP\) normalised so that \(\lambda(\BP)=1\). We have
\[\int_{\BP} f \diff \lambda=\frac{1}{\pi} \int_{x_0-\frac{\pi}{2}}^{x_0+\frac{\pi}{2}} f(x) \diff x\]
for any \(x_0 \in \R\). \(L^1(\BP)\) is defined to be the set of \(f\) such that 
\[\int_{\BP} |f| \diff \lambda<\infty.\]
For \(f \in L^1(\BP)\) we can write its Fourier series expansion as
\[f(x)=\sum_{n \in \Z} c_n(f) e^{2 i n  x},\]
where the \(n\)th Fourier coefficient is defined as 
\[ c_n(f):=\int_{\BP} f(x) e^{-2 i nx} \diff \lambda(x), \: n \in \Z.\]
Similarly, for \(\nu \in \MCP(\BP)\), the \(n\)th Fourier coefficient is defined to be  
\[c_n(\nu):= \int_{\BP} e^{-2i nx} \diff \nu(x) , \: n \in \Z.\]

Let \(C^{\infty}(\BP)\) denote the space of complex-valued functions which are smooth (with respect to the induced Euclidean metric) on \(\BP\) endowed with its usual Fr\'etchet topology. By a \textit{distribution} on \(C^{\infty}(\BP)\) we mean a continuous linear functional on \(C^{\infty}(\BP)\), and we denote by \(\mathcal{D}'(\BP)\) the space of distributions. Note any \(f \in L^1(\BP)\) and \(\nu \in \MCP(\BP)\) can be identified with a distribution in the usual way. We will say \(D_1,D_2 \in \mathcal{D}'(\BP)\) are \textit{equal in distribution} if \(D_1 \varphi=D_2 \varphi\) for all \(\varphi \in C^{\infty}(\BP)\).

For \(f,g \in L^1(\BP)\) and \(\nu \in \MCP(\BP)\) we let
\[f*g(x):= \int_{\BP} f(x-y) g(y) \diff \lambda(y)\]
and
\[f*\nu(x):=\int_{\BP} f(x-y)\diff \nu(y).\]
We have \(c_n(f * g)=c_n(f) c_n(g)\) and \(c_n(f* \nu)=c_n(f)c_n(\nu)\).

For \(\beta:=t+1\), we fix some \(0<\zeta<\beta\). We define 
\[G_{\beta}(x):=C_{\beta} \sum_{n \not=0 } |n|^{-\beta} e^{2 i nx},\]
where 
\[C_{\beta}:= \pi^{-1/2}\frac{\Gamma(\beta/2)}{\Gamma((1-\beta)/2)}.\]
\begin{prop}[{\cite[Theorem 2.17]{SW71}}]\label{prop:SW71}
For some \(b_1 \in C^{\infty}((-\pi/2,\pi/2])\), 
    \[G_{\beta}(x)=|x|^t+b_1(x), \quad \forall x \in (-\pi/2,\pi/2].\]
\end{prop}

By \(b_1 \in C^{\infty}((-\pi/2,\pi/2])\) we mean that \(b_1\) can be extended to a smooth function (with respect to the Euclidean metric) in an open neighbourhood of \((-\pi/2,\pi/2]\). Similarly, in the following we say \(b_2\) is continuously differentiable on \((-\pi/2,\pi/2]\) if it can be extended to a continuously differentiable function (in the Euclidean metric) in an open neighbourhood of \((-\pi/2,\pi/2]\). 

\begin{lemma}\label{lem:GbetaequalsDt+lipschitz}
For some continuously differentiable \(b_2 : (-\pi/2,\pi/2] \rightarrow \R\),
    \[G_{\beta}(x)=d_{\BP}(0, x)^t+b_2(x), \quad \forall x \in (-\pi/2,\pi/2].\]
In particular, for some Lipschitz \(b_3 :\BP \rightarrow \R\),
\[G_{\beta}(x)=d_{\BP}(0, x)^t+b_3(x), \quad \forall x \in \BP.\]
\end{lemma}

\begin{proof}
Clearly \(d_{\BP}(0, x)^t-  |x|^t =\sin(|x|)^t-|x|^t\) is smooth on \((-\pi,\pi) \setminus\{0\} \supset (-\pi/2,\pi/2] \setminus \{0\}\). We have 
    \begin{align*}
        d_{\BP}(0, x)^t-|x|^t&=\sin(|x|)^t-|x|^t \\
        &=|x|^t \left( \left(1-\frac{|x|^2}{6}+O(|x|^4)\right)^t-1\right) \\
        &=|x|^t \left(-\frac{t}{6}|x|^2+O(|x|^4) \right) 
    \end{align*}
Near \(x=0\) this behaves like \(x \mapsto -\frac{t}{6}|x|^{2+t}\). In particular, its derivative behaves like \(-\frac{t(2+t)}{6}|x|^{1+t}\) which is continuous at 0. Combined with Proposition \ref{prop:SW71}, this proves the first statement with \(b_2(x):=b_1(x)+|x|^t-d_{\BP}(0,x)^t\).

The second statement follows since \( b_2\) is continuously differentiable on an open neighbourhood of \( (-\pi/2,\pi/2]\), so has finite derivative at the boundary. Hence, since \(d_{\BP}\) and the induced Euclidean metric are bi-Lipschitz equivalent on \(\BP\), the function \(b_3:\BP \rightarrow \R\) defined for \(x \in  (-\pi/2,\pi/2]\) by
\[b_3(x):= b_2(x), \]
and extended periodically, is Lipschitz on \((\BP, d_{\BP})\). 
\end{proof}

\begin{lemma}
There exists a \(\zeta\)-H\"older continuous function \(F:\BP \rightarrow \R\) such that
    \(G_{\beta} * \nu =F\)
in distribution.
\end{lemma}

\begin{proof}
    We have 
    \[G_{\beta} * \nu(x) = d_{\BP}(0,\cdot)^t * \nu(x)+b_3*\nu(x).\]
    By assumption, \(d_{\BP}(0,\cdot)^t * \nu\) is equal in distribution to a \(\zeta\)-H\"older continuous function. Moreover, \( b_3\) is Lipschitz, so \(b_3*\nu(x)\) is Lipschitz. Hence, the lemma follows.
\end{proof}

The Fourier multiplier \(T_{\beta}\) is defined by
\[T_{\beta}f(x)=\sum_{n \in \Z} |n|^\beta c_n(f) e^{2 i nx}.\]
Theorem 1.4 \cite{RS16} implies that \(T_{\beta} f \in C^{\infty}(\BP)\) whenever \(f \in C^{\infty}(\BP)\). Recall that \(f \in C^{\infty}(\BP)\) has \(c_n(f)=O(|n|^{-k})\) for any \(k \in \N\). The next two lemmas follow straightforwardly from Fubini's theorem. We omit their proofs.

\begin{lemma}[Zero-mean of \(T_{\beta}\)]\label{lem:zeromeanofLambda}
For any smooth \(f \in C^{\infty}(\BP)\),
    \[\int_{\BP} T_{\beta}f(x) \diff \lambda(x)=0.  \]
\end{lemma}

\begin{lemma}\label{lem:Lambdabetasmooth}
 For any \(f,g \in C^{\infty}(\BP)\),
\[\int_{\BP} f(x) T_{\beta} g(x)  \diff \lambda(x)=\int_{\BP} g(x) T_{\beta} f(x) \diff \lambda(x).\]
\end{lemma}

\begin{lemma}\label{lem:lambdabetaFnu}
\(T_{\beta}F =C_{\beta} (\nu-\lambda)\) in distribution.
\end{lemma}

\begin{proof}
By Proposition \ref{prop:SW71}, \(G_{\beta} \in L^1(\BP)\), so \(c_n(G_{\beta})=C_{\beta}|n|^{-\beta}\) for all \(n \in \N \setminus \{0\}\). Hence, for \(n \not=0\) we have
\[ c_n(F)=c_n(G_{\beta} * \nu) =c_n(G_{\beta}) c_n(\nu)= C_{\beta} |n|^{-\beta} c_n(\nu) ,\]
so \(c_n(\nu)=C_{\beta}^{-1}|n|^{\beta} c_n(F)\). We also have \(c_0(\nu)=1\). Hence, this implies
\[\nu(x)-\lambda(x)= \sum_{n \not= 0} c_n(\nu) e^{2 i n x}= \sum_{n \in \Z}C_{\beta}^{-1} |n|^\beta c_n(F) e^{2 i nx} = C_{\beta}^{-1} T_{\beta} F(x),\] 
where the equalities are in the sense of distribution. We note that this further implies that \(c_n(T_{\beta}F)=|n|^{\beta}c_n(F)\) (a priori it was not clear that \(c_n(T_{\beta}F)\) is well defined).
\end{proof}

Let \(0<R<\pi/2\) and define
\(\phi:\R \rightarrow \R \) by
\[\phi(x):=\begin{cases}
    c \exp \left( \frac{1}{|x/R|^2-1}\right), \quad |x|<R  \\
    0, \quad |x| \geq R,
\end{cases}\]
where \(c\) is chosen so that \( \frac{1}{\pi} \int_{\R} \phi \diff x=1\). For \(0<\epsilon \leq 1\), we let \(\phi_{\epsilon}(x)= \epsilon^{-1}\phi(x/\epsilon)\). We also let \(\varphi_{\epsilon}: (-\pi/2,\pi/2] \rightarrow \R\) be the function defined by \(\varphi_{\epsilon}(x)=\phi_{\epsilon}(x)\). Then, \(\varphi_{\epsilon}\) defines a smooth function on \(\BP\) in the natural way, which we also denote by \(\varphi_{\epsilon}\). For \(f \in L^1(\BP)\) we let \(f_{\epsilon}:=\varphi_{\epsilon} *f\).

\begin{lemma}\label{lem:evans}
\phantom{x}
\begin{enumerate}[label=(\roman*)]
    \item For any \(f \in L^1(\BP)\), \(f_{\epsilon} \in C^{\infty}(\BP)\). 
    \item For any \(f \in C(\BP)\), \(f_{\epsilon} \rightarrow f\) uniformly on \(\BP\).
\end{enumerate}
\end{lemma}

\begin{proof}
Recalling that for \(f:\BP \rightarrow \R\) we are also denoting by \(f\) the corresponding \(\pi\)-periodic function on \(\R\), for \(x \in \R\) we have
\begin{align*}
     f_\epsilon(x)
     &= \frac{1}{\pi} \int_{x-\frac{\pi}{2}}^{x+\frac{\pi}{2}} \varphi_{\epsilon}(x-y) f(y) \diff y \\
    &=\frac{1}{\pi} \int_{x-\frac{\pi}{2}}^{x+\frac{\pi}{2}} \phi_{\epsilon}(x-y) f(y) \diff y \\ 
    &= \frac{1}{\pi} \int_\R \phi_{\epsilon}(x-y) f(y) \diff y.
\end{align*}
Then (i) and (ii) follow from parts (i) and (ii) of \cite[Theorem 1,\S 4.2]{EG92}, respectively.
\end{proof}

\begin{lemma}\label{lem:nulambdarelation}
For any \(f \in C^{\infty}(\BP)\),
\[\int_{\BP} F(x) T_{\beta} f(x)  \diff \lambda(x)=\int_{\BP} f(x) T_{\beta} F(x) \diff \lambda(x).\]  
\end{lemma}

\begin{proof} 
Since \(F_\epsilon\) and \(T_{\beta}f\) are in \(C^{\infty}(\BP)\), we can apply Lemma \ref{lem:Lambdabetasmooth} to get
\[ \int_{\BP} F_\epsilon(x) T_{\beta} f(x) \diff \lambda(x) =  \int_{\BP} T_{\beta} F_\epsilon(x)  f(x) \diff \lambda(x).\]
We further have
\[c_n(T_{\beta} F_{\epsilon})=|n|^{\beta}c_n(F_{\epsilon})=|n|^{\beta} c_n(F) c_n(\varphi_{\epsilon}) =c_n(T_{\beta}F) c_n(\varphi_{\epsilon})=  c_n(\varphi_{\epsilon}*(T_{\beta}F) ),\]
so \(T_{\beta} F_{\epsilon}=\varphi_{\epsilon}*T_{\beta}F=C_{\beta}\varphi_{\epsilon}*(\nu-\lambda)\) in distribution. Hence,
\begin{align}
     C_{\beta}^{-1} \int_{\BP} F_\epsilon(x) T_{\beta} f(x) \diff \lambda(x) 
     &= \int_{\BP} (\varphi_{\epsilon}*T_{\beta} F)(x)    f(x) \diff \lambda(x) \nonumber  \\
     &= \int_{\BP} (\varphi_{\epsilon}*(\nu-\lambda))(x)    f(x) \diff \lambda(x) \nonumber  \\
     &=  \int_{\BP} \int_{\BP} \varphi_{\epsilon}(x-y) \diff \nu(y)  f(x) \diff \lambda(x)- \int_{\BP}  f(x) \diff \lambda(x) \nonumber  \\
     &= \int_{\BP} f_{\epsilon}(y) \diff \nu(y)-\int_{\BP}  f(x) \diff \lambda(x) \label{eqn:equationfrommoll},
\end{align}
where the last equality follows by Fubini's theorem.

By Lemma \ref{lem:evans} we have \(F_{\epsilon} \rightarrow F\) and \(f_{\epsilon} \rightarrow f\) uniformly on \(\BP\). Taking the limit as \(\epsilon \rightarrow 0\) implies
\[C_{\beta}^{-1} \int_{\BP} F(x) T_{\beta} f(x) \diff \lambda(x)= \int_{\BP} f(x) \diff \nu(x)-\int_{\BP} f(x) \diff \lambda(x).\]
The lemma now follows from Lemma \ref{lem:lambdabetaFnu}.
\end{proof}

Fix some \(0<R<\pi/8\). Let \( \psi: \R \rightarrow \R\) be a smooth bump function with \( \psi\equiv 1 \) on \([-R,R]\) and \( \psi \equiv 0\) on \(\R \setminus [-2R,2R] \). For \(0<r \leq 1\), \(x_0 \in [-\pi,\pi)\) and \(x \in [x_0-\pi,x_0+\pi)\), we define
\[ \psi_{x_0,r}(x):= \psi \left(\frac{R(x-x_0)}{r} \right).\]
We extend \( \psi_{x_0,r}(x)\) to a \(\pi\)-periodic function on \(\R\). That is, for each \(x \in \R\) we let \( \psi_{x_0,r}(x)= \psi_{x_0,r}(x+n\pi)\), where \(n \in \Z\) is the unique integer such that \(x+n\pi \in [x_0-\pi/2,x_0+\pi/2)\). Then, \( \psi_{x_0,r}\) also defines a function on \(\BP\), which by an abuse of notation we also denote by \( \psi_{x_0,r}\).

For \(x \not=0\), we define \(K^{\beta/2}(x)\) to be the positive \(\pi\)-periodic kernel 
\[K^{\beta/2}(x):= C_{\beta}' \sum_{n \in \Z}\frac{1}{|x- n\pi|^{1+\beta}}, \]
where 
\[C_{\beta}':=\frac{ \Gamma \left( \frac{1+\beta}{2}\right)}{|\Gamma(-\beta/2)|\pi^{1/2}}.\]
By Theorem 1.5 in \cite{RS16}, for all \(x,x_0 \in \BP\) and \(r \leq R\),
\begin{align*}
    T_{\beta} \psi_{x_0,r}(x) &=C_{\beta}' \int_{\BP} (\psi_{x_0,r}(x)- \psi_{x_0,r}(y)) K^{\beta/2}(x-y) \diff \lambda(y) \\
    &=C_{\beta }'  \int_{\BP} \sum_{n \in \Z}  \frac{ \psi_{x_0,r}(x)- \psi_{x_0,r}(y)}{|x-y- n\pi|^{1+\beta}} \diff \lambda(y).
\end{align*}
We therefore have
\begin{align}
    |T_{\beta} \psi_{x_0,r}(x)|&=  \frac{C_{\beta }'}{\pi}  \left|\int_{x_0-\frac{\pi}{2}}^{x_0+\frac{\pi}{2}} \sum_{n \in \Z}  \frac{ \psi_{x_0,r}(x)- \psi_{x_0,r}(y)}{|x-y- n\pi|^{1+\beta}} \diff y \right| \nonumber \\ &\leq  \frac{C_{\beta }'}{\pi}  \int_{x_0-\frac{\pi}{2}}^{x_0+\frac{\pi}{2}} \sum_{n \in \Z}  \frac{| \psi_{x_0,r}(x)- \psi_{x_0,r}(y)|}{|x-y- n\pi|^{1+\beta}} \diff y \nonumber \\
    &= \frac{C_{\beta }'}{\pi}  \sum_{n \in \Z} \int_{x_0-\frac{\pi}{2}}^{x_0+\frac{\pi}{2}}   \frac{\left| \psi \left(\frac{R(x-x_0)}{r} \right)- \psi \left(\frac{R(y-x_0)}{r} \right) \right|}{|x-y- n\pi|^{1+\beta}} \diff y, \label{eqn:lessthanbigsum}
\end{align}
where in the last equality we are using Tonelli's theorem.

In the \(\R^d\) case in \cite{Wal66} it is only necessary to estimate the contribution from a single singular kernel, which Wallin split into a near-- and far-- field contribution. On \(\BP  \simeq \R/\pi\Z\) there are instead ``wrap-around'' terms, the \(n \not=0\) terms in the sum, which can be handled in a similar way to the far-field contribution. The next two lemmas are to control the more difficult \(n=0\) term.

\begin{lemma}\label{lem:n=0termprelim}
    \[ \int_{-\infty}^\infty \int_{-\infty}^{\infty} \frac{| \psi \left(x \right)- \psi(y)|}{|x-y|^{1+\beta}} \diff y \diff x<\infty.\]
\end{lemma}

\begin{proof}  
We first bound
 \[ \int_{-2R}^{2R} \int_{-\infty}^{\infty} \frac{| \psi \left(x \right)- \psi(y)|}{|x-y|^{1+\beta}} \diff y \diff x.\]
For \(x \in [-2R,2R]\),
\begin{align*}
    \int_{-\infty}^{\infty}\frac{| \psi \left(x \right)- \psi(y)|}{|x-y|^{1+\beta}} \diff y 
    &=\int_{-\infty}^{-2R }\frac{ \psi(x)}{(x-y)^{1+\beta}} \diff y +\int_{-2R}^{2R }\frac{| \psi(x)- \psi(y)|}{|x-y|^{1+\beta}} \diff y  +\int_{2R}^{\infty }\frac{ \psi(x)}{(y-x)^{1+\beta}} \diff y.
\end{align*}
Note that \( \psi\) is Lipschitz, so there exists some \(C>0\) such that \(| \psi(x)- \psi(y)| \leq C|x-y|\). Hence, the middle term is bounded by
\[C\int_{-2R}^{2R }\frac{1}{|x-y|^{\beta}} \diff y <\infty. \]
In particular, for some \(C'>0\),
\begin{align*}
    \int_{-2R}^{2R} \int_{-\infty}^{\infty}\frac{| \psi \left(x \right)- \psi(y)|}{|x-y|^{1+\beta}} \diff y \diff x  \\
    \leq C' +\int_{-2R}^{2R}  \int_{-\infty}^{-2R} \frac{ \psi(x)}{(x-y)^{1+\beta}} \diff y \diff x+\int_{-2R}^{2R}  \int_{2R}^{\infty} \frac{ \psi(x)}{(y-x)^{1+\beta}} \diff y \diff x.
\end{align*}
We further have
\begin{align*}
    \int_{-2R}^{2R}  \int_{2R}^{\infty} \frac{ \psi(x)}{(y-x)^{1+\beta}} \diff y \diff x
    &\leq  \int_{-2R}^{2R}  \int_{2R}^{\infty}  \frac{1}{(y-x)^{1+\beta}} \diff y \diff x \\
    &=\beta^{-1} \int_{-2R}^{2R} (2R-x)^{-\beta} \diff x<\infty.
\end{align*}
Similarly, 
\[\int_{-2R}^{2R}  \int_{-\infty}^{-2R} \frac{ \psi(x)}{(x-y)^{1+\beta}} \diff y \diff x \leq \beta^{-1} \int_{-2R}^{2R} (2R-x)^{-\beta} \diff x<\infty.\]
Hence,
\[\int_{-2R}^{2R} \int_{-\infty}^{\infty}\frac{| \psi \left(x \right)- \psi(y)|}{|x-y|^{1+\beta}} \diff y \diff x<\infty. \]

We now bound 
\[ \int_{\R \setminus [-2R,2R]} \int_{-\infty}^{\infty} \frac{| \psi \left(x \right)- \psi(y)|}{|x-y|^{1+\beta}} \diff y \diff x,\]
and this will complete the lemma. For \(x \not\in [-2R,2R]\), \(| \psi(x)- \psi(y)|= \psi(y)\), so 
\begin{align*}
    \int_{-\infty}^{\infty} \frac{| \psi \left(x \right)- \psi(y)|}{|x-y|^{1+\beta}} \diff y &=\int_{-\infty}^{\infty} \frac{ \psi(y)}{|x-y|^{1+\beta}} \diff y \\
    &= \int_{-2R}^{2R} \frac{ \psi(y)}{|x-y|^{1+\beta}} \diff y \\
    &=\int_{-2R}^{2R} \frac{| \psi(y)- \psi(x)|}{|x-y|^{1+\beta}} \diff y.
\end{align*}
Thus,
\begin{align*}
    \int_{\R \setminus [-2R,2R]} \int_{-\infty}^{\infty} \frac{| \psi \left(x \right)- \psi(y)|}{|x-y|^{1+\beta}} \diff y \diff x &= \int_{\R \setminus [-2R,2R]} \int_{-2R}^{2R}\frac{| \psi(x)- \psi(y)|}{|x-y|^{1+\beta}} \diff y \diff x \\
    &\leq \int_{-\infty}^{\infty} \int_{-2R}^{2R} \frac{| \psi(x)- \psi(y)|}{|x-y|^{1+\beta}} \diff y \diff x \\
    &=   \int_{-2R}^{2R} \int_{-\infty}^{\infty} \frac{| \psi(x)- \psi(y)|}{|x-y|^{1+\beta}} \diff x \diff y,
\end{align*}
which we have already shown is finite.
\end{proof}

\begin{lemma}\label{lem:n=0term}
For \(0<\zeta<\beta\),
    \[ \int_{-\infty}^\infty \int_{-\infty}^{\infty} |x|^{\zeta} \frac{| \psi \left(x \right)- \psi(y)|}{|x-y|^{1+\beta}} \diff y \diff x<\infty.\]
\end{lemma}

\begin{proof}
    Let \(X>2R\). We have
    \begin{align}
    \begin{split}\label{eqn:lessthanXdoubleintegral}
        \int_{-X}^{X} \int_{-\infty}^{\infty} |x|^{\zeta} \frac{| \psi \left(x \right)- \psi(y)|}{|x-y|^{1+\beta}} \diff y \diff x &\leq X^{\zeta} \int_{-X}^{X} \int_{-\infty}^{\infty} \frac{| \psi \left(x \right)- \psi(y)|}{|x-y|^{1+\beta}} \diff y \diff x \\
        &\leq X^{\zeta} \int_{-\infty}^{\infty} \int_{-\infty}^{\infty} \frac{| \psi \left(x \right)- \psi(y)|}{|x-y|^{1+\beta}} \diff y \diff x <\infty.
    \end{split}
    \end{align}
    We now show that \begin{equation}\label{eqn:greaterthanXdoubleintegral}
        \int_{X}^\infty \int_{-\infty}^{\infty} |x|^{\zeta} \frac{| \psi \left(x \right)- \psi(y)|}{|x-y|^{1+\beta}} \diff y \diff x<\infty.
    \end{equation}
    By symmetry, this will also imply that
    \[\int_{-\infty}^{-X} \int_{-\infty}^{\infty} |x|^{\zeta} \frac{| \psi \left(x \right)- \psi(y)|}{|x-y|^{1+\beta}} \diff y \diff x <\infty,\]
    which combined with (\ref{eqn:lessthanXdoubleintegral}) and (\ref{eqn:greaterthanXdoubleintegral}) will prove the lemma.
    
    Let us show (\ref{eqn:greaterthanXdoubleintegral}). For \(x>X\), we have 
     \(| \psi(x)- \psi(y)|= \psi(y)\), so 
     \begin{align*}
         \int_{-\infty}^{\infty} |x|^{\zeta} \frac{| \psi \left(x \right)- \psi(y)|}{|x-y|^{1+\beta}} \diff y &= \int_{-\infty}^{\infty} x^{\zeta} \frac{ \psi(y)}{|x-y|^{1+\beta}} \diff y \\
         &=  \int_{-2R}^{2R} \frac{ x^{\zeta}  \psi(y)}{(x-y)^{1+\beta}} \diff y \\
         &\leq \int_{-2R}^{2R} \frac{ x^{\zeta}}{(x-y)^{1+\beta}} \diff y.
     \end{align*}
     Hence,
     \begin{align}
         \int_{X}^\infty \int_{-\infty}^{\infty} |x|^{\zeta} \frac{| \psi \left(x \right)- \psi(y)|}{|x-y|^{1+\beta}} \diff y \diff x  \nonumber \\ \leq  \int_{X}^\infty  \int_{-2R}^{2R} \frac{ x^{\zeta}}{(x-y)^{1+\beta}} \diff y \diff x \nonumber \\ 
         =  \int_{X}^\infty  \int_{-2R}^{2R}  \frac{1}{(x-y)^{1+\beta-\zeta}} \diff y \diff x+  \int_{X}^\infty \int_{-2R}^{2R} \frac{x^{\zeta}-(x-y)^{\zeta}}{(x-y)^{1+\beta}} \diff y \diff x \nonumber \label{eqn:doubleintegral}.
     \end{align}
     
     By assumption \(\zeta<\beta\), so
     \begin{align*}
            \int_{X}^\infty  \int_{-2R}^{2R}  \frac{1}{(x-y)^{1+\beta-\zeta}} \diff y \diff x&= \int_{-2R}^{2R} \int_{X}^\infty  \frac{1}{(x-y)^{1+\beta-\zeta}} \diff x \diff y \\
            &=(\beta-\zeta)^{-1}  \int_{-2R}^{2R} (X-y)^{-(\beta-\zeta)} \diff y \\
           &\leq 4R(\beta-\zeta)^{-1}(X-2R)^{-(\beta-\zeta)}  <\infty.
     \end{align*}

When \(y\leq 0\), 
\(x^{\zeta}-(x-y)^{\zeta} \leq 0\), so \[ \int_{X}^\infty \int_{-2R}^{0} \frac{x^{\zeta}-(x-y)^{\zeta}}{(x-y)^{1+\beta}} \diff y \diff x<0.\]
When \(y >0\), by the mean value theorem we have 
\[\frac{x^{\zeta}-(x-y)^{\zeta}}{y} \leq \zeta (x-y)^{\zeta-1}.\]
Therefore, by Tonelli's theorem,
\begin{align*}
    \int_{X}^\infty \int_{0}^{2R} \frac{x^{\zeta}-(x-y)^{\zeta}}{(x-y)^{1+\beta}} \diff y \diff x &\leq 2R\zeta \int_{X}^\infty \int_{0}^{2R} \frac{1}{(x-y)^{2+\beta-\zeta}} \diff y \diff x \\
    &=2R\zeta \int_{0}^{2R} \int_{X}^\infty  \frac{1}{(x-y)^{2+\beta-\zeta}} \diff x \diff y \\
    &=2R \zeta(1+\beta-\zeta)^{-1} \int_{0}^{2R} (X-y)^{-(1+\beta-\zeta)} \diff y \\
    &\leq 4R^2 \zeta (1+\beta-\zeta)^{-1}(X-2R)^{-(1+\beta-\zeta)}<\infty.
\end{align*}
Hence, combining the above estimates, we have shown 
\[\int_{X}^\infty \int_{-\infty}^{\infty} |x|^{\zeta} \frac{| \psi \left(x \right)- \psi(y)|}{|x-y|^{1+\beta}} \diff y \diff x<\infty.\]
As justified above, this concludes the proof of the lemma. 
\end{proof}

\begin{lemma}\label{lem:integrallessthanCRr1minusbetapluszeta}
    There exists \(C_R>0\) such that for all \(x_0 \in \BP\) and \(r \leq R\),
    \[\int_{\BP} d_{\BP}(x,x_0)^{\zeta} |T_{\beta} \psi_{x_0,r}(x)| \diff \lambda(x) \leq C_R r^{1-\beta+\zeta}=C_R r^{-t+\zeta}.\]
\end{lemma}

\begin{proof}
We have
\begin{align*}
     \int_{\BP} d_{\BP}(x,x_0)^{\zeta} |T_{\beta} \psi_{x_0,r}(x)| \diff \lambda(x) &= \frac{1}{\pi} \int_{x_0-\frac{\pi}{2}}^{x_0+\frac{\pi}{2}} d_{\BP}(x,x_0)^{\zeta} |T_{\beta} \psi_{x_0,r}(x)| \diff x   \\
     &\leq \frac{1}{\pi} \int_{x_0-\frac{\pi}{2}}^{x_0+\frac{\pi}{2}} |x-x_0|^{\zeta} |T_{\beta} \psi_{x_0,r}(x)| \diff x.
\end{align*}
Hence, by (\ref{eqn:lessthanbigsum}),
     \begin{align*}
     \int_{\BP} d_{\BP}(x,x_0)^{\zeta} |T_{\beta} \psi_{x_0,r}(x)| \diff \lambda(x) 
     &\leq \frac{C_{\beta}'}{\pi^2} \int_{x_0-\frac{\pi}{2}}^{x_0+\frac{\pi}{2}} \int_{x_0-\frac{\pi}{2}}^{x_0+\frac{\pi}{2}}  \sum_{n \in \Z} \left|x-x_0 \right|^{\zeta} \frac{\left| \psi \left(\frac{R(x-x_0)}{r} \right) - \psi \left(\frac{R(y-x_0)}{r} \right) \right|}{|x-y- n\pi|^{1+\beta}} \diff y \diff x \\
     &=\frac{C_{\beta}'}{\pi^2} \sum_{n \in \Z}  \int_{x_0-\frac{\pi}{2}}^{x_0+\frac{\pi}{2}} \int_{x_0-\frac{\pi}{2}}^{x_0+\frac{\pi}{2}} \left|x-x_0 \right|^{\zeta} \frac{\left| \psi \left(\frac{R(x-x_0)}{r} \right) - \psi \left(\frac{R(y-x_0)}{r} \right) \right|}{|x-y- n\pi|^{1+\beta}} \diff y \diff x
\end{align*}
Making the change of variables \(u=\frac{R(x-x_0)}{r}\) and  \(v=\frac{R(y-x_0)}{r}\), the \(n=0\) term in the sum is equal to 
\[ R^{-(1-\beta+\zeta) } r^{1-\beta+\zeta}  \int_{-\frac{R\pi}{2r}}^{\frac{R\pi}{2r}} \int_{-\frac{R\pi}{2r}}^{\frac{R\pi}{2r}}  |u|^{\zeta} \frac{\left| \psi \left(u \right) - \psi \left(v \right) \right|}{|u-v|^{1+\beta}} \diff v \diff u,\]
where
\[ \int_{-\frac{R\pi}{2r}}^{\frac{R\pi}{2r}} \int_{-\frac{R\pi}{2r}}^{\frac{R\pi}{2r}}  |u|^{\zeta} \frac{\left| \psi \left(u \right) - \psi \left(v \right) \right|}{|u-v|^{1+\beta}} \diff v \diff u \leq \int_{-\infty}^\infty \int_{-\infty}^{\infty} |u|^{\zeta} \frac{| \psi \left(u \right)- \psi(v)|}{|u-v|^{1+\beta}} \diff v \diff u<\infty\]
by Lemma \ref{lem:n=0term}.

Observe that the \(n \not= 0\) terms in the sum are equal to
\begin{align}
    \int_{x_0-2r}^{x_0+2r} \int_{x_0-\frac{\pi}{2}}^{x_0+\frac{\pi}{2}} \left|x-x_0 \right|^{\zeta} \frac{\left| \psi \left(\frac{R(x-x_0)}{r} \right) - \psi \left(\frac{R(y-x_0)}{r} \right) \right|}{|x-y- n\pi|^{1+\beta}} \diff y \diff x \label{eqn:term1in3sum}\\+ \int_{x_0-\pi/2}^{x_0-2r} \int_{x_0-2r}^{x_0+2r} \left|x-x_0 \right|^{\zeta} \frac{\psi \left(\frac{R(y-x_0)}{r} \right)}{|x-y- n\pi|^{1+\beta}} \diff y \diff x \label{eqn:term2in3sum} \\
    +\int_{x_0+2r}^{x_0+\pi/2} \int_{x_0-2r}^{x_0+2r} \left|x-x_0 \right|^{\zeta} \frac{\psi \left(\frac{R(y-x_0)}{r} \right)}{|x-y- n\pi|^{1+\beta}} \diff y \diff x \label{eqn:term3in3sum}.
\end{align}
For all \((x, y) \in [-\pi/2,\pi/2] \times [-2r,2r] \cup [-2r,2r] \times [-\pi/2,\pi/2] \) and all \(n \in \Z \setminus \{0\}\), we have
\[|x-y-n\pi| \geq |n| \pi -|x-y| \geq |n|\pi-(2r+\pi/2) \geq \left(|n|-\frac{3}{4} \right)\pi \geq \frac{|n|\pi}{4},\]
where we use in the penultimate inequality that \(r \leq R <\pi/8\). Hence, for some \(C_1>0\), (\ref{eqn:term1in3sum}) is bounded above by \(C_1 |n|^{-(1+\beta)} r^{1+\zeta}\) and, for some \(C_2>0\), (\ref{eqn:term2in3sum}) and (\ref{eqn:term3in3sum}) are bounded above by \(C_2 |n|^{-(1+\beta)} r\). Thus, for some \(C_3>0\),
\begin{align*}
    \sum_{|n| \geq 1}  \int_{x_0-\frac{\pi}{2}}^{x_0+\frac{\pi}{2}} \int_{x_0-\frac{\pi}{2}}^{x_0+\frac{\pi}{2}} \left|x-x_0 \right|^{\zeta} \frac{\left| \psi \left(\frac{R(x-x_0)}{r} \right) - \psi \left(\frac{R(y-x_0)}{r} \right) \right|}{|x-y- n\pi|^{1+\beta}} \diff y \diff x &\leq \sum_{|n|\geq 1} |n|^{-(1+\beta)}(C_1 r^{1+\zeta} +C_2 r) \\
    &\leq C_3 r,
\end{align*}
where we are using that \(1+\beta>1\). Combining this with our estimate for the \(n=0\) term, we have shown that there exists some \(C_R=C_R(t,\zeta)>0\) such that for all \(r \in (0,R]\),
\[\int_{\BP} d_{\BP}(x,x_0)^{\zeta} |T_{\beta} \psi_{x_0,r}(x)| \diff \lambda(x) \leq C_R r^{1-\beta+\zeta}=C_R r^{-t+\zeta}.\]
\end{proof}

\begin{proof}[Proof of Proposition \ref{prop:Holdergivesdimensionbound}]
Fix some \(0<R<\pi/8 \). Let \(C_R>0\) be as given by Lemma \ref{lem:integrallessthanCRr1minusbetapluszeta} and let \(r \leq R/2\). For \(x_0 \in \BP\) note that \(1_{B(x_0,r)}(x) \equiv 1_{B_{\BP}(x_0,r)}(x) \leq \psi_{x_0,2r}(x)\) for all \(x \in \BP\) since \(\sin \theta \geq \theta/2\) for \(\theta \in [0,\pi/2]\). Hence, by Lemma \ref{lem:nulambdarelation},
    \begin{align}
        \nu(B(x_0,r)) &\leq \int  \psi_{x_0,2r}(x) \diff \nu(x) \nonumber \\
        &=\int  \psi_{x_0,2r}(x) \diff \lambda(x) +C_{\beta}^{-1} \int  F(x) T_{\beta}  \psi_{x_0,2r}(x) \diff \lambda(x) \nonumber \\
        &\leq 8\pi^{-1}r+C_{\beta}^{-1} \int  F(x) T_{\beta}  \psi_{x_0,2r}(x) \diff \lambda(x). \label{eqn:nuBx0rinequality}
    \end{align}
    By the zero mean of \(T_{\beta}\) (Lemma \ref{lem:zeromeanofLambda}), the second term is equal to 
    \begin{align*}
        \int  F(x) T_{\beta}  \psi_{x_0,2r}(x) \diff \lambda(x) &= \int  (F(x)-F(x_0)) T_{\beta}  \psi_{x_0,2r}(x) \diff \lambda(x).
    \end{align*}
    Hence,
    \begin{align*}
        \left|\int  F(x) T_{\beta}  \psi_{x_0,2r}(x) \diff \lambda(x) \right| &\leq \int  |F(x)-F(x_0)| |T_{\beta}  \psi_{x_0,2r}(x)| \diff \lambda(x) \\
        & \leq |F|_{\zeta} \int d_{ \BP}(x,x_0)^{\zeta} |T_{\beta}  \psi_{x_0,2r}(x)| \diff \lambda(x) \\
        &\leq  C_R |F|_{\zeta} (2r)^{-t+\zeta}.
    \end{align*}
    Combining this with (\ref{eqn:nuBx0rinequality}) implies that there exists some \(C=C(R,t,\zeta,|F|_{\zeta})>0\) such that for all \(x_0 \in \BP\) and all \(r \in (0,R/2]\),
    \[\nu(B(x_0,r)) \leq C r^{-t+\zeta}.\]
    Making the constant \(C\) larger if necessary (i.e. larger than \((R/2)^{-(-t+\zeta)}\)), we can have the above inequality hold for all \(r>0\). Hence, we have shown that \(\dim_F \nu \geq -t+\zeta\). 
    
    By Proposition \ref{prop:hisholder} we further have that \(\MCI_{t,\nu}\) is H\"older continuous. Since  \(\MCI_{t,\nu}(x)=h(x) \) for \(\lambda\)-a.e. \(x\) and both are continuous, we must have \(\MCI_{t,\nu} \equiv h\).
\end{proof}

\begin{prop}\label{prop:Inu1=Inu2impliesnu1=nu2}
    Let \(-1<t<0\). Suppose that for \(\lambda\)-almost every \(x \in \BP\),
    \begin{equation}\label{eqn:Inu1=Inu2}
        \int d_{\BP}(x,w)^t \diff \nu_1(w)=  \int d_{\BP}(x,w)^t \diff \nu_2(w) <\infty.
    \end{equation}
     Then, \(\nu_1=\nu_2\).
\end{prop}

\begin{lemma}\label{lem:intC2function}
    Let \(f \in C^2((0,\pi/2))\) be such that \(f(x),f''(x)>0\) and \(f'(x)<0\) on \((0,\pi/2)\) and \(\int_0^{\pi/2} f(x) \diff x <\infty\). Then, for all \(n \in \Z \setminus \{0\}\)
    \[\int_0^{\pi/2} f(x) \cos(2nx) \diff x>0.\]
\end{lemma}

\begin{proof}
Fix some  \(n \in \Z \setminus \{0\}\) and let \(\epsilon \in (0,\pi/8)\). Integrating by parts twice gives    
    \begin{align*}
        \int_{\epsilon}^{\pi/2-\epsilon} f(x) \cos(2nx) \diff x \\ = \left[\frac{1}{2n} f(x) \sin 2 n x \right]_\epsilon^{\pi/2-\epsilon} +\frac{1}{(2n)^2}[ f'(x)\cos 2 n x  ]_{\epsilon}^{\pi/2-\epsilon}-\frac{1}{(2n)^2} \int_\epsilon^{\pi/2-\epsilon}f''(x) \cos(2nx) \diff x. 
    \end{align*}
    Since \(f''\) is strictly positive and continuous, there exists \(\delta>0\), independent of \(\epsilon \in (0,\pi/8)\), such that 
    \[\frac{1}{(2n)^2} \int_{\epsilon}^{\pi/2-\epsilon}f''(x) \cos(2nx) \diff x<\frac{1}{(2n)^2} \int_\epsilon^{\pi/2-\epsilon}f''(x)\diff x -\delta.\]
    Hence,
    \begin{align*}
        \int_{\epsilon}^{\pi/2-\epsilon} f(x) \cos(2nx) \diff x
        >\left[\frac{1}{2n}f(x) \sin 2 n x \right]_\epsilon^{\pi/2-\epsilon}+ \frac{1}{(2n)^2}[ f'(x)(\cos 2 n x -1) ]_{\epsilon}^{\pi/2-\epsilon}+\delta.
    \end{align*}
    It is straightforward to check that the assumptions in the lemma imply that \(\lim_{x \rightarrow 0} xf(x)=\lim_{x \rightarrow 0} x^2 f'(x)=0\). 
    Therefore, letting \(\epsilon \rightarrow 0 \),
    \[\int_{0}^{\pi/2} f(x) \cos(2nx) \diff x \geq \delta.\]
    
\end{proof}
\begin{proof}[Proof of Proposition \ref{prop:Inu1=Inu2impliesnu1=nu2}]
Let \(\nu:=\nu_1-\nu_2\). For \(\lambda\)-a.e. \(x \in \BP\), 
\[ d_{\BP}(0,\cdot)^t* \nu(x) = \int d_{\BP}(0,x-y)^t \diff \nu(y)= \int d_{\BP}(x,y)^t \diff \nu(y)=0.\]
Hence, \(c_n(d_{\BP}(0,\cdot)^t) c_n(\nu)=0\), so to prove the proposition it suffices to show that \(c_n(d_{\BP}(0,\cdot)^t) \not = 0\) for \(n \in \Z\). We have
\begin{align*}
    c_n(d_{\BP}(0,\cdot)^t) &= \int_{\BP} d_{\BP}(0,x)^t e^{-2inx} \diff \lambda(x) \\
    &= \int_{\BP} d_{\BP}(0,x)^t \cos (2nx) \diff \lambda(x) \\
    &=  \frac{2}{\pi} \int_{0}^{\pi/2} \sin(x)^t \cos (2nx) \diff x.
\end{align*}
The proposition now follows from Lemma \ref{lem:intC2function}.
\end{proof}

\end{document}